\documentclass[a4paper,10pt]{article}
\usepackage[latin1]{inputenc}
\usepackage[T1]{fontenc}
\usepackage{amsmath}
\usepackage{amsfonts}
\usepackage{amstext}
\usepackage{amsthm}
\usepackage[all]{xy}
\usepackage{geometry}
\usepackage{srcltx}
\usepackage{graphics}
\usepackage{epsfig}
\usepackage{subfigure}
\usepackage{slashed}
\usepackage{color}
\usepackage{amssymb}
\usepackage{srcltx}
\newtheorem{theorem}{Theorem}[section]
\newtheorem{lemma}[theorem]{Lemma}

\newtheorem{rem}[theorem]{Remark}
\newtheorem{prop}[theorem]{Proposition}
\newtheorem{cor}[theorem]{Corollary}
\newtheorem*{theoremA}{Theorem A}
\newtheorem*{theoremB}{Theorem B}
\newtheorem*{theoremC}{Theorem C}
\newtheorem*{corA}{Corollary}

\DeclareMathOperator{\supp}{supp}
\DeclareMathOperator{\im}{im}
\DeclareMathOperator{\ind}{ind}
\DeclareMathOperator{\sind}{s-ind}

\DeclareMathOperator{\sfl}{sf}

\DeclareMathOperator{\codim}{codim}
\DeclareMathOperator{\gra}{graph}

\title{On the Fredholm Lagrangian Grassmannian, Spectral Flow and ODEs in Hilbert Spaces}
 
\author{Nils Waterstraat\thanks{This work was supported by the Royal Society Research Grant RG170456.}}

\begin{document}
\date{}
\maketitle

\footnotetext[1]{{\bf 2010 Mathematics Subject Classification: Primary 58E05; Secondary 58J30, 34G10, 47A53}}

\begin{abstract}
\noindent
We consider homoclinic solutions for Hamiltonian systems in symplectic Hilbert spaces and generalise spectral flow formulas that were proved by Pejsachowicz and the author in finite dimensions some years ago. Roughly speaking, our main theorem relates the spectra of infinite dimensional Hamiltonian systems under homoclinic boundary conditions to intersections of their stable and unstable spaces. Our proof has some interest in its own. Firstly, we extend a celebrated theorem by Cappell, Lee and Miller about the classical Maslov index in $\mathbb{R}^{2n}$ to symplectic Hilbert spaces. Secondly, we generalise the classical index bundle for families of Fredholm operators of Atiyah and J\"anich to unbounded operators for applying it to Hamiltonian systems under varying boundary conditions. Finally, we substantially make use of striking results by Abbondandolo and Majer to study Fredholm properties of infinite dimensional Hamiltonian systems.  
\end{abstract}

\section{Introduction}
Let $E$ be a real separable Hilbert space, let $I:=[0,1]$ denote the unit interval and let\linebreak $S:I\times\mathbb{R}\rightarrow\mathcal{S}(E)$ be a family of bounded selfadjoint operators on $E$ which is continuous with respect to the norm topology. We assume that $J:E\rightarrow E$ is a bounded operator such that $J^2=-I_E$, $J^T=-J$, and consider differential equations of the form

\begin{equation}\label{Hamiltonian}
\left\{
\begin{aligned}
Ju'(t)+S_\lambda(t)u(t)&=0,\quad t\in\mathbb{R}\\
\lim_{t\rightarrow\pm\infty}u(t)&=0.
\end{aligned}
\right.
\end{equation}
Let us denote for $\lambda\in I$ and $t_0\in\mathbb{R}$ by 

\begin{align}\label{stableunstable}
\begin{split}
E^u_\lambda(t_0)&=\{u(t_0)\in E:\, Ju'(t)+S_\lambda(t)u(t)=0,\,\, \lim_{t\rightarrow-\infty}u(t)=0\}\subset E,\\
E^s_\lambda(t_0)&=\{u(t_0)\in E:\, Ju'(t)+S_\lambda(t)u(t)=0,\,\, \lim_{t\rightarrow+\infty}u(t)=0\}\subset E
\end{split}
\end{align}
the unstable and stable subspaces of \eqref{Hamiltonian}. Note that \eqref{Hamiltonian} has a non-trivial solution if and only if $E^u_\lambda(t_0)\cap E^s_\lambda(t_0)\neq\{0\}$ for some (and hence any) $t_0\in\mathbb{R}$.\\  
If we denote by $L^2(\mathbb{R},E)$ and $H^1(\mathbb{R},E)$ the usual spaces of maps having values in $E$, then we obtain differential operators

\begin{align}\label{A}
\mathcal{A}_\lambda:H^1(\mathbb{R},E)\subset L^2(\mathbb{R},E)\rightarrow L^2(\mathbb{R},E),\quad (\mathcal{A}_\lambda u)(t)=Ju'(t)+S_\lambda(t)u(t)
\end{align}
which have a non-trivial kernel if and only if \eqref{Hamiltonian} has a non-trivial solution (see \cite{AlbertoODE}).\\
In \cite{AlbertoODE} Abbondandolo and Majer studied Fredholm properties of the operators $\mathcal{A}_\lambda$ in relation to the stable and unstable subspaces \eqref{stableunstable}. Their motivation came from a Morse homology in Hilbert spaces that they constructed in \cite{AlbertoMorseHilbert} (see also \cite{AlbertoNotes}) and where differential equations of the form \eqref{Hamiltonian} naturally appear. In these works, the linear theory was developed that is necessary for the set-up of Morse homology on Hilbert manifolds (see \cite{AlbertoInfinite}, \cite{AlbertoManifold}). Applications of their theory can be found, e.g., in the study of periodic orbits of Hamiltonian systems, periodic solutions of wave equations and solutions of classes of elliptic systems as in \cite{AlbertoTMNA}, \cite{AlbertoNonlinear}, \cite{vanderVorst}, \cite{MarekI}, \cite{MarekII}, \cite{Kryszewski}, \cite{Maalaoui} and \cite{Szulkin}. One of the long term aims of this paper is to open up recent methods from variational bifurcation theory to such classes of nonlinear equations by following the author's work \cite{WaterstraatHomoclinics}, where the case $E=\mathbb{R}^{2n}$ was considered. Moreover, we intend to make our findings applicable to such Hamiltonian PDEs by using new comparison methods for the spectral flow from our paper \cite{Edinburgh}.\\
In what follows, we assume that the family $S:I\times\mathbb{R}\rightarrow\mathcal{S}(E)$ is of the form

\begin{align}\label{form}
S_\lambda(t)=B_\lambda+K_\lambda(t),
\end{align}
where

\begin{enumerate}
\item[(A1)] $K_\lambda(t)$ is compact for all $(\lambda,t)\in I\times\mathbb{R}$, and the limits

\[K_\lambda(\pm\infty)=\lim_{t\rightarrow\pm\infty}K_\lambda(t)\]
exist uniformly in $\lambda$,
\item[(A2)] the operators $JB_\lambda$ and

\[JS_\lambda(\pm\infty):=J(B_\lambda+K_\lambda(\pm\infty))\]
are hyperbolic, i.e., there are no purely imaginary points in their spectra.
\end{enumerate} 
We will show below that it follows from \cite{AlbertoMorseHilbert} that the operators $\mathcal{A}_\lambda$ are selfadjoint Fredholm operators under the assumptions (A1) and (A2). Consequently, the spectral flow $\sfl(\mathcal{A})$ is defined, which is a homotopy invariant for paths of selfadjoint Fredholm operators (see Section \ref{section-sfl} below).\\
The operator $J:E\rightarrow E$ induces a symplectic form on $E$ by $\omega(u,v)=\langle Ju,v\rangle$, which makes it a symplectic Hilbert space (see \cite{BoossZhu} and \cite{NicolaescuII}). The Maslov index for paths of Lagrangian subspaces was first generalised to infinite dimensional symplectic Hilbert spaces by Swanson in \cite{Swanson}. Here we follow the survey \cite{Furutani} of Boo{\ss}-Bavnbek and Furutani's approach \cite{BoossFurutani}. Henceforth, let $\mathcal{FL}^2(E,\omega)$ denote the Fredholm Lagrangian Grassmannian of pairs of spaces (see Section \ref{subsection-Maslov}). We will explain below that it follows from \cite{AlbertoODE} that $\{(E^u_\lambda(t_0),E^s_\lambda(t_0))\}_{\lambda\in I}$ is a path in $\mathcal{FL}^2(E,\omega)$, and so its Maslov index $\mu_{Mas}(E^u_\cdot(t_0),E^s_\cdot(t_0))$ is defined for every fixed $t_0\in\mathbb{R}$. Actually, it is readily seen that this integer does not depend on the choice of $t_0$. The main theorem of this paper reads as follows.

\begin{theoremA}
If the assumptions (A1) and (A2) hold, then

\[\sfl(\mathcal{A})=\mu_{Mas}(E^u_\cdot(0),E^s_\cdot(0)).\]
\end{theoremA}
\noindent
Let us point out that, in the special case $E=\mathbb{R}^{2n}$, this is a generalisation of the author's previous work \cite{WaterstraatHomoclinics} as well as of the recent paper \cite{Hu} by Hu and Portaluri.\\
Let us now consider for the operators $S_\lambda(\pm\infty)$ from (A2) the families of subspaces 

\begin{align*}
E^s_\lambda(\pm\infty)&=\{x\in E: \exp(tJS_\lambda(\pm\infty))x\rightarrow 0\,\,\text{as}\,\, t\rightarrow\infty\},\\
E^u_\lambda(\pm\infty)&=\{x\in E: \exp(tJS_\lambda(\pm\infty))x\rightarrow 0\,\,\text{as}\,\, t\rightarrow-\infty\}.
\end{align*}
Note that these are the stable and unstable subspaces \eqref{stableunstable} for the autonomous equations

\begin{equation}\label{HamiltonianJac}
\left\{
\begin{aligned}
Ju'(t)+S_\lambda(\pm\infty)u(t)&=0,\quad t\in\mathbb{R}\\
\lim_{t\rightarrow\pm\infty}u(t)&=0.
\end{aligned}
\right.
\end{equation}
Moreover, it can be seen from \cite{AlbertoODE} that $\{(E^u_\lambda(+\infty), E^s_\lambda(-\infty))\}_{\lambda\in I}$ is a path in $\mathcal{FL}^2(E,\omega)$. We obtain below the following corollary of Theorem A, which was proved for $E=\mathbb{R}^{2n}$ by Pejsachowicz in \cite{Jacobo}.

\begin{corA}
If (A1), (A2) hold and $S_0(t)=S_1(t)$ for all $t\in\mathbb{R}$ in \eqref{Hamiltonian}, then 

\[\sfl(\mathcal{A})=\mu_{Mas}(E^u_\cdot(+\infty),E^s_\cdot(-\infty)).\] 
\end{corA}
\noindent
Consequently, if the path of operators $\mathcal{A}$ in \eqref{A} is periodic, then its spectral flow can be computed from the stable and unstable spaces of the autonomous equations \eqref{HamiltonianJac}. Obviously, the latter spaces are much easier to obtain than those in  \eqref{stableunstable} for the original problem \eqref{Hamiltonian}. Finally, it is worth mentioning that, by proving the above corollary, we further show that, for $E=\mathbb{R}^{2n}$, Pejsachowicz' theorem \cite{Jacobo} can be obtained from our work \cite{WaterstraatHomoclinics}. To the best of our knowledge, this has not been noted before.\\     
The argument for proving Theorem A partially follows an approach to the author's work \cite{WaterstraatHomoclinics} that was recently proposed by Hu and Portaluri in \cite{Hu}. They pointed out that Theorem A can be obtained for $E=\mathbb{R}^{2n}$ from Cappell, Lee and Miller's seminal investigations about the Maslov index \cite{Cappell} in finite dimensions. As we want to use \cite{Cappell} as well, we in particular need to generalise one of the main theorems of that paper to infinite dimensions. This result is of independent interest, and we now want to introduce it briefly. For every path $\{(\Lambda_0(\lambda),\Lambda_1(\lambda))\}_{\lambda\in I}$ in $\mathcal{FL}^2(E,\omega)$, we obtain differential operators

\begin{align}\label{Qopp}
\mathcal{Q}_\lambda:\mathcal{D}(\mathcal{Q}_\lambda)\subset L^2([a,b],E)\rightarrow L^2([a,b],E),\quad (\mathcal{Q}_\lambda u)(t)=Ju'(t),
\end{align}
on the domains

\[\mathcal{D}(\mathcal{Q}_\lambda)=\{u\in H^1([a,b],E):\, u(a)\in\Lambda_0(\lambda),\, u(b)\in\Lambda_1(\lambda)\}.\]
We show below that each $\mathcal{Q}_\lambda$ is selfadjoint and Fredholm. Moreover, we investigate whether the path $\mathcal{Q}=\{\mathcal{Q}_\lambda\}_{\lambda\in I}$ is continuous with respect to the so called gap-metric (see Section \ref{section-gap}), in which case its spectral flow $\sfl(\mathcal{Q})$ is defined. In what follows, we call a path $\{(\Lambda_0(\lambda),\Lambda_1(\lambda))\}_{\lambda\in I}$ in $\mathcal{FL}^2(E,\omega)$ admissible if $\Lambda_0(0)\cap\Lambda_1(0)=\Lambda_0(1)\cap\Lambda_1(1)=\{0\}$.

\begin{theoremB}
If $\{(\Lambda_0(\lambda),\Lambda_1(\lambda))\}_{\lambda\in I}$ is an admissible path in $\mathcal{FL}^2(E,\omega)$, then the corresponding path of differential operators $\mathcal{Q}$ is a continuous path of selfadjoint Fredholm operators, and

\[\sfl(\mathcal{Q})=\mu_{Mas}(\Lambda_0(\cdot),\Lambda_1(\cdot)).\]
\end{theoremB}
\noindent
For obtaining Theorem A from Theorem B, we need to deal with the spectral flow for paths of operators having varying domains. This is usually a delicate problem, as apart from non-obvious continuity issues like in Theorem B, we also essentially loose the opportunity to apply crossing forms, which is probably the most powerful method for computing spectral flows (see, e.g., \cite{Robin}, \cite{FPR}, \cite{FPRII}, \cite{WaterstraatHomoclinics}). However, when Atiyah, Patodi and Singer introduced the spectral flow for closed paths of selfadjoint Fredholm operators in \cite{APS}, they showed that it can be computed as first Chern number of a family index. The latter index is an element of the odd $K$-theory group $K^{-1}(S^1)\cong \mathbb{Z}$, and a further aim of this paper is to show that an adapted construction can be used for non-closed paths that are continuous in the gap-topology, where we mainly review material from our PhD thesis \cite{thesis} that has not been published yet. Let us assume that $H$ is a complex Hilbert space and let us denote by $\Omega(\mathcal{CF}^\textup{sa}(H),G\mathcal{C}^\textup{sa}(H))$ the set of all paths of selfadjoint Fredholm operators on $H$ that are continuous with respect to the gap-topology and have invertible endpoints. In what follows, we denote by $K^{-1}(X,Y)$ the odd $K$-theory group of a compact pair of spaces $(X,Y)$, which we recall in Appendix \ref{app-K}, and by $\partial I$ the boundary of the unit interval $I$. Moreover, we use that the Chern number \eqref{Chern} is an isomorphism $c_1:K^{-1}(I,\partial I)\rightarrow\mathbb{Z}$.

\begin{theoremC}
There exists a map

\[\sind:\Omega(\mathcal{CF}^\textup{sa}(H),G\mathcal{C}^\textup{sa}(H))\rightarrow K^{-1}(I,\partial I)\]
such that 

\[c_1(\sind(\mathcal{A}))=\sfl(\mathcal{A})\in\mathbb{Z}\]
for every $\mathcal{A}\in\Omega(\mathcal{CF}^\textup{sa}(H),G\mathcal{C}^\textup{sa}(H))$. 
\end{theoremC}
\noindent  
Previous versions of Theorem C for paths of selfdjoint Fredholm operators having a fixed domain have been applied in \cite{IchK} and \cite{CalcVar}. However, to the best of our knowledge, Theorem C is the first approach to compute the spectral flow for general gap-continuous paths by using $K$-theory. This shall be of independent interest as it provides a new method for computing spectral flows for paths of operators having non-constant domains. The capability of Theorem C will be demonstrated in the third step of the proof of Theorem A in Section \ref{Section-ProofA}.\\
Actually, we show below that for every gap-continuous family $\mathcal{A}:(X,Y)\rightarrow(\mathcal{CF}^\textup{sa}(H),G\mathcal{C}^\textup{sa}(H))$ of selfadjoint Fredholm operators there is an element $\sind(\mathcal{A})\in K^{-1}(X,Y)$ which coincides with Atiyah, Patodi and Singer's family index for families of bounded selfadjoint Fredholm operators when $Y=\emptyset$. In the construction of $\sind(\mathcal{A})$, we also generalise the classical index bundle for families of bounded and not necessarily selfadjoint Fredholm operators, which was introduced independently by Atiyah and J\"anich in the sixties, to gap-continuous families.\\     
Let us outline the structure of the paper. The content is rather technical and requires a couple of preliminaries, and so we begin in the next section with a recap of the spectral flow and the Maslov index. In the section on the spectral flow, we will recall several facts about the gap-metric on the space of selfadjoint Fredholm operators. Afterwards, we introduce Theorem A,B and C in detail and prove them in the order B, C, A. However, we want to point out that B and C are independent of each other and so the reader might also read C before B.\\
Finally, we want to make a few remarks on our notation. In general, $H$ is a real or complex Hilbert space unless otherwise stated. However, if we require the Hilbert space to be real, we denote it by $E$ instead of $H$. The symbols $\mathcal{L}(H)$ and $GL(H)$ stand for the bounded and invertible operators on $H$, respectively. Moreover, $\mathcal{BF}(H)$ denotes the subspace of $\mathcal{L}(H)$ of all bounded Fredholm operators. In this paper, we mostly deal with unbounded operators, and we denote by $\mathcal{C}(H)\supset\mathcal{L}(H)$ the closed operators on $H$. As usual, $\sigma(S)$ stands for the spectrum of $S\in\mathcal{C}(H)$, and $S^\ast$ denotes its adjoint. However, also here we use a different notation for real Hilbert spaces $E$, where the adjoint will be denoted by $S^T$ as for the transpose of a real matrix. This is in accordance with the notation in \cite{WaterstraatHomoclinics}, where we considered the equations \eqref{Hamiltonian} in finite dimensions. Finally, the identity on $E$ is denoted by $I_E$, which we abbreviate by $I_{2n}$ in the case that $E=\mathbb{R}^{2n}$.   

\subsubsection*{Acknowledgement}
\vspace*{-0.2cm}
There was a further assumption in Theorem A in a previous version of this work, which required that the relative dimension \eqref{Hermann} vanishes. We want to thank Hermann Schulz-Baldes (Erlangen) who explained to us that this actually follows from (A1) and (A2).

%%%%%%%%%%%%%%%%%%%%%%%%%%%%%%%%%%%%%%%%%%%%%%%%%%%%%%%%%%%%%%%%%%%%%%%%%%%%%%%%%%%%%%%%%%%%%%%%%%%%%%%%%%%%%%%%%%%%%%%%%%%%%%%%%%%%%%%%%%%%%%%%%%%%%%%%%%%%%%%%%%%%%%%%%%%%%%%%%%%%%%%%%%%%%%%%%%%%%%%%%%%%%%%%%%%%%%%%%%%%%%%%%%%%%%%%%%%%%%%%%%%%%%%%%%%%%%%%%%%%%%%%%%%%%%%%%%%%%%%%%%%%%%%%%%%%%%%%%%%%%%%%%%%%%%%%%%%%%%%%%%%%%%%%%%%%%%%%%%%%%%%%%%%%%%%%%%%%%%%%%%

\section{Preliminaries: Maslov Index and Spectral Flow}

\subsection{Fredholm Lagrangian Grassmannian and the Maslov Index}\label{subsection-Maslov}
The aim of this section is to recall some basic facts about the Maslov index for paths of Fredholm pairs of Lagrangian subspaces in a symplectic Hilbert space. There are different notions of symplectic Hilbert spaces (see e.g. \cite{BoossZhuII}), but here such spaces are Hilbert spaces $(E,\langle\cdot,\cdot\rangle)$ with an invertible bounded operator $J:E\rightarrow E$ such that $J^T=-J$ and $J^2=-I_E$, where $J^T$ denotes the adjoint of $J$. The corresponding symplectic form on $E$ is given by $\omega(x,y)=\langle Jx,y\rangle$. Our main reference in this section is Furutani's work \cite{Furutani}, who defines a symplectic Hilbert space as a pair $(E,\widetilde{\omega})$ where $\widetilde{\omega}$ is a non-degenerate skew-symmetric bounded bilinear form. Of course, our form $\omega$ has all these properties. Moreover, it is shown in \cite{Furutani} that every $\widetilde{\omega}$ is of the form of our $\omega$ for some operator $J$ which has the above properties if we just modify the scalar product of $E$ to an equivalent one.\\
An important observation, here and below, is that the set of all closed subspace $\mathcal{G}(E)$ in a Hilbert space $E$ is canonically a metric space with respect to the \textit{gap-metric}

\begin{align}\label{gap-subspaces}
d_G(U,V)=\|P_U-P_V\|,\quad U,V\in\mathcal{G}(E),
\end{align}
where $P_U$ and $P_V$ denote the orthogonal projections onto $U$ and $V$, respectively. Actually, $\mathcal{G}(E)$ is an analytic Banach manifold (see \cite{AbbondandoloMajerGrass}), which however, will not be needed in this paper. Instead, we now suppose that $E$ is a symplectic Hilbert space and consider a submanifold of $\mathcal{G}(E)$. Let us first recall that a closed subspace $L\subset E$ is called \textit{Lagrangian} if  

\[L=L^\circ:=\{x\in E:\omega(x,y)=0\text{ for all } y\in E\}.\]
Henceforth, we will hardly make use of this definition, but use the elementary fact that $L\in\mathcal{G}(E)$ is Lagrangian if and only if

\begin{align}\label{Lperp=JL}
L^\perp=J(L),
\end{align}
where $L^\perp$ denotes the orthogonal complement with respect to the scalar product $\langle\cdot,\cdot\rangle$ of $E$ (\cite[Prop. 1.7]{Furutani}). Note that the set of all Lagrangian subspaces $\Lambda(E,\omega)$ of $E$ inherits a metric from $\mathcal{G}(E)$, but let us mention in passing that the topology of $\Lambda(E,\omega)$ depends substantially on the dimension of $E$. Indeed, if $E$ is of finite dimension, then $\Lambda(E,\omega)$ has an infinitely cyclic fundamental group and an isomorphism to the integers is given by the Maslov index \cite{Maslov}. In contrast, if $E$ is of infinite dimension, as we usually assume in this paper, then $\Lambda(E,\omega)$ is contractible as a consequence of Kuiper's Theorem (see \cite[Thm. 1.14]{Furutani}). However, $\Lambda(E,\omega)$ contains an interesting subset which is topologically non-trivial and has shown a lot of times to be the right setting for generalising the Maslov index to infinite dimensions.\\
Let us recall that two subspaces $L,M\in\mathcal{G}(E)$ are a \textit{Fredholm pair} if

\[\dim(L\cap M)<\infty\,\, \text{ and }\, \codim(L+M)<\infty,\]
and the \textit{index} of a Fredholm pair is defined by

\[\ind(L,M)=\dim(L\cap M)-\codim(L+M).\] 
It is often required in the definition of a Fredholm pair that $L+M$ is closed. That this is redundant was explained, e.g., in \cite{BoossFurutani}.\\
For a fixed $W\in\Lambda(E,\omega)$, we denote by $\mathcal{FL}_W(E,\omega)$ the set of all $L\in\Lambda(E,\omega)$ such that $(L,W)$ is a Fredholm pair, and we note that 

\begin{align}\label{ind=0Lagrangian}
\ind(L,W)=0,\quad L\in\mathcal{FL}_W(E,\omega),
\end{align}
by \eqref{Lperp=JL} (see \cite[(1.3)]{NicolaescuI}). It can be shown that $\mathcal{FL}_W(E,\omega)$ is an open subset of $\Lambda(E,\omega)$, and moreover, it has an infinitely cyclic fundamental group by \cite[Thm. 1.54]{Furutani}. The Maslov index extends to this infinite dimensional setting as integer valued invariant $\mu_{Mas}(\Lambda,W)$ for paths $\Lambda$ in $\mathcal{FL}_W(E,\omega)$. Its heuristic interpretation is as in the finite dimensional case, namely, it is the net number of intersections of $\Lambda(\lambda)$ and $W$ whilst $\lambda$ travels along the unit interval. The construction of the Maslov index consists of two parts. Firstly, there is a map from $\mathcal{FL}_W(E,\omega)$ to a set $U_J$ of unitary operators on a complex Hilbert space. Secondly, there is a winding number for paths of operators in $U_J$. The composition of these maps is the Maslov index and indeed reduces to the classical one if $E$ is of finite dimension. We recap this construction from \cite{Furutani} in Appendix \ref{app-Maslov}, where we need it to prove Lemma \ref{MaslovMidpoint} which is crucial in the final step of the proof of Theorem A. Apart from this, we will not use any particular details about the construction, but just need the following three basic properties which can all be found in \cite{Furutani}:

\begin{enumerate}
\item[(i)] If $\Lambda(\lambda)\cap W=\{0\}$ for all $\lambda\in I$, then $\mu_{Mas}(\Lambda,W)=0$.
\item[(ii)] The Maslov index is additive under the concatenation of paths, i.e.
\[\mu_{Mas}(\Lambda_1\ast\Lambda_2,W)=\mu_{Mas}(\Lambda_1,W)+\mu_{Mas}(\Lambda_2,W)\]
if $\Lambda_1,\Lambda_2:I\rightarrow\mathcal{FL}_W(E,\omega)$ are two paths such that $\Lambda_1(1)=\Lambda_2(0)$.
\item[(iii)] If $\Lambda:I\times I\rightarrow\mathcal{FL}_W(E,\omega)$ is a homotopy such that $\Lambda(s,0)$ and $\Lambda(s,1)$ are constant for all $s\in I$, then
\[\mu_{Mas}(\Lambda(0,\cdot),W)=\mu_{Mas}(\Lambda(1,\cdot),W).\]
\end{enumerate}
Finally, given the fact that the fundamental group of $\mathcal{FL}_W(E,\omega)$ is infinitely cyclic, it is not difficult to see that $\mu_{Mas}$ actually provides an explicit isomorphism between $\mathcal{FL}_W(E,\omega)$ and the integers (see \cite[\S 3]{Furutani}).\\
As in the finite dimensional case, the Maslov index can be generalised to pairs of subspaces. Note that the diagonal $\Delta$ in $E\times E$ is a Lagrangian subspace, when $E\times E$ is considered as symplectic Hilbert space with respect to the symplectic form $\omega_{E\times E}=\omega_E\times(-\omega_E)$. It is readily seen that $\Lambda_1(\lambda)\times\Lambda_2(\lambda)\in \mathcal{FL}_\Delta(E\times E,\omega_{E\times E})$ if $(\Lambda_1(\lambda),\Lambda_2(\lambda))\in\mathcal{FL}^2(E,\omega)$, where the latter set denotes the set of all Fredholm pairs of Lagrangian subspaces of $E$. The Maslov index of a path of pairs $(\Lambda_1,\Lambda_2)$ in $\mathcal{FL}^2(E,\omega)$ is defined as the Maslov index of $\Lambda_1\times\Lambda_2$ as a path in $\mathcal{FL}_\Delta(E\times E,\omega_{E\times E})$. It is shown in \cite[Prop. 2.32]{Furutani} that $\mu_{Mas}(\Lambda_1,\Lambda_2)=\mu_{Mas}(\Lambda_1,W)$ if $\Lambda_2\equiv W$ is a constant path, so that this is indeed an extension of the Maslov index for paths in $\mathcal{FL}_W(E,\omega)$. Of course, we obtain as immediate results from the above properties (i)-(iii)

\begin{enumerate}
\item[(i')] If $\Lambda_1(\lambda)\cap \Lambda_2(\lambda)=\{0\}$ for all $\lambda\in I$, then $\mu_{Mas}(\Lambda_1,\Lambda_2)=0$.
\item[(ii')] The Maslov index is additive under the concatenation of paths, i.e.
\[\mu_{Mas}((\Lambda_1,\Lambda_2)\ast(\tilde{\Lambda}_1,\tilde{\Lambda}_2))=\mu_{Mas}(\Lambda_1,\Lambda_2)+\mu_{Mas}(\tilde{\Lambda}_1,\tilde{\Lambda}_2)\]
if $(\Lambda_1,\Lambda_2),(\tilde{\Lambda}_1,\tilde{\Lambda}_2):I\rightarrow\mathcal{FL}^2(E,\omega)$ are two pairs of paths such that $(\Lambda_1(1),\Lambda_2(1))=(\tilde{\Lambda}_1(0),\tilde{\Lambda}_2(0))$.
\item[(iii')] If $(\Lambda_1,\Lambda_2):I\times I\rightarrow\mathcal{FL}^2(E,\omega)$ is a homotopy such that $\Lambda_1(s,0),\Lambda_2(s,0)$, $\Lambda_1(s,1)$ and $\Lambda_2(s,1)$ are constant for all $s\in I$, then
\[\mu_{Mas}((\Lambda_1(0,\cdot),\Lambda_2(0,\cdot)))=\mu_{Mas}((\Lambda_1(1,\cdot),\Lambda_2(1,\cdot))).\]
\end{enumerate}
As $\pi_1(\mathcal{FL}^2(E,\omega))\cong\mathbb{Z}$ by \cite[Cor. 1.6]{NicolaescuI} if $E$ is of infinite dimension, the following assertion is an immediate consequence of the definition of $\mu_{Mas}$ on $\mathcal{FL}^2(E,\omega)$ and the fact that it is an isomorphism $\pi_1(\mathcal{FL}_W(E,\omega))\rightarrow\mathbb{Z}$ for fixed Lagrangian subspaces $W$.

\begin{theorem}\label{Furutani-iso}
The Maslov index

\[\mu_{Mas}:\pi_1(\mathcal{FL}^2(E,\omega))\rightarrow\mathbb{Z}\]
is an isomorphism if $E$ is of infinite dimension.
\end{theorem} 
\noindent
Note that Theorem \ref{Furutani-iso} is wrong for finite dimensional spaces $E$. Indeed, if $E=\mathbb{R}^{2n}$, then $\mathcal{FL}^2(\mathbb{R}^{2n},\omega)=\Lambda(\mathbb{R}^{2n})\times\Lambda(\mathbb{R}^{2n})$ and consequently 

\[\pi_1(\mathcal{FL}^2(\mathbb{R}^{2n},\omega))=\pi_1(\Lambda(\mathbb{R}^{2n}))\times\pi_1(\Lambda(\mathbb{R}^{2n}))=\mathbb{Z}\oplus\mathbb{Z}.\]

%%%%%%%%%%%%%%%%%%%%%%%%%%%%%%%%%%%%%%%%%%%%%%%%%%%%%%%%%%%%%%%%%%%%%%%%%%%%%%%%%%%%%%%%%%%%%%%%%%%%%%%%%%%%%%%%%%%%%%%%%%%%%%%%%%%%%%%%%%%%%%%%%%%%%%%%%%%%%%%%%%%%%%%%%%%%%%%%%%%%%%%%%%%%%%%%%%%%%%%%%%%%%%%%%%%%%%%%%%%%%%%%%%%%%%%%%%%%%%%%%%%%%%%%%%%%%%%%%%%%%%%%%%%%%%%%%%%%%%%%%%%%%%%%%%%%%%%%%%%%%%%%%%%%%%%%%%%%%%%%%%%%%%%%%%%%%%%%%%%%%%

\subsection{The Spectral Flow in the Gap Metric}

%%%%%%%%%%%%%%%%%%%%%%%%%%%%%%%%%%%%%%%%%%%%%%%%%%%%%%%%%%%%%%%%%%%%%%%%%%%%%%%%%%%%%%%%%%%%%%%%%%%%%%%%%%%%%%%%%
\subsubsection{Fredholm Operators and the Gap Metric}\label{section-gap}
In this section, we consider (possibly) unbounded operators $T:\mathcal{D}(T)\subset H\rightarrow H$, which are defined on a dense subspace $\mathcal{D}(T)$ of the Hilbert space $H$ which can be either real or complex. Let us recall that $T$ is \textit{closed} if its graph, which we henceforth denote by $\gra(T)$, is a closed subspace of $H\times H$. Note that the set $\mathcal{C}(H)$ of all closed operators on $H$ can be canonically embedded into the Grassmannian $\mathcal{G}(H\times H)$ and so inherits a metric. In other words, 

\begin{align}\label{gap operators}
d_G(S,T)=\|P_{\gra(S)}-P_{\gra(T)}\|,\quad S,T\in\mathcal{C}(H),
\end{align}
defines a metric on $\mathcal{C}(H)$, which is called the \textit{gap-metric}. The topologies induced by the operator norm and the gap-metric on the subset of bounded operators $\mathcal{L}(H)\subset\mathcal{C}(H)$ are equivalent (see \cite[Rem. IV.2.16]{Kato}). In particular, every norm-continuous family of operators in $\mathcal{L}(H)$ is also continuous in $\mathcal{C}(H)$. In what follows, we will use this fact without further reference. Finally, note that even though $\mathcal{G}(H\times H)$ is complete, $(\mathcal{C}(H),d_G)$ is not, which is readily seen by considering a sequence of graphs that converges in $\mathcal{G}(H\times H)$ to a space which has a non-trivial intersection with $\{0\}\times H$.\\
There are two subsets of $\mathcal{C}(H)$ that will be of particular interest for us. Firstly, let us recall that a densely defined operator $T$ is called selfadjoint if it is symmetric and $\mathcal{D}(T)=\mathcal{D}(T^\ast)$, where $T^\ast$ denotes the adjoint of $T$. Clearly, every selfadjoint operator is closed, and in what follows we denote by $\mathcal{C}^\textup{sa}(H)\subset\mathcal{C}(H)$ the subset of all selfadjoint operators. Secondly, an operator $T\in\mathcal{C}(H)$ is \textit{Fredholm} if its kernel and its cokernel are of finite dimension. The difference of these numbers is the \textit{index} of $T$. Let us point out that every Fredholm operator has a closed range $\im(T)\subset H$ (see \cite{Gohberg}). Henceforth, we denote by $\mathcal{CF}(H)\subset\mathcal{C}(H)$ the subset of all Fredholm operators, and by $\mathcal{CF}_k(H)$ the elements in $\mathcal{CF}(H)$ of index $k\in\mathbb{Z}$. Note that there is an important difference between the previous definitions: a selfadjoint operator is automatically closed, whereas we require a Fredholm operator to be closed in its definition. In what follows, we will be in particular interested in the intersection of $\mathcal{C}^\textup{sa}(H)$ and $\mathcal{CF}(H)$, i.e. the set of selfadjoint Fredholm operators, which we denote by $\mathcal{CF}^\textup{sa}(H)$. Note that every element of $\mathcal{CF}^\textup{sa}(H)$ has Fredholm index $0$, i.e. $\mathcal{CF}^\textup{sa}(H)\subset\mathcal{CF}_0(H)$. The next lemma, that we will use several times below, gives a complete characterisation of which elements of $\mathcal{CF}_0(H)$ actually belong to $\mathcal{CF}^\textup{sa}(H)$.

\begin{lemma}\label{symmetric-selfadjoint}
If $T\in\mathcal{CF}_0(H)$ is symmetric, then $T\in\mathcal{CF}^\textup{sa}(H)$. In other words, a symmetric Fredholm operator of index $0$ is selfadjoint. 
\end{lemma}

\begin{proof}
As $T$ is symmetric, we see that $\ker(T)\subset(\im T)^\perp$, and since both spaces are of the same dimension for Fredholm operators of index $0$, this shows that

\begin{align}\label{ker=im}
(\im T)^\perp=\ker(T).
\end{align}  
We now claim that every symmetric Fredholm operator which satisfies \eqref{ker=im} is selfadjoint, where we follow an argument that we have learnt from the proof of Proposition 3.1 in \cite{SalamonWehrheim}. We let $u\in\mathcal{D}(T^\ast)$ and note at first that $\langle u,T v\rangle=\langle w,v\rangle$ for $w=T^\ast u\in H$ and all $v\in\mathcal{D}(T)$. As $\im(T)$ is closed, we see from \eqref{ker=im} that there are $w_1\in\ker(T)$ and $u_1\in\mathcal{D}(T)$ such that $w=w_1+Tu_1$. Therefore,

\begin{align}\label{selfdomain}
\langle u-u_1,Tv\rangle=\langle T^\ast u,v\rangle-\langle Tu_1,v\rangle=\langle w-T u_1,v\rangle=\langle w_1,v\rangle,
\end{align}
and the latter term vanishes for all $v\in\im(T)\cap\mathcal{\mathcal{D}}(T)$ by \eqref{ker=im}.\\
By \eqref{ker=im}, every $v\in\mathcal{D}(T)$ can be written as $v=v_1+v_2$ where $v_1\in\ker(T)$ and $v_2\in\im(T)$. As $\ker(T)\subset\mathcal{D}(T)$, we see that actually $v_2\in \im(T)\cap\mathcal{D}(T)$. Hence, by \eqref{selfdomain},

\[\langle u-u_1,Tv\rangle=\langle u-u_1,Tv_2\rangle=0,\quad v\in\mathcal{D}(T),\]
and so $u-u_1\in(\im T)^\perp=\ker(T)\subset\mathcal{D}(T)$, where we have used once again \eqref{ker=im}. Since $u_1\in\mathcal{D}(T)$, we finally obtain $u\in\mathcal{D}(T)$ and so $T$ is selfadjoint.
\end{proof}
\noindent
As usual, if $T:\mathcal{D}(T)\subset H\rightarrow H$ and $S:\mathcal{D}(S)\subset H\rightarrow H$ are densely defined, their composition $TS$ is an operator on $\mathcal{D}(TS)=S^{-1}(\mathcal{D}(T))$. We note the following simple corollary of Lemma \ref{symmetric-selfadjoint} for later reference.

\begin{cor}\label{cor-orthequ}
If $T\in\mathcal{CF}^\textup{sa}(H)$ and $M\in GL(H)$, then $M^\ast TM\in\mathcal{CF}^\textup{sa}(H)$.
\end{cor}

\begin{proof}
We just need to note that $M^\ast TM$ is obviously symmetric. Moreover, it is Fredholm of index $0$, as the product of densely defined Fredholm operators is Fredholm and 

\[\ind(M^\ast TM)=\ind(M^\ast)+\ind(T)+\ind(M)=\ind(T)=0\]
by \cite[Thm. XVIII 3.1]{Gohberg}.
\end{proof}
\noindent
If $W\subset H$ is a dense subset that is a Hilbert space in its own right, then we can consider

\begin{align}\label{BF}
\mathcal{BF}^\textup{sa}(W,H):=\{T\in\mathcal{L}(W,H):\, T\,\text{ Fredholm},\, T^\ast=T\},
\end{align}
where the adjoint is meant as adjoint of an unbounded operator on $H$ with dense domain $W$. Note that $\mathcal{BF}^\textup{sa}(W,H)$ inherits a topology from the space of bounded operators $\mathcal{L}(W,H)$. On the other hand, $\mathcal{BF}^\textup{sa}(W,H)$ is a subset of $\mathcal{CF}^\textup{sa}(H)$ and so one might ask about the relation of the different topologies. This was answered by Lesch in \cite[Prop. 2.2]{Lesch} as follows.

\begin{theorem}\label{Leschincl}
The canonical inclusion 

\[\mathcal{BF}^\textup{sa}(W,H)\subset\mathcal{CF}^\textup{sa}(H)\]
is continuous.
\end{theorem}  
\noindent
In particular, any path in $\mathcal{BF}^\textup{sa}(W,H)$ is also continuous with respect to the gap-topology, which we will use below in the proof of Theorem A.

%%%%%%%%%%%%%%%%%%%%%%%%%%%%%%%%%%%%%%%%%%%%%%%%%%%%%%%%%%%%%%%%%%%%%%%%%%%%%%%%%%%%%%%%%%%%%%%%%%%%%%%%%%%%%%%%%

\subsubsection{The Spectral Flow}\label{section-sfl}
The reader who is well acquainted with the spectral flow as introduced in \cite{UnbSpecFlow} will just need to skim through the rest of this section to become familiar with our notations. Let us point out that we denote, as in previous sections, parameters by $\lambda$. We are aware that it is common in the literature to use $t$ instead, but this would clash with the variable $t$ in \eqref{Hamiltonian}. In particular, let us emphasize that in what follows, $\lambda$ is never an element of the spectrum of an operator.\\
We recall at first that for every selfadjoint Fredholm operator $T$ there is $\varepsilon>0$ and a neighbourhood $\mathcal{N}_{T,\varepsilon}\subset\mathcal{CF}^\textup{sa}(H)$ such that $\pm\varepsilon\notin\sigma(S)$ and the spectral projection $\chi_{[-\varepsilon,\varepsilon]}(S)$ is of finite rank for all $S\in\mathcal{N}_{T,\varepsilon}$. If now $\mathcal{A}:I\rightarrow\mathcal{CF}^\textup{sa}(H)$ is a path, then there are $0=\lambda_0<\lambda_1<\ldots<\lambda_N=1$ such that the restriction of $\mathcal{A}$ to $[\lambda_{i-1},\lambda_i]$ is contained in a neighbourhood $\mathcal{N}_{T_i,\varepsilon_i}$ for some $T_i\in\mathcal{CF}^\textup{sa}(H)$ and $\varepsilon_i>0$. We set

\begin{align}\label{sfl}
\sfl(\mathcal{A})=\sum^N_{i=1}{\left(\dim(\im(\chi_{[0,\varepsilon_i]}(\mathcal{A}_{\lambda_i}))-\dim(\im(\chi_{[0,\varepsilon_i]}(\mathcal{A}_{\lambda_{i-1}}))\right)}.
\end{align}
Note that the dimensions of the images of the spectral projections in \eqref{sfl} are just the number of eigenvalues in $[0,\varepsilon_i]$ including their multiplicities.\\ 
It was first observed by Philips in \cite{Philips} that this definition does not depend on the choices of the numbers $\lambda_i$ and $\varepsilon_i$. Note that, roughly speaking, the spectral flow is the net number of eigenvalues of $\mathcal{A}_0$ that cross zero whilst the parameter $\lambda$ travels along the interval $I$. The most important properties of the spectral flow are

\begin{enumerate}
\item[(i)] If $\mathcal{A}_\lambda$ is invertible for all $\lambda\in I$, then $\sfl(\mathcal{A})=0$.
\item[(ii)] If $\mathcal{A}^1$ and $\mathcal{A}^2$ are two paths in $\mathcal{CF}^\textup{sa}(H)$ such that $\mathcal{A}^1_1=\mathcal{A}^2_0$, then

\[\sfl(\mathcal{A}^1\ast\mathcal{A}^2)=\sfl(\mathcal{A}^1)+\sfl(\mathcal{A}^2).\]
\item[(iii)] Let $h:I\times I\rightarrow\mathcal{CF}^\textup{sa}(H)$ be a homotopy such that $h(s,0)$ and $h(s,1)$ are invertible for all $s\in I$. Then 
\[\sfl(h(0,\cdot))=\sfl(h(1,\cdot)).\]
\end{enumerate} 
Let us point out that the first two properties are immediate consequences of the definition \eqref{sfl}, whereas the third one requires a little bit of work. Actually, it is easy to see that the homotopy invariance even holds when the endpoints are not invertible as long as the dimension of the kernels of $h(s,0)$ and $h(s,1)$ are constant. This is obvious from the interpretation of the spectral flow, and also easy to see from the proof of (iii) in \cite{Philips}. Let us finally note the following stability property of the spectral flow for later reference (cf. \cite[\S 7]{Pejsachowicz}).

\begin{lemma}\label{sflperturbation}
Let $\mathcal{A}:I\rightarrow\mathcal{CF}^\textup{sa}(H)$ be gap-continuous and $\mathcal{A}^\delta=\mathcal{A}+\delta I_H$ for $\delta\in\mathbb{R}$. Then

\[\sfl(\mathcal{A})=\sfl(\mathcal{A}^\delta)\]  
for any sufficiently small $\delta>0$.
\end{lemma}

\begin{proof}
We note at first that the operators $\mathcal{A}^\delta_\lambda$ are in $\mathcal{CF}^\textup{sa}(H)$ for $\delta$ sufficiently small, and moreover the path $\mathcal{A}^\delta$ is gap-continuous by \cite[Thm. IV.2.17]{Kato}. Hence $\sfl(\mathcal{A}^\delta)$ is well defined.\\ 
Let now $0=\lambda_0<\ldots<\lambda_N=1$ be a partition of the unit interval and $\varepsilon_i>0$, $i=1,\ldots,N$, for the path $\mathcal{A}$ as in \eqref{sfl}. Note that $\sigma(\mathcal{A}^\delta_\lambda)=\sigma(\mathcal{A}_\lambda)+\delta$, $\lambda\in I$. We now let $\delta>0$ be so small that $\pm\varepsilon_i\notin\sigma(\mathcal{A}^{s\delta}_\lambda)$, $s\in I$, $\lambda\in[\lambda_{i-1},\lambda_i]$, for $i=1,\ldots,N$, and $\sigma(\mathcal{A}_{\lambda_i})\cap[-\delta,0)=\{0\}$ for $i=0,\ldots N$. Then

\[\dim(\im(\chi_{[0,\varepsilon_i]}(\mathcal{A}_{\lambda})))=\dim(\im(\chi_{[0,\varepsilon_i]}(\mathcal{A}^\delta_{\lambda}))),\quad \lambda=\lambda_{i-1},\lambda_i,\]
for $i=1,\ldots,N$, and the assertion is an immediate consequence of the definition \eqref{sfl}. 
\end{proof}
\noindent
Let us point out that it is important that $\delta$ is positive in the previous lemma. For negative $\delta$ the difference of the kernel dimensions of $\mathcal{A}_0$ and $\mathcal{A}_1$ appears as additional term, which is readily seen by a similar argument.\\
We will need below a characterisation of the spectral flow that is due to Lesch \cite{Lesch}. Let us denote by $\Omega(\mathcal{CF}^\textup{sa}(H),G\mathcal{C}^\textup{sa}(H))$ the set of all paths in $\mathcal{CF}^\textup{sa}(H)$ having invertible endpoints. Note that a selfadjoint Fredholm operator is invertible if and only if its kernel is trivial. Let $P_+,P_-$ and $P_0$ be three orthogonal projections in $H$ such that $P_+,P_-$ have infinite dimensional kernel and range, and $\dim(\im P_0)=1$. We also assume that these projections are complementary, which means that the products of each two of them vanish and that $P_++P_0+P_-$ is the identity $I_H$. Then 

\begin{align}\label{normalisationpath}
\mathcal{A}_\lambda=P_-+(\lambda-\frac{1}{2})P_0+P_+
\end{align}
is a path of bounded selfadjoint operators which is invertible as long as $\lambda\neq\frac{1}{2}$. For $\lambda=\frac{1}{2}$ the image of $P_0$ is the kernel and cokernel of $\mathcal{A}_{\frac{1}{2}}$ and so this operator is Fredholm. As the canonical inclusion of the bounded selfadjoint Fredholm operators $\mathcal{BF}^\textup{sa}(H)$ into $\mathcal{CF}^\textup{sa}(H)$ is continuous by \cite[Prop. 2.2]{Lesch}, we see that $\mathcal{A}_{nor}:=\{\mathcal{A}_\lambda\}_{\lambda\in I}$ is a path in $\mathcal{CF}^\textup{sa}(H)$. The reader will have no difficulty to see from \eqref{sfl} that $\sfl(\mathcal{A}_{nor})=1$. The following theorem was proved by Lesch in \cite{Lesch}. 

\begin{theorem}\label{LeschUniqueness}
Assume that $\mu:\Omega(\mathcal{CF}^\textup{sa}(H),G\mathcal{C}^\textup{sa}(H))\rightarrow\mathbb{Z}$ is a map that has the same properties (ii) and (iii) as the spectral flow. If $\mu(\mathcal{A}_{nor})=1$, then 

\[\mu=\sfl:\Omega(\mathcal{CF}^\textup{sa}(H),G\mathcal{C}^\textup{sa}(H))\rightarrow\mathbb{Z}.\]  
\end{theorem}
\noindent
The reader should not be puzzled that the property (i) is not mentioned in Theorem \ref{LeschUniqueness}, as it follows from (ii) and (iii). Indeed, it is readily seen from (ii) that the spectral flow of a constant path vanishes. As every path of invertible operators can be contracted to a point by a homotopy of invertible operators, (i) now follows from (iii).\\
Finally, let us consider the case of a path $\mathcal{A}$ in $\mathcal{CF}^\textup{sa}(E)$, where $E$ is a real Hilbert space. In this case there are two ways to define the spectral flow of $\mathcal{A}$. Firstly, as we previously allowed our Hilbert spaces to be real or complex, we can use \eqref{sfl} as introduced above. Secondly, we can consider the complexification $E^\mathbb{C}=E+iE$ of $E$ which is canonically a complex Hilbert space. The complexified operators $\mathcal{A}^\mathbb{C}_\lambda$ are in $\mathcal{CF}^\textup{sa}(E^\mathbb{C})$, and so the spectral flow of the complexified path $\mathcal{A}^\mathbb{C}=\{\mathcal{A}^\mathbb{C}_\lambda\}_{\lambda\in I}$ is defined as well. As the complex dimensions of eigenspaces of $\mathcal{A}^\mathbb{C}_\lambda$ is equal to the real dimension of the eigenspaces of $\mathcal{A}_\lambda$, we see from \eqref{sfl} that

\begin{align}\label{complexification}
\sfl(\mathcal{A})=\sfl(\mathcal{A}^\mathbb{C}).
\end{align} 
Even though the operators for studying the equations \eqref{Hamiltonian} are defined in real Hilbert spaces, one of our topological constructions below requires operators in complex Hilbert spaces. The obtained equation \eqref{complexification} will become important in that step.

%\begin{lemma}\label{sflequal}
%Let $\mathcal{A}^1,\mathcal{A}^2:I\rightarrow\mathcal{CF}^\textup{sa}(H)$ be two paths and $\delta>0$ such that

%\begin{align}\label{sflequalker}
%\dim E(\mathcal{A}^1_\lambda,\mu)=\dim E(\mathcal{A}^2_\lambda,\mu),\quad |\mu|<\delta.
%\end{align}
%Then
%\[\sfl(\mathcal{A}^1)=\sfl(\mathcal{A}^2).\]
%\end{lemma}

%\begin{proof}

%\end{proof}

%%%%%%%%%%%%%%%%%%%%%%%%%%%%%%%%%%%%%%%%%%%%%%%%%%%%%%%%%%%%%%%%%%%%%%%%%%%%%%%%%%%%%%%%%%%%%%%%%%%%%%%%%%%%%%%%%%%%%%%%%%%%%%%%%%%%%%%%%%%%%%%%%%%%%%%%%%%%%%%%%%%%%%%%%%%%%%%%%%%%%%%%%%%%%%%%%%%%%%%%%%%%%%%%%%%%%%%%%%%%%%%%%%%%
%%%%%%%%%%%%%%%%%%%%%%%%%%%%%%%%%%%%%%%%%%%%%%%%%%%%%%%%%%%%%%%%%%%%%%%%%%%%%%%%%%%%%%%%%%%%%%%%%%%%%%%%%%%%%%%%%%%%%%%%%%%%%%%%%%%%%%%%%%%%%%%%%%%%%%%%%%%%%%%%%%%%%%%%%%%%%%%%%%%%%%%%%%%%%%%%%%%%%%%%%%%%%%%%%%%%%%%%%%%%%%%%%%
%%%%%%%%%%%%%%%%%%%%%%%%%%%%%%%%%%%%%%%%%%%%%%%%%%%%%%%%%%%%%%%%%%%%%%%%%%%%%%%%%%%%%%%%%%%%%%%%%%%%%%%%%%%%%%%%%%%%%%%%%%%%%%%%%%%%%%%%%%%%%%%%%%%%%%%%%%%%%%%%%%%%%%%%%%%%%%%%%%%%%%%%%%%%%%%%%%%%%%%%%%%%%%%%%%%%%%%%%%%%%%%%%%%%%%%%%%%%%%%%%%%%%%%%%%%%%%%%%%%%%%%%%%%%%%%%%%%%%%%%%%%%%%%%%%%%%%%%%%%%%%%%%%%%%%%%%%%%%%%%%%%%%%%%%%%%%%%%%%%

\section{Theorem B}
We now have recalled all necessary preliminaries for discussing Theorem B in detail. Let us point out once again that Theorem C, which we prove in the following section, is independent of Theorem B.

\subsection{Setting and Statement of the Theorem}
Let $(E,\omega)$ be a symplectic Hilbert space and $\{(\Lambda_0(\lambda),\Lambda_1(\lambda))\}_{\lambda\in I}$ a path in $\mathcal{FL}^2(E,\omega)$. As before, we let $J:E\rightarrow E$ be the almost complex structure induced by $\omega$ and assume that $J$ is compatible with the scalar product of $E$, i.e. $J^2=-I_E$ and $J^T=-J$. Now we consider for $a,b\in\mathbb{R}$, $a<b$, the differential operators

\begin{align}\label{Q}
\mathcal{Q}_\lambda:\mathcal{D}(\mathcal{Q}_\lambda)\subset L^2([a,b],E)\rightarrow L^2([a,b],E),\quad \mathcal{Q}_\lambda u=Ju',
\end{align}
where

\[\mathcal{D}(\mathcal{Q}_\lambda)=\{u\in H^1([a,b],E):\, u(a)\in\Lambda_0(\lambda), u(b)\in\Lambda_1(\lambda)\}.\]
Our first aim is to show that the Fredholm and Lagrangian properties of $(\Lambda_0(\lambda),\Lambda_1(\lambda))$ are strictly related to the Fredholmness and selfadjointness of $\mathcal{Q}_\lambda$.

\begin{lemma}\label{lemma-selfFredIFF}
The operator $\mathcal{Q}_\lambda$ belongs to $\mathcal{CF}^\textup{sa}(L^2([a,b],E))$ if and only if 

\[(\Lambda_0(\lambda),\Lambda_1(\lambda))\in\mathcal{FL}^2(E,\omega).\]
\end{lemma}

\begin{proof}
We begin by examining the kernel and cokernel of $\mathcal{Q}_\lambda$. Obviously, the kernel of $\mathcal{Q}_\lambda$ is isomorphic to $\Lambda_0(\lambda)\cap\Lambda_1(\lambda)$. For studying the cokernel, we note at first that for $w\in L^2([a,b],E)$ the functions

\[u(t)=-J\int^t_a{w(s)\,ds}+c,\, t\in [a,b],\quad c\in E,\]
are in $H^1([a,b],E)$ and satisfy $Ju'(t)=w(t)$, $t\in [a,b]$. Now $u$ belongs to $\mathcal{D}(\mathcal{Q}_\lambda)$ if and only if

\begin{align}\label{BC}
\begin{split}
u(a)&=c\in\Lambda_0(\lambda),\\
u(b)&=-J\int^b_a{w(s)\,ds}+c\in\Lambda_1(\lambda).
\end{split}
\end{align} 
By writing 

\[w(t)=\int^b_a{w(s)\,ds}+(w(t)-\int^b_a{w(s)\,ds}),\]
it is clear that we have a decomposition

\begin{align*}
L^2([a,b],E)=U\oplus V
\end{align*}
into closed subspaces, where $U\cong E$ denotes the space of constant functions and 

\[V=\left\{w\in L^2([a,b],E):\,\int^b_a{w(s)\,ds}=0\right\}.\]
We first note that \eqref{BC} holds for any $w\in V$ just by setting $c=0$. If, however, $w\in U$, then \eqref{BC} holds if and only if there exists $c\in\Lambda_0(\lambda)$ such that $v+c\in\Lambda_1(\lambda)$, where $v:=-J(b-a)w$. Now the latter assertion is true if and only if $v\in\Lambda_0(\lambda)+\Lambda_1(\lambda)$, and so the range of $\mathcal{Q}_\lambda$ is $J(\Lambda_0(\lambda)+\Lambda_1(\lambda))\oplus V$. In summary, the kernel and cokernel of $\mathcal{Q}_\lambda$ are of finite dimension if and only if $(\Lambda_0(\lambda),\Lambda_1(\lambda))$ is a Fredholm pair.\\ 
Next, we note that for $u,v\in\mathcal{D}(\mathcal{Q}_\lambda)$

\begin{align*}
\langle\mathcal{Q}_\lambda u,v\rangle_{L^2([a,b],E)}&=\int^b_a{\langle Ju'(t),v(t)\rangle dt}=\langle Ju(b),v(b)\rangle-\langle Ju(a),v(a)\rangle+\int^b_a\langle u(t),Jv'(t)\rangle dt.
\end{align*}
The right hand side of this equation is equal to $\langle u,\mathcal{Q}_\lambda v\rangle_{L^2([a,b],E)}$ for all $u,v\in\mathcal{D}(\mathcal{Q}_\lambda)$ if and only if

\[\langle Jx,y\rangle=\langle J\tilde{x},\tilde{y}\rangle=0,\text{ for all } x,y\in\Lambda_0(\lambda),\, \tilde{x},\tilde{y}\in\Lambda_1(\lambda),\]
which means that $J\Lambda_1(\lambda)=\Lambda_1(\lambda)^\perp$ and $J\Lambda_0(\lambda)=\Lambda_0(\lambda)^\perp$. Hence, by \eqref{Lperp=JL}, $\mathcal{Q}_\lambda$ is symmetric if and only if $\Lambda_0(\lambda)$ and $\Lambda_1(\lambda)$ are Lagrangian.\\
Let us point out that we now have already shown that $(\Lambda_0(\lambda),\Lambda_1(\lambda))\in\mathcal{FL}^2(E,\omega)$ if $\mathcal{Q}_\lambda\in\mathcal{CF}^\textup{sa}(L^2([a,b],E))$.\\
As next step, let us briefly explain why $\mathcal{Q}_\lambda$ is closed if $\Lambda_0(\lambda),\Lambda_1(\lambda)\in\mathcal{G}(E)$. We assume that $\{u_n\}_{n\in\mathbb{N}}\subset\mathcal{D}(\mathcal{Q}_\lambda)$ is a sequence and $u,v\in L^2([a,b],E)$ are such that $u_n\rightarrow u$ and $\mathcal{Q}_\lambda u_n\rightarrow v$. We need to show that $u\in\mathcal{D}(\mathcal{Q}_\lambda)$ and $\mathcal{Q}_\lambda u=v$. Using that

\begin{align}\label{urep}
u_n(t)=u_n(a)+\int^t_a{u'_n(s)\,ds},\quad t\in[a,b],\, n\in\mathbb{N},
\end{align} 
it is easily seen that

\[\|u_n(a)-u_m(a)\|\leq\|u_n-u_m\|_{L^2([a,b],E)}+\|u'_n-u'_m\|_{L^2([a,b],E)},\]
and so $\{u_n(a)\}_{n\in\mathbb{N}}$ is a Cauchy sequence in $E$ converging to some $x\in E$. We set

\begin{align}\label{w}
w(t)=x-\int^t_a{Jv(s)\, ds},\quad t\in[a,b],
\end{align}
and see from \eqref{urep} that 

\begin{align}\label{uniforminequ}
\|u_n(t)-w(t)\|\leq\|u_n(a)-x\|+\|u'_n-v\|_{L^2([a,b],E)},\, t\in[a,b],
\end{align}
which clearly shows that $\{u_n\}_{n\in\mathbb{N}}$ converges to $w$ in $L^2([a,b],E)$, and so $u$ is given by \eqref{w}.\\
Now we note at first that, by \eqref{uniforminequ}, $u(a)\in\Lambda_0(\lambda)$ and $u(b)\in\Lambda_1(\lambda)$ as $u_n(a)\in\Lambda_0(\lambda)$, $u_n(b)\in\Lambda_1(\lambda)$, $n\in\mathbb{N}$, and these spaces are closed. This shows that $u\in\mathcal{D}(\mathcal{Q}_\lambda)$. Secondly, we see from \eqref{w} that $\mathcal{Q}_\lambda u=v$. Hence the operator $\mathcal{Q}_\lambda$ is closed, i.e. $\mathcal{Q}_\lambda\in\mathcal{C}(L^2([a,b],E)))$.\\  
Finally, we just need to recall that we computed the kernel and range of $\mathcal{Q}_\lambda$ in the first step of the proof, which now shows that $\mathcal{Q}_\lambda$ is Fredholm if $(\Lambda_0(\lambda),\Lambda_1(\lambda))$ is a Fredholm pair. Moreover, if $(\Lambda_0(\lambda),\Lambda_1(\lambda))\in\mathcal{FL}^2(E,\omega)$, then, by \eqref{ind=0Lagrangian},

\begin{align*}
\ind(\mathcal{Q}_\lambda)=\dim(\Lambda_0(\lambda)\cap\Lambda_1(\lambda))-\codim(\Lambda_0(\lambda)+\Lambda_1(\lambda))=0.
\end{align*}
Therefore, $\mathcal{Q}_\lambda\in\mathcal{CF}^\textup{sa}(L^2([a,b],E))$ by Lemma \ref{symmetric-selfadjoint}.
\end{proof}
\noindent
By the previous lemma, the operators $\mathcal{Q}_\lambda$ are in $\mathcal{CF}^\textup{sa}(L^2([a,b],E))$. Our Theorem B now shows that these operators actually define a path in $\mathcal{CF}^\textup{sa}(L^2([a,b],E))$ whose spectral flow can be computed by the Maslov index. Let us recall from the introduction that we call a path $\{(\Lambda_0(\lambda),\Lambda_1(\lambda))\}_{\lambda\in I}$ in $\mathcal{FL}^2(E,\omega)$ admissible if $\Lambda_0(0)\cap\Lambda_1(0)=\Lambda_0(1)\cap\Lambda_1(1)=\{0\}$.

\begin{theoremB}\label{main}
Let $\{(\Lambda_0(\lambda),\Lambda_1(\lambda))\}_{\lambda\in I}$ be an admissible path in $\mathcal{FL}^2(E,\omega)$. Then the path $\mathcal{Q}$ in \eqref{Q} is continuous in $\mathcal{CF}^\textup{sa}(L^2([a,b],E))$ and

\[\sfl(\mathcal{Q})=\mu_{Mas}(\Lambda_0(\cdot),\Lambda_1(\cdot)).\]    
\end{theoremB}
\noindent
Let us point out that the proof of the first assertion in the previous theorem does not make use of the assumption that the path in $\mathcal{FL}^2(E,\omega)$ is admissible. Finally, let us mention once again that Theorem B generalises Theorem 0.4 of \cite{Cappell} for admissible paths from $\mathbb{R}^{2n}$ to Hilbert spaces $E$.

%%%%%%%%%%%%%%%%%%%%%%%%%%%%%%%%%%%%%%%%%%%%%%%%%%%%%%%%%%%%%%%%%%%%%%%%%%%%%%%%%%%%%%%%%%%%%%%%%%%%%%%%%%%%%%%%%%%%%%%%%%%%%%%%%%%%%%%%%%%%%%%%%%%%%%%%%%%%%%%%%%%%%%%%%%%%%%%%%%%%%%%%%%%%%%%%%%%%%%%%%%%%%%%%%%%%%%%%%%%%%%%%%%%%%%%%%%%%%%%%%%%%%%%%%%%%%%%%%%%%%%%%%%%%%%%%%%%%%%%%%%%%%%%%%%%%%%%%%%%%%%%%%%%%%%%%%%%%%%%%%%%%%%%%%%%%%%%%%%%%%%

\subsection{Proof of Theorem B}
We divide the proof into two parts and show at first that the path $\mathcal{Q}=\{\mathcal{Q}_\lambda\}_{\lambda\in I}$ is gap-continuous. In the second step, we prove the equality of the spectral flow and the Maslov index. Throughout the proof, we abbreviate $L^2([a,b],E)$ by $L^2$ and $H^1([a,b],E)$ by $H^1$. Moreover, we simplify the presentation of the proof by setting $a=0$ and $b=1$, which clearly is no loss of generality.

\subsubsection*{Step 1: $\mathcal{Q}$ is Gap-Continuous}
We summarise at first some facts about the gap-metric $d_G$ from \cite{Kato}, which we introduced in \eqref{gap-subspaces}. An exhaustive treatment can also be found in \cite[\S 2.4]{thesis}. Let $M,N\subset E$ be two closed subspaces such that $M,N\neq\{0\}$. We set

\[\delta(M,N)=\sup_{u\in S_M}{d(u,N)},\]
where $S_M$ denotes the unit sphere in $M$ and $d(u,N)=\inf_{v\in N}\|u-v\|$. By \cite[p.198]{Kato},

\begin{align}\label{dG}
d_G(M,N)=\|P_M-P_N\|=\max\{\delta(M,N),\delta(N,M)\},
\end{align}
where $P_M$ and $P_N$ denote the orthogonal projections onto $M$ and $N$, respectively.\\
Let us now consider $d_G(\mathcal{Q}_\lambda,\mathcal{Q}_{\lambda_0})$ for some $\lambda,\lambda_0\in I$ (see \eqref{gap operators}). Following \eqref{dG}, we focus on 

\[\delta(\gra(\mathcal{Q}_\lambda),\gra(\mathcal{Q}_{\lambda_0}))\]
and note at first that for $u\in\mathcal{D}(\mathcal{Q}_\lambda)$ and $v\in\mathcal{D}(\mathcal{Q}_{\lambda_0})$

\begin{align}\label{gapcontI}
\begin{split}
\|(u,\mathcal{Q}_\lambda u)-(v,\mathcal{Q}_{\lambda_0}v)\|_{L^2\oplus L^2}&=\|(u-v,J(u'-v'))\|_{L^2\oplus L^2}\\
&\leq \left(\|u-v\|^2_{L^2}+\|J\|\|u'-v'\|^2_{L^2}\right)^\frac{1}{2}\\
&=\|u-v\|_{H^1},
\end{split}
\end{align}
where we have used that $\|J\|=1$. As we assume the continuity of the pair $\{(\Lambda_0(\lambda),\Lambda_1(\lambda))\}_{\lambda\in I}$ in $\mathcal{FL}^2(E,\omega)$, there are two families of orthogonal projections $\hat{P},\tilde{P}:I\rightarrow\mathcal{L}(E)$ such that 

\[\im(\hat{P}_\lambda)=\Lambda_0(\lambda),\quad \im(\tilde{P}_\lambda)=\Lambda_1(\lambda),\quad\lambda\in I.\]
We now set for $w\in H^1$

\[(P_\lambda w)(t)=w(t)-(1-t)(I_E-\hat{P}_\lambda)w(0)-t(I_E-\tilde{P}_\lambda)w(1),\]
and note that it is readily seen that $P^2_\lambda w=P_\lambda w$, and $P_\lambda w\in\mathcal{D}(\mathcal{Q}_\lambda)$ for all $w\in H^1$ and $\lambda\in I$. Hence

\begin{align}\label{gapcontII}
\inf_{v\in\mathcal{D}(\mathcal{Q}_{\lambda_0})}\|u-v\|_{H^1}\leq\|u-P_{\lambda_0}u\|_{H^1}.
\end{align}
As 

\[u(t)-(P_{\lambda_0}u)(t)=(1-t)(I_E-\hat{P}_{\lambda_0})(u(0))+t(I_E-\tilde{P}_{\lambda_0})(u(1)),\]
we obtain for $u\in\mathcal{D}(\mathcal{Q}_\lambda)$

\begin{align}\label{gapcontIII}
\begin{split}
\|u-P_{\lambda_0}u\|_{H^1}&\leq 2(\|(I_E-\hat{P}_{\lambda_0})(u(0))\|+\|(I_E-\tilde{P}_{\lambda_0})(u(1))\|)\\
&= 2(\|(I_E-\hat{P}_{\lambda_0})\hat{P}_\lambda(u(0))\|+\|(I_E-\tilde{P}_{\lambda_0})\tilde{P}_\lambda(u(1))\|)\\
&\leq 2(\|(I_E-\hat{P}_{\lambda_0})\hat{P}_\lambda\|\|u(0)\|+\|(I_E-\tilde{P}_{\lambda_0})\tilde{P}_\lambda\|\|u(1)\|),
\end{split}
\end{align}
where we have used that $\hat{P}_\lambda(u(0))=u(0)$ and $\tilde{P}_\lambda(u(1))=u(1)$ as $u\in\mathcal{D}(\mathcal{Q}_\lambda)$. Since the point evaluation is continuous in $H^1$, there exists a constant $c>0$ such that for $t=0$ and $t=1$

\begin{align}\label{gapcontIV}
\|u(t)\|\leq c\|u\|_{H^1}=c\left(\|u\|^2_{L^2}+\|u'\|^2_{L^2}\right)^\frac{1}{2}=c\left(\|u\|^2_{L^2}+\|Ju'\|^2_{L^2}\right)^\frac{1}{2},
\end{align}
where we use that $\|Jx\|=\|x\|$ for every $x\in E$. Hence, by \eqref{gapcontI}--\eqref{gapcontIV},

\begin{align*}
d((u,\mathcal{Q}_\lambda u),\gra(\mathcal{Q}_{\lambda_0}))&=\inf_{v\in\mathcal{D}(\mathcal{Q}_{\lambda_0})}\|(u,\mathcal{Q}_\lambda u)-(v,\mathcal{Q}_{\lambda_0}v)\|_{L^2\oplus L^2}\\
&\leq \inf_{v\in\mathcal{D}(\mathcal{Q}_{\lambda_0})}\|u-v\|_{H^1}\leq\|u-P_{\lambda_0}u\|_{H^1}\\
&\leq 2(\|(I_E-\hat{P}_{\lambda_0})\hat{P}_\lambda\|\|u(0)\|+\|(I_E-\tilde{P}_{\lambda_0})\tilde{P}_\lambda\|\|u(1)\|)\\
&\leq 2c(\|(I_E-\hat{P}_{\lambda_0})\hat{P}_\lambda\|+\|(I_E-\tilde{P}_{\lambda_0})\tilde{P}_\lambda\|)(\|u\|^2_{L^2}+\|Ju'\|^2_{L^2})^\frac{1}{2}. 
\end{align*}
As the unit sphere in $\gra(\mathcal{Q}_\lambda)$ is given by 

\[\{(u,\mathcal{Q}_\lambda u):\, u\in\mathcal{D}(\mathcal{Q}_\lambda),\, \|u\|^2_{L^2}+\|Ju'\|^2_{L^2}=1\},\]
we finally get

\begin{align}\label{gapfinalI}
\begin{split}
\delta(\gra(\mathcal{Q}_\lambda),\gra(\mathcal{Q}_{\lambda_0}))&=\sup_{{{u\in\mathcal{D}(\mathcal{Q}_\lambda)}\atop{\|u\|^2+\|Ju'\|^2=1}}}d((u,\mathcal{Q}_\lambda u),\gra(\mathcal{Q}_{\lambda_0}))\\
&\leq 2c(\|(I_E-\hat{P}_{\lambda_0})\hat{P}_\lambda\|+\|(I_E-\tilde{P}_{\lambda_0})\tilde{P}_\lambda\|).
\end{split}
\end{align}
Note that if we swap $\lambda$ and $\lambda_0$ and repeat the above argument, we likewise have

\begin{align}\label{gapfinalII}
\delta(\gra(\mathcal{Q}_{\lambda_0}),\gra(\mathcal{Q}_{\lambda}))\leq 2c(\|(I_E-\hat{P}_{\lambda})\hat{P}_{\lambda_0}\|+\|(I_E-\tilde{P}_{\lambda})\tilde{P}_{\lambda_0}\|).
\end{align}
By \cite[I.6.34]{Kato}, given two orthogonal projections $P,Q$ in $E$, if 

\[\|(I_E-P)Q\|<1\,\,\text{and  } \|(I_E-Q)P\|<1,\]
then 

\[\|(I_E-P)Q\|=\|(I_E-Q)P\|=\|P-Q\|.\]
Hence, as $(I_E-\hat{P}_{\lambda})\hat{P}_{\lambda_0}=(I_E-\tilde{P}_{\lambda})\tilde{P}_{\lambda_0}=0$ for $\lambda=\lambda_0$, we have for all $\lambda$ in a neighbourhood of $\lambda_0$

\[\|(I_E-\hat{P}_{\lambda})\hat{P}_{\lambda_0}\|=\|(I_E-\hat{P}_{\lambda_0})\hat{P}_{\lambda}\|=\|\hat{P}_{\lambda}-\hat{P}_{\lambda_0}\|\]
and
\[\|(I_E-\tilde{P}_{\lambda})\tilde{P}_{\lambda_0}\|=\|(I_E-\tilde{P}_{\lambda_0})\tilde{P}_{\lambda}\|=\|\tilde{P}_\lambda-\tilde{P}_{\lambda_0}\|.\]
Consequently, we obtain from \eqref{dG}, \eqref{gapfinalI} and \eqref{gapfinalII} for all $\lambda$ sufficiently close to $\lambda_0$

\begin{align*}
d_G(\mathcal{Q}_\lambda,\mathcal{Q}_{\lambda_0})&=\max\{\delta(\gra(\mathcal{Q}_\lambda),\gra(\mathcal{Q}_{\lambda_0})),\delta(\gra(\mathcal{Q}_{\lambda_0}),\gra(\mathcal{Q}_{\lambda}))\}\\
&\leq 2c(\|\hat{P}_\lambda-\hat{P}_{\lambda_0}\|+\|\tilde{P}_\lambda-\tilde{P}_{\lambda_0}\|),
\end{align*}
which shows that $\mathcal{Q}=\{\mathcal{Q}_\lambda\}_{\lambda\in I}$ is indeed continuous in $\mathcal{CF}(L^2([a,b],E))$.

%%%%%%%%%%%%%%%%%%%%%%%%%%%%%%%%%%%%%%%%%%%%%%%%%%%%%%%%%%%%%%%%%%%%%%%%%%%%%%%%%%%%%%%%%%%%%%%%%%%%%%%%%%%%%%%%%%%%%%%%%%%%%%%%%%%%%%%%%%%%%%%%%%%%%%%%%%%%%%%%%%%%%%%%%%%%%%%%%%%%%%%%%%%%%%%%%%%%%%%%%%%%%%%%%%%%%%%%%%%%%%%%%%%%%%%%%%%%%%%%%%%%%%%%%%%%%%%%%%%%%%%%%%%%%%%%%%%%%%%%%%%%%%%%%%%%%%%%%%%%%%%%%%%%%%%%%%%%%%%%%%%%%%%%%%%%%%%%%%%%%%

\subsubsection*{Step 2: The Spectral Flow Formula}
We assume in this step of our proof that $E$ is of infinite dimension so that we can apply Theorem \ref{Furutani-iso}. As the finite dimensional case was shown by Cappell, Lee and Miller in \cite{Cappell}, this is no restriction of the generality. Actually, a simple modification of the below argument shows the assertion in finite dimensions, and the reader is invited to work out the details.\\   
In what follows, we let $\mathcal{FL}^2_0(E,\omega)$ be the set of all pairs $(\Lambda_0,\Lambda_1)\in\mathcal{FL}^2(E,\omega)$ which are transversal, i.e. $\Lambda_0\cap\Lambda_1=\{0\}$. 

\begin{lemma}\label{lema-FL0pathconnected}
The set $\mathcal{FL}^2_0(E,\omega)$ of transversal pairs is path-connected in $\mathcal{FL}^2(E,\omega)$. 
\end{lemma}

\begin{proof}
We note at first that by \cite[Rem. 1.33]{Furutani}, for every $W\in\Lambda(E,\omega)$ the set 

\[\mathcal{FL}^0_W(E,\omega)=\{L\in\mathcal{FL}_W(E,\omega):\, \dim(L\cap W)=0\}\subset\mathcal{FL}_W(E,\omega)\]
is homeomorphic to the space of bounded selfadjoint operators on a Hilbert space. Hence $\mathcal{FL}^0_W(E,\omega)$ is contractible and so in particular path-connected.\\
Let $(\Lambda_1,\Lambda_2)$ and $(\Lambda_3,\Lambda_4)$ be two transversal pairs. By \cite{Piccione}, there is $\Lambda'_1\in\Lambda(E,\omega)$ such that $\Lambda'_1\oplus\Lambda_2=\Lambda'_1\oplus\Lambda_4=E$. In particular, we obtain a path connecting $(\Lambda_1,\Lambda_2)$ and $(\Lambda'_1,\Lambda_2)$ inside $\mathcal{FL}^0_{\Lambda_2}(E,\omega)\subset\mathcal{FL}^2_0(E,\omega)$. Also, as $\mathcal{FL}^0_{\Lambda'_1}(E,\omega)$ is path-connected, there is a path connecting $(\Lambda'_1,\Lambda_2)$ and $(\Lambda'_1,\Lambda_4)$ inside $\mathcal{FL}^2_0(E,\omega)$. Finally, there is a path from $(\Lambda'_1,\Lambda_4)$ to $(\Lambda_3,\Lambda_4)$ inside $\mathcal{FL}^2_0(E,\omega)$ as $\mathcal{FL}^0_{\Lambda_4}(E,\omega)$ is path-connected.   
\end{proof}
\noindent
Our proof of the spectral flow formula in Theorem B is based on the following proposition. In its statement and proof we simplify our notation by denoting paths $\{(\Lambda_0(\lambda),\Lambda_1(\lambda))\}_{\lambda\in I}$ in $\mathcal{FL}^2(E,\omega)$ by $\gamma$.

\begin{prop}\label{uniqueness}
Let $\mu:\Omega(\mathcal{FL}^2(E,\omega),\mathcal{FL}^2_0(E,\omega))\rightarrow\mathbb{Z}$ be a map such that

\begin{itemize}
\item[(i)] $\mu(\gamma)=0$ if $\gamma(\lambda)\in\mathcal{FL}^2_0(E,\omega)$ for all $\lambda\in I$.
\item[(ii)] $\mu(\gamma_1\ast\gamma_2)=\mu(\gamma_1)+\mu(\gamma_2)$ whenever the concatenation $\gamma_1\ast\gamma_2$ is defined.
\item[(iii)] $\mu(\gamma_1)=\mu(\gamma_2)$ if the paths $\gamma_1$ and $\gamma_2$ are homotopic with fixed endpoints.
\item[(iv)] There is a path $\gamma_{nor}$ such that $\mu(\gamma_{nor})=\mu_{Mas}(\gamma_{nor})=1$.
\end{itemize}
Then \[
\mu=\mu_{Mas}:\Omega(\mathcal{FL}^2(E,\omega),\mathcal{FL}^2_0(E,\omega))\rightarrow\mathbb{Z}.\]
\end{prop}
\noindent
Note that the Maslov index indeed has the properties (i) - (iii) in Proposition \ref{uniqueness} as we have recalled in Section \ref{subsection-Maslov}, (i') - (iii').

\begin{proof}
Let $p_0=(\Lambda_0,\Lambda'_0)\in\mathcal{FL}^2_0(E,\omega)$ be any transversal pair. Note that, by (ii) and (iii), $\mu$ induces a homomorphism $\pi_1(\mathcal{FL}^2(E,\omega),p_0)\rightarrow\mathbb{Z}$. Moreover, let us recall from Theorem \ref{Furutani-iso} that $\mu_{Mas}:\pi_1(\mathcal{FL}^2(E,\omega),p_0)\rightarrow\mathbb{Z}$ is an isomorphism.\\
By Lemma \ref{lema-FL0pathconnected}, we can connect $p_0$ to $\gamma_{nor}(0)$ and $\gamma_{nor}(1)$ to $p_0$ by paths $\gamma_1,\gamma_2$ in $\mathcal{FL}^2_0(E,\omega)$. Then the concatenation $\gamma_1\ast\gamma_{nor}\ast\gamma_2$ is a closed path based at $p_0$ and consequently defines an element of the infinite cyclic group $\pi_1(\mathcal{FL}^2(E,\omega),p_0)$. We obtain by $(i)$, $(ii)$ and $(iv)$

\begin{align*}
\mu_{Mas}(\gamma_1\ast\gamma_{nor}\ast\gamma_2)&=\mu_{Mas}(\gamma_1)+\mu_{Mas}(\gamma_{nor})+\mu_{Mas}(\gamma_2)=\mu_{Mas}(\gamma_{nor})=1
\end{align*}
and see that $\gamma_1\ast\gamma_{nor}\ast\gamma_2$ is a generator of $\pi_1(\mathcal{FL}^2(E,\omega),p_0)$. As, by the same argument, $\mu(\gamma_1\ast\gamma_{nor}\ast\gamma_2)=1$, we have shown that $\mu=\mu_{Mas}:\pi_1(\mathcal{FL}^2(E,\omega),p_0)\rightarrow\mathbb{Z}$.\\
We now consider the general case of an arbitrary element $\gamma\in\Omega(\mathcal{FL}^2(E,\omega),\mathcal{FL}^2_0(E,\omega))$. As the endpoints of $\gamma$ are in $\mathcal{FL}^2_0(E,\omega)$, we can again find two paths $\gamma_1$, $\gamma_2$ such that $\gamma_1$ connects $p_0$ and $\gamma(0)$ inside $\mathcal{FL}^2_0(E,\omega)$, and $\gamma_2$ connects $\gamma(1)$ and $p_0$ inside $\mathcal{FL}^2_0(E,\omega)$. Then $\gamma_1\ast\gamma\ast\gamma_2$ defines an element of $\pi_1(\mathcal{FL}^2(E,\omega),p_0)$ and we obtain from the previous paragraph, (i) and (ii)

\begin{align*}
\mu(\gamma)&=\mu(\gamma_1)+\mu(\gamma)+\mu(\gamma_2)=\mu(\gamma_1\ast\gamma\ast\gamma_2)=\mu_{Mas}(\gamma_1\ast\gamma\ast\gamma_2)\\
&=\mu_{Mas}(\gamma_1)+\mu_{Mas}(\gamma)+\mu_{Mas}(\gamma_2)=\mu_{Mas}(\gamma).
\end{align*}  
\end{proof}
\noindent

\begin{rem}
Let us point out that a similar argument as in the proof of Proposition \ref{uniqueness} can be used to characterise the Maslov index axiomatically by just three axioms. Indeed, it follows from the construction in \cite{Furutani} that the Maslov index is invariant under homotopies having endpoints in $\mathcal{FL}^2_0(E,\omega)$. It is readily seen that $(i)$ is redundant when using this homotopy invariance in $(iii)$.   
\end{rem}
\noindent
We now define a map 

\[\mu:\Omega(\mathcal{FL}^2(E,\omega),\mathcal{FL}^2_0(E,\omega))\rightarrow\mathbb{Z},\quad \mu(\{(\Lambda_0(\lambda),\Lambda_1(\lambda))\}_{\lambda\in I})=\sfl(\mathcal{Q}),\]
where $\mathcal{Q}$ is the path of differential operators in \eqref{Q}. By the properties of the spectral flow from Section \ref{section-sfl}, we see at once that $\mu$ satisfies (ii) and (iii) in Proposition \ref{uniqueness}. Moreover, if $(\Lambda_0(\lambda),\Lambda_1(\lambda))\in\mathcal{FL}^2_0(E,\omega)$ for all $\lambda\in I$, then

\[\dim\ker(\mathcal{Q}_\lambda)=\dim(\Lambda_0(\lambda)\cap\Lambda_1(\lambda))=0,\quad\lambda\in I.\]
Therefore, $\mathcal{Q}_\lambda$ is invertible for all $\lambda\in I$, which shows that $\mu(\{(\Lambda_0(\lambda),\Lambda_1(\lambda))\}_{\lambda\in I})=\sfl(\mathcal{Q})=0$.\\
Hence in order to deduce Theorem B from Proposition \ref{uniqueness} we only need to find a path $\gamma_{nor}$ in $\mathcal{FL}^2(E,\omega)$ with transversal endpoints for which $\mu_{Mas}(\gamma_{nor})=\mu(\gamma_{nor})=1$. This will now keep us busy for the remainder of this section.\\
At first, to obtain a path of Lagrangian subspaces having Maslov index $1$, we argue similar as in Example 2.31 in \cite{Furutani}. Let $\Lambda$ be an arbitrary Lagrangian subspace of $(E,\omega)$, and $V\subset\Lambda^\perp$ a one-dimensional subspace. We consider the bounded linear operators $\Psi_\lambda$ defined by

\begin{align}\label{psi}
\Psi_\lambda u=\begin{cases}
(\cos(\pi \lambda)I_E+\sin(\pi \lambda)J)u,&u\in V\\
u, &u\in V^\perp\cap\Lambda^\perp,
\end{cases}
\end{align}
and set $\Lambda(\lambda)=\Psi_\lambda\left(\Lambda^\perp\right)$ for $\lambda\in I$. It is readily seen that $(\Lambda(\lambda),\Lambda)\in\mathcal{FL}_\Lambda(E,\omega)$. Hence $\gamma_{nor}(\lambda):=(\Lambda(\lambda),\Lambda)\in\mathcal{FL}^2(E,\omega)$ for all $\lambda\in I$. Note that $\Lambda(\lambda)\cap\Lambda\neq\{0\}$ if and only if $\lambda=\frac{1}{2}$, where the intersection is the one-dimensional space $J(V)\subset\Lambda$.\\
Furutani explained in \cite[\S 3.4]{Furutani} how the Maslov index for paths $(\Lambda(\lambda),\Lambda)\in\mathcal{FL}^2(E,\omega)$ can be computed in cases when there is only one single parameter value $\lambda_0$ for which $\Lambda(\lambda_0)\cap\Lambda\neq\{0\}$. For all $\lambda$ sufficiently close to $\lambda_0$, there is a bounded linear operator $A_\lambda:\Lambda(\lambda_0)\rightarrow\Lambda(\lambda_0)$ such that 

\[\Lambda(\lambda)=\{u+JA_\lambda u:\,u\in\Lambda(\lambda_0)\}.\]
If $\Lambda(\cdot)$ is differentiable in $\Lambda(E,\omega)$, then $A$ is a differentiable path in $\mathcal{L}(\Lambda(\lambda_0))$. The Maslov index of $\{(\Lambda(\lambda),\Lambda)\}_{\lambda\in I}$ is given by the signature of the quadratic form

\begin{align}\label{crossing}
\Gamma:\Lambda(\lambda_0)\cap\Lambda\rightarrow\mathbb{R},\quad \Gamma[u]=\frac{d}{d\lambda}\mid_{\lambda=\lambda_0}\omega(u,J A_\lambda u).
\end{align}
For our particular path $\Lambda(\cdot)$ defined by \eqref{psi}, we see that

\[\Lambda(\lambda_0)=(V^\perp\cap\Lambda^\perp)\oplus J(V),\]
where $\lambda_0=\frac{1}{2}$. Let us now consider for $\lambda\in(0,1)$ the operators $A_\lambda$ on $\Lambda(\lambda_0)$ which are defined by 

\begin{align*}
A_\lambda u=\begin{cases}
-\frac{\cos(\pi\lambda)}{\sin(\pi\lambda)}\,u, &u\in J(V)\\
0, &u\in V^\perp\cap\Lambda^\perp.
\end{cases}
\end{align*}
We claim that 

\begin{align}\label{graphA}
\Lambda(\lambda)=\{u+JA_\lambda u:\,u\in\Lambda(\lambda_0)\}.
\end{align}
Indeed, we note at first that trivially $V^\perp\cap\Lambda^\perp=\{u+JA_\lambda u:\, u\in V^\perp\cap\Lambda^\perp\}$. Moreover, as $\lambda\in(0,1)$, we see that $u\in J(V)$ if and only if there is $v\in V$ such that $u=\sin(\pi\lambda)Jv$. Since 

\begin{align*}
u+JA_\lambda u=\sin(\pi\lambda)Jv-J^2\cos(\pi\lambda)v=\cos(\pi\lambda)v+\sin(\pi\lambda)Jv,
\end{align*}
this shows \eqref{graphA}. The crossing form \eqref{crossing} is defined on $\Lambda(\lambda_0)\cap\Lambda$, which is the one-dimensional space $J(V)$ on which $A_\lambda$ acts by 

\[A_\lambda u=-\frac{\cos(\pi\lambda)}{\sin(\pi\lambda)}u.\]
Hence

\[\frac{d}{d\lambda}\omega(u,JA_\lambda u)=\frac{d}{d\lambda}\langle u,A_\lambda u\rangle =\frac{\pi}{\sin^2(\pi\lambda)} \langle u,u\rangle,\]
and so $\Gamma[u]=\pi \langle u,u\rangle $, which has signature $1$. Consequently, we have shown that 

\[\mu_{Mas}(\gamma_{nor})=1.\]
Let us now consider our other invariant $\mu$, which is by definition the spectral flow of the path of operators $\mathcal{Q}$ associated to $\gamma_{nor}$ as in \eqref{Q}. Let us recall that the kernel of $\mathcal{Q}_\lambda$ is the intersection $\Lambda(\lambda)\cap\Lambda$. As $\mathcal{Q}_\lambda$ is Fredholm of index $0$ for all $\lambda\in I$, this means that $\mathcal{Q}_{\lambda}$ is not invertible if and only if $\lambda=\lambda_0=\frac{1}{2}$, and $\mathcal{Q}_{\lambda_0}$ has a one-dimensional kernel. Therefore, the spectral flow of $\mathcal{Q}$ can only be $-1$, $0$ or $1$. The final step of our proof is the following lemma about the eigenvalues of $\mathcal{Q}_\lambda$. Let us point out once again that we denote by $\lambda$ our parameter in $I$ and not elements of spectra of operators.

\begin{lemma}\label{eigenvaluesQ}
The eigenvalues of $\mathcal{Q}_\lambda$ are given by

\[\sigma_p(\mathcal{Q}_\lambda)=\left\{\pi \lambda-\frac{\pi}{2}+\pi k:\,k\in\mathbb{Z}\right\}\cup\left\{\frac{\pi}{2}+k\pi:\, k\in\mathbb{Z}\right\}.\]
Moreover, the elements of the first set on the right hand side are simple eigenvalues as long as $\lambda\in(0,1)$.
\end{lemma} 

\begin{proof} 
We first note that, for $\mu\in\mathbb{R}$, all solutions of $Ju'=\mu u$ are given by

\[u(t)=\exp(-\mu tJ)c,\, t\in[0,1],\,c\in E.\]
Such a function belongs to $\mathcal{D}(\mathcal{Q}_\lambda)$ if and only if

\[u(0)\in\Lambda(\lambda),\qquad u(1)=\exp(-\mu J)u(0)\in\Lambda,\]
or, in other words,

\begin{align*}
\exp(\mu J)(\Lambda)\cap\Lambda(\lambda)\neq\{0\}.
\end{align*}
It is readily seen from $J^2=-I_E$ that

\begin{align}\label{EulerJ}
\exp(\mu J)=\cos(\mu)I_E+\sin(\mu)J.
\end{align}
If we compare \eqref{EulerJ} and \eqref{psi}, we see that $\mu$ is an eigenvalue of $\mathcal{Q}_\lambda$ if and only if either

\begin{itemize}
\item[(i)] $(\cos(\mu)I_E+\sin(\mu)J)\Lambda\cap(\Lambda^\perp\cap V^\perp)\neq\{0\}$, or
\item[(ii)] $(\cos(\mu)I_E+\sin(\mu)J)\Lambda\cap (\cos(\pi \lambda)I_E+\sin(\pi \lambda)J)V\neq\{0\}$.
\end{itemize}
From \eqref{Lperp=JL}, we see that the first case happens if and only if $\cos(\mu)=0$, i.e. $\lambda=\frac{\pi}{2}+k\pi$ for $k\in\mathbb{Z}$. Note that the intersections are of infinite dimension and so we have eigenvalues of infinite multiplicity. This reflects the fact that $\mathcal{Q}_\lambda$ does not have a compact resolvent.\\
In the second case, there are non-zero elements $x\in V\subset\Lambda^\perp$ and $y\in\Lambda$ such that

\[\begin{cases}
\sin(\pi \lambda)Jx=\cos(\mu)y\\
\cos(\pi \lambda)x=\sin(\mu)Jy.
\end{cases}\]
These equations can hold if and only if there is $\alpha\in\mathbb{R}$ such that $Jy=\alpha x$, which yields

\[\begin{cases}
\sin(\pi \lambda)=-\alpha\cos(\mu)\\
\cos(\pi \lambda)=\alpha\sin(\mu)
\end{cases}\]
or, in other words, $e^{-i(\pi \lambda-\frac{\pi}{2})}=-\alpha e^{-i\mu}$. This clearly implies that $|\alpha|=1$ and $\mu=\pi \lambda-\frac{\pi}{2}+k\pi$, $k\in\mathbb{Z}$. Note that these are eigenvalues of multiplicity one as $V$ is one-dimensional. 
\end{proof}
\noindent
It is now easy to obtain the spectral flow formula in Theorem B from the previous lemma. Indeed, we know from the construction of the spectral flow that there is $0<\varepsilon<\frac{\pi}{2}$ such that, for all $\lambda\in I$, $\mu\in(-\varepsilon,\varepsilon)$ is either in the resolvent set of $\mathcal{Q}_\lambda$ or it is an eigenvalue of finite multiplicity. Therefore, by the properties (i) and (ii) of the spectral flow from Section \ref{section-sfl} and Lemma \ref{eigenvaluesQ}, we only need to consider the eigenvalue 

\[\mu(\lambda)=\pi\lambda-\frac{\pi}{2}\]
for $\lambda$ in a neighbourhood of $\lambda_0=\frac{1}{2}$. It is a straightforward consequence of the definition of the spectral flow \eqref{sfl} that $\sfl(\mathcal{Q})=1$. Hence $\mu(\gamma_{nor})=1$ and Theorem B is proved.

\section{Theorem C}
In this section, we let $H$ be a complex Hilbert space unless otherwise stated. Let us point out that, however, most of our results carry over to real Hilbert spaces when $KO$-theory is used instead of $K$-theory.\\ 
The index bundle for families of bounded Fredholm operators in a Hilbert space was independently introduced by Atiyah and J\"anich in the sixties (see \cite{Atiyah} and \cite{Janich}). It assigns to any family $L:X\rightarrow\mathcal{BF}(H)$ of bounded Fredholm operators on $H$ a $K$-theory class

\begin{align}\label{AtiyahJanich}
\ind(L)\in K(X)
\end{align}
which has several properties that are similar to the Fredholm index of a single operator. One of the main applications of this generalisations to families is that it shows that the space of Fredholm operators on a separable Hilbert space is a classifying space for $K$-theory. The index bundle was later generalised to families of Fredholm operators in Banach spaces (see \cite{Russian}), and to morphisms between Banach bundles in \cite{indbundleIch}. Let us briefly recall the latter construction, as we will need it below in the definition of the index bundle for gap-continuous families.\\
Let $\mathcal{E}$ and $\mathcal{F}$ be Banach bundles over a compact and connected base space $X$ and let $L:\mathcal{E}\rightarrow\mathcal{F}$ be a bundle morphism which is Fredholm in every fibre. It was shown in \cite{indbundleIch} that there is a finite dimensional subbundle $\mathcal{V}\subset\mathcal{F}$ such that

\begin{align}\label{transversal}
\im(L_\lambda)+\mathcal{V}_\lambda=\mathcal{F}_\lambda,\quad\lambda\in X.
\end{align} 
As $\mathcal{V}$ is finite dimensional, there is a bundle morphism $P:\mathcal{F}\rightarrow\mathcal{F}$ such that $P^2=P$ and $\im(P_\lambda)=\mathcal{V}_\lambda$, $\lambda\in X$. Hence the composition

\[\mathcal{E}\xrightarrow{L}\mathcal{F}\xrightarrow{I_\mathcal{F}-P}\mathcal{V}'\]
is a surjective Banach bundle morphism onto $\mathcal{V}'=\im(I_\mathcal{F}-P)$. As the fibrewise kernels of surjective Banach bundle morphisms are Banach bundles (see \cite{Lang}), and $\ker((I_{\mathcal{F}_\lambda}- P_\lambda)\circ L_\lambda)=L^{-1}_\lambda(\mathcal{V}_\lambda)$, we obtain a subbundle 

\[E(L,\mathcal{V}):=L^{-1}(\mathcal{V})\subset\mathcal{E}.\]
It is readily seen that 

\begin{align}\label{dimE}
\dim(E(L,\mathcal{V}))=\ind(L_\lambda)+\dim(\mathcal{V}),\quad \lambda\in X,
\end{align}
where $\ind(L_\lambda)$ denotes the Fredholm index of the operator $L_\lambda$.\\ 
Let us now assume that $Y\subset X$ is a closed subset of $X$ such that $L_\lambda$ is invertible for all $\lambda\in Y$. Then $L$ induces a morphism $L:E(L,\mathcal{V})\rightarrow\mathcal{V}$ between finite dimensional vector bundles, which is an isomorphism over $Y$. Hence we obtain a $K$-theory class (see Appendix \ref{app-K})

\[\ind(L)=[E(L,\mathcal{V}),\mathcal{V},L]\in K(X,Y),\]
which we call the \textit{index bundle} of $L$. This definition is sensible as it can be shown that $\ind(L)$ does not depend on the choice of the bundle $\mathcal{V}$ in \eqref{transversal}. Note that we do not exclude the case that $Y=\emptyset$ in the definition of $\ind(L)$, however, if $Y\neq\emptyset$, then the operators $L_\lambda$ are necessarily of Fredholm index $0$ by \eqref{dimE}.\\
In the following list of properties of the index bundle, we assume throughout that $L:\mathcal{E}\rightarrow\mathcal{F}$ is a Fredholm morphism such that $L_\lambda$ is an isomorphism for every $\lambda\in Y$.

\begin{itemize}
\item If $L$ is a bundle isomorphism, then $\ind(L)=0\in K(X,Y)$.
\item If $M:\tilde{\mathcal{E}}\rightarrow\tilde{\mathcal{F}}$ is a further Fredholm morphisms such that $M_\lambda$ is invertible for all $\lambda\in Y$, then

\[\ind(L\oplus M)=\ind(L)+\ind(M)\in K(X,Y).\]
\item Let $\mathcal{E}$ and $\mathcal{F}$ be Banach bundles over $X\times I$ and $h:\mathcal{E}\rightarrow\mathcal{F}$ a Fredholm morphism. If $h_{(\lambda,s)}$ is invertible for all $(\lambda,s)\in Y\times I$, then

\[\ind(h\mid_{X\times\{0\}})=\ind(h\mid_{X\times\{1\}})\in K(X,Y).\]
\item Let $K:\mathcal{E}\rightarrow\mathcal{F}$ be a morphism which is compact in every fibre, and let $L_\lambda+sK_\lambda$ be invertible for all $\lambda\in Y$, $s\in I$. Then

\[\ind(L+K)=\ind(L)\in K(X,Y).\]
\end{itemize}
\noindent
We now assume that $(X',Y')$ is a further compact pair.

\begin{itemize}
\item If $f:(X',Y')\rightarrow(X,Y)$ is continuous, then 

\[(f^\ast L)_\lambda:=L_{f(\lambda)}:\mathcal{E}_{f(\lambda)}\rightarrow\mathcal{F}_{f(\lambda)}\]
defines a morphism $f^\ast L:f^\ast\mathcal{E}\rightarrow f^\ast\mathcal{F}$ which is Fredholm. Moreover,

\[\ind(f^\ast L)=f^\ast\ind(L)\in K(X',Y').\]
\end{itemize}
\noindent
Finally, let $\mathcal{G}$ be a further Banach bundle over $X$. The following rule is usually called the \textit{logarithmic property} of the index bundle.

\begin{itemize}
\item If $M:\mathcal{F}\rightarrow\mathcal{G}$ is a further Fredholm morphism which is invertible over $Y$, then

\[\ind(M\circ L)=\ind(M)+\ind(L)\in K(X,Y).\]
\end{itemize}

%%%%%%%%%%%%%%%%%%%%%%%%%%%%%%%%%%%%%%%%%%%%%%%%%%%%%%%%%%%%%%%%%%%%%%%%%%%%%%%%%%%%%%%%%%%%%%%%%%%%%%%%%%%%%%%%%%%%%%%%%%%%%%%%%%%%%%%%%%%%%%%%%%%%%%%%%%%%%%%%%%%%%%%%%%%%%%%%%%%%%%%%%%%%%%%%%%%%%%%%%%%%%%%%%%%%%%%%%%%%%%%%%%%%%%%%%%%%%%%%%%%%%%%%%%%%%%%%%%%%%%%%%%%%%%%%%%%%%%%%%%%%%%%%%%%%%%%%%%%%%%%%%%%%%%%%%%%%%%%%%%%%%%%%%%%%%%%%%%%%%%%%%%%%%%%%%%%%%%%%%%%%%%%%%%%%%%%%%%%%%%%%%%%%%%%%%%%%%%%%%%%%%%%%%%%%%%%%%%%%%%%%%%%%%%%%%%%%%%%%%%%%%%%%%%%%%

\subsection{The Index Bundle for Gap-Continuous Families}
Let us now consider a gap-continuous family $\mathcal{A}:X\rightarrow\mathcal{CF}(H)$ of Fredholm operators on $H$ which are parametrised by a compact space $X$. The aim of this section is to generalise the index bundle \eqref{AtiyahJanich} of Atiyah and J\"anich to this setting of unbounded Fredholm operators, where we follow \cite{thesis} (see also \cite{Mike}). Note that the domains $\mathcal{D}(\mathcal{A}_\lambda)$ are not constant and so the classical construction cannot be adapted straight away just by using graph norms. The key step of our approach is the construction of the \textit{domain bundle}, for which we want to recall at first the following well known theorem that can be found, e.g.,  in \cite[Thm. 3.2]{Steenrod}.

\begin{theorem}\label{bundleconstructiontheorem}
Let $p:\mathcal{E}\rightarrow X$ be a surjective map from some set $\mathcal{E}$ to a topological space
$X$, and let $\mathcal{J}$ be an index set. Let $\{U_j\}_{j\in\mathcal{J}}$ be an open cover of $X$, and suppose that we are given for each $U_j$ a Banach space $E_j$ and a bijection 

\[\varphi_j:p^{-1}(U_j)\rightarrow U_j\times E_j\]
such that $p=p_1\circ\varphi_j$ on $p^{-1}(U_j)$, where $p_1:U_j\times E_j\rightarrow U_j$ denotes the projection onto the first component. Moreover, we assume that, for each pair $U_i,U_j$ such that $U_i\cap U_j\neq\emptyset$, the map

\[U_i\cap U_j\rightarrow GL(E_j,E_i),\quad \lambda\mapsto(\varphi_i\circ\varphi^{-1}_j)_\lambda\]
is continuous with respect to the norm topology.\\
Then there exists a unique topology on $\mathcal{E}$ making it into the total space of a Banach bundle
with projection $p$ and trivialising covering $\{U_j\}_{j\in\mathcal{J}}$.
\end{theorem}
\noindent
If $\mathcal{A}:X\rightarrow\mathcal{CF}(H)$ is continuous with respect to the gap-topology on $\mathcal{CF}(H)$, then there is a family of projections $P:X\rightarrow\mathcal{L}(H\times H)$ such that $\im(P_\lambda)=\gra(\mathcal{A}_\lambda)$. We fix some $\lambda_0\in X$ and consider the open set 

\[U_{\lambda_0}:=\left\{\lambda\in X:d_G(\mathcal{A}_\lambda,\mathcal{A}_{\lambda_0})<\frac{1}{3}\right\}\subset X.\]
It can be shown by the Neumann series that

\[P_{\gra(\mathcal{A}_{\lambda_0})}\mid_{\gra(\mathcal{A}_\lambda)}:\gra(\mathcal{A}_\lambda)\rightarrow\gra(\mathcal{A}_{\lambda_0})\]
is an isomorphism for all $\lambda\in U_{\lambda_0}$, and that the map

\begin{align}\label{bundle1}
U_{\lambda_0}\ni \lambda\mapsto (P_{\gra(\mathcal{A}_{\lambda_0})}\mid_{\gra(\mathcal{A}_\lambda)})^{-1}\in\mathcal{L}(\gra(\mathcal{A}_{\lambda_0}),H\times H)
\end{align}
is continuous with respect to the norm topology on the latter space (see \cite{AbbondandoloMajerGrass} or \cite[\S 6.1]{thesis}). We now consider the disjoint union

\[\mathfrak{D}(\mathcal{A}):=\coprod_{\lambda\in X}\mathcal{D}(\mathcal{A}_\lambda)\]
and in what follows we denote by $\pi:\mathfrak{D}(\mathcal{A})\rightarrow X$ the canonical surjection. We define

\[\tau_{\lambda_0}:\pi^{-1}(U_{\lambda_0})\rightarrow U_{\lambda_0}\times\gra(\mathcal{A}_{\lambda_0}),\quad \tau_{\lambda_0}(\lambda,u)=(\lambda,P_{\gra(\mathcal{A}_{\lambda_0})}(u,\mathcal{A}_{\lambda}u)),\]
and note that $\gra(\mathcal{A}_{\lambda_0})\subset H\times H$ is a Banach space as $\mathcal{A}_{\lambda_0}$ is a closed operator. It is readily seen that this map has an inverse $\tau^{-1}_{\lambda_0}:U_{\lambda_0}\times\gra(\mathcal{A}_{\lambda_0})\rightarrow\pi^{-1}(U_{\lambda_0})$ given by

\[\tau^{-1}_{\lambda_0}(\lambda,u)=P_1(P_{\gra(\mathcal{A}_{\lambda_0})}\mid_{\gra(\mathcal{A}_\lambda)})^{-1},\]
where $P_1$ denotes the projection onto $H\times\{0\}$ in $H\times H$. Moreover, if $\lambda_1\in X$ is a point in our parameter space such that $U_{\lambda_0}\cap U_{\lambda_1}\neq\emptyset$, then

\[(\tau_{\lambda_1}\circ\tau^{-1}_{\lambda_0})_\lambda=P_{\gra(\mathcal{A}_{\lambda_1})}(P_{\gra(\mathcal{A}_{\lambda_0})}\mid_{\gra(\mathcal{A}_\lambda)})^{-1}\in \mathcal{L}(\gra(\mathcal{A}_{\lambda_0}),\gra(\mathcal{A}_{\lambda_1})),\] 
and, by \eqref{bundle1}, this depends continuously on $\lambda$. As the maps

\begin{align*}
P_{\gra(\mathcal{A}_{\lambda_0})}\mid_{\gra(\mathcal{A}_\lambda)}&:\gra(\mathcal{A}_\lambda)\rightarrow\gra(\mathcal{A}_{\lambda_0}),\\
P_{\gra(\mathcal{A}_{\lambda_1})}\mid_{\gra(\mathcal{A}_\lambda)}&:\gra(\mathcal{A}_\lambda)\rightarrow\gra(\mathcal{A}_{\lambda_1})
\end{align*} 
are isomorphisms, we finally see that 

\[(\tau_{\lambda_1}\circ\tau^{-1}_{\lambda_0})_\lambda\in GL(\gra(\mathcal{A}_{\lambda_0}),\gra(\mathcal{A}_{\lambda_1})),\quad \lambda\in U_{\lambda_0}\cap U_{\lambda_1}.\]
Hence, by Theorem \ref{bundleconstructiontheorem}, $\pi:\mathfrak{D}(\mathcal{A})\rightarrow X$ is a Hilbert bundle, which we we call the \textit{domain bundle} of the family $\mathcal{A}$.\\
Note that, by the definition of our trivialisations, the map $P_1\mid_{\gra(\mathcal{A}_\lambda)}:\gra(\mathcal{A}_\lambda)\rightarrow\mathfrak{D}(\mathcal{A})_\lambda$ is a topological isomorphism, where $\mathfrak{D}(\mathcal{A})_\lambda$ denotes the fibre of $\mathfrak{D}(\mathcal{A})$ over $\lambda\in X$. Hence the topology of $\mathfrak{D}(\mathcal{A})_\lambda$ is the same as the one induced by the graph norm of $\mathcal{A}_\lambda$. In particular, $\mathcal{A}_\lambda$ induces a bounded operator between $\mathfrak{D}(\mathcal{A})_\lambda$ and $H$. Actually, the family $\mathcal{A}$ is a bundle morphism between $\mathfrak{D}(\mathcal{A})$ and the product bundle $X\times H$, which can be seen as follows. If $U_{\lambda_0}$ is a trivialising neighbourhood about some $\lambda_0$ and $\tau_{\lambda_0}:\pi^{-1}(U_{\lambda_0})\rightarrow U_{\lambda_0}\times \gra(\mathcal{A}_{\lambda_0})$ a trivialisation as above, then

\[\mathcal{A}_\lambda(\tau^{-1}_{\lambda_0}(\lambda,\cdot))=P_2(P_{\gra(\mathcal{A}_{\lambda_0})}\mid_{\gra(\mathcal{A}_\lambda)})^{-1}\in\mathcal{L}(\gra(\mathcal{A}_{\lambda_0}),H),\]  
where $P_2$ is the projection onto $\{0\}\times H$ in $H\times H$. As these bounded operators depend continuously on $\lambda\in U_{\lambda_0}$ by \eqref{bundle1}, we see that $\mathcal{A}$ is indeed a bundle morphisms.\\
Let us now assume that $\mathcal{A}:X\rightarrow\mathcal{CF}(H)$ is a family of Fredholm operators and $Y\subset X$ a closed subset such that $\mathcal{A}_\lambda$ is invertible for all $\lambda\in Y$. As $\mathfrak{D}(\mathcal{A})_\lambda$ has the topology induced by the graph norm, we see that $\mathcal{A}_\lambda$ is a bounded Fredholm operator between this space and $H$. Hence we obtain a Fredholm morphism $\mathcal{A}:\mathfrak{D}(\mathcal{A})\rightarrow X\times H$, which is invertible for $\lambda\in X$ if and only if the unbounded operator $\mathcal{A}_\lambda$ is invertible.\\
We can now apply the index bundle construction for Fredholm morphisms between Banach bundles to obtain a $K$-theory class

\[\ind(\mathcal{A})\in K(X,Y),\]  
which we henceforth call the \textit{index bundle} of the family $\mathcal{A}:X\rightarrow\mathcal{CF}(H)$.\\
The following properties of the index bundle are straightforward consequences of the corresponding rules for Fredholm morphisms from the final part of the previous section.

\begin{itemize}
\item If $\mathcal{A}_\lambda\in\mathcal{CF}(H)$ is invertible for every $\lambda\in X$, then

\[\ind(\mathcal{A})=0\in K(X,Y).\]
\item Let $\mathcal{A}_1,\mathcal{A}_2:X\rightarrow\mathcal{CF}(H)$ be such that $\mathcal{A}_{1,\lambda}$ and $\mathcal{A}_{2,\lambda}$ are invertible for all $\lambda\in Y$. Then

\[\ind(\mathcal{A}_1\oplus\mathcal{A}_2)=\ind(\mathcal{A}_1)\oplus\ind(\mathcal{A}_2)\in K(X,Y).\]
\item If $h:I\times X\rightarrow\mathcal{CF}(H)$ is continuous and $h(s,\lambda)$ is invertible for all $s\in I$ and $\lambda\in Y$, then

\[\ind(h(0,\cdot))=\ind(h(1,\cdot))\in K(X,Y).\]
\item If $(X',Y')$ is another compact pair and $f:(X',Y')\rightarrow(X,Y)$ continuous, then $(f^\ast\mathcal{A})_\lambda=\mathcal{A}_{f(\lambda)}$ defines a gap-continuous family $f^\ast\mathcal{A}:X'\rightarrow\mathcal{CF}(H)$ such that $(f^\ast\mathcal{A})_\lambda$ is invertible for all $\lambda\in Y'$. Moreover,

\[\ind(f^\ast\mathcal{A})=f^\ast\ind(\mathcal{A})\in K(X',Y').\]  
\end{itemize}
\noindent
Note that the last property requires to show that $\mathfrak{D}(f^\ast\mathcal{A})=f^\ast\mathfrak{D}(\mathcal{A})$ which, however, readily follows from the definition of the domain bundle.\\
We will need in the proof of Theorem A the following important property of the index bundle, which was not shown in \cite{thesis} in this generality.

\begin{lemma}\label{gap-product}
Let $\mathcal{A}:X\rightarrow\mathcal{CF}(H)$ be gap-continuous and such that $\mathcal{A}_\lambda$ is invertible for all $\lambda\in Y\subset X$. If $M,N:X\rightarrow GL(H)$ are continuous families of invertible operators on $H$, then $M\mathcal{A}N$ is gap-continuous, and 

\[\ind(M\mathcal{A}N)=\ind(\mathcal{A})\in K(X,Y).\]
\end{lemma}

\begin{proof}
Note that

\begin{align*}
\gra(M_\lambda\mathcal{A}_\lambda N_\lambda)&=\{(u,M_\lambda\mathcal{A}_\lambda N_\lambda u):\, u\in N^{-1}_\lambda(\mathcal{D}(\mathcal{A}_\lambda))\}=\{(N^{-1}_\lambda v,M_\lambda\mathcal{A}_\lambda v):\,v\in\mathcal{D}(\mathcal{A}_\lambda)\}\\
&=\begin{pmatrix}
N^{-1}_\lambda&0\\
0&M_\lambda
\end{pmatrix}\,\gra(\mathcal{A}_\lambda)=:U_\lambda\gra(\mathcal{A}_\lambda)\subset H\times H,
\end{align*}
and so $\{U_\lambda P_{\gra(\mathcal{A}_\lambda)}U^{-1}_\lambda\}_{\lambda\in X}$ is a continuous family of oblique projections onto\linebreak $\{\gra(M_\lambda\mathcal{A}_\lambda N_\lambda)\}_{\lambda\in X}$ in $\mathcal{L}(H\times H)$. By \cite[Thm. I.6.35]{Kato}, the corresponding orthogonal projections $P_{\gra(M_\lambda\mathcal{A}_\lambda N_\lambda)}$ onto $\gra(M_\lambda\mathcal{A}_\lambda N_\lambda)$ satisfy

\[\|P_{\gra(M_\mu\mathcal{A}_\mu N_\mu)}-P_{\gra(M_\lambda\mathcal{A}_\lambda N_\lambda)}\|\leq \|U_\mu P_{\gra(\mathcal{A}_\mu)}U^{-1}_\mu-U_\lambda P_{\gra(\mathcal{A}_\lambda)}U^{-1}_\lambda\|,\quad \mu,\lambda\in X,\]
and consequently $\{P_{\gra(M_\lambda\mathcal{A}_\lambda N_\lambda)}\}_{\lambda\in X}$ is continuous. This shows that $M\mathcal{A}N$ is gap-continuous.\\ 
For the second claim, we just need to note that, by Kuiper's Theorem, $M$ and $N$ are homotopic to the constant family $G_\lambda=I_H$, $\lambda\in X$. Hence we obtain by the homotopy invariance

\begin{align*}
\ind(M\mathcal{A}N)=\ind(\mathcal{A})\in K(X,Y),
\end{align*}
where the continuity of the homotopy follows as in the first part of this proof.
\end{proof}

%%%%%%%%%%%%%%%%%%%%%%%%%%%%%%%%%%%%%%%%%%%%%%%%%%%%%%%%%%%%%%%%%%%%%%%%%%%%%%%%%%%%%%%%%%%%%%%%%%%%%%%%%%%%%%%%%%%%%%%%%%%%%%%%%%%%%%%%%%%%%%%%%%%%%%%%%%%%%%%%%%%%%%%%%%%%%%%%%%%%%%%%%%%%%%%%%%%%%%%%%%%%%%%%%%%%%%%%%%%%%%%%%%%%%%%%%%%%%%%%%%%%%%%%%%%%%%%%%%%%%%%%%%%%%%%%%%%%%%%%%%%%%%%%%%%%%%%%%%%%%%%%%%%%%%%%%%%%%%%%%%%%%%%%%%%%%%%%%%%%%%

\subsection{Spectral Flow and the Index Bundle}\label{section-sflbundle}
We now consider families $\mathcal{A}:X\rightarrow\mathcal{CF}^\textup{sa}(H)$, and we assume again that $Y\subset X$ is a closed subset such that $\mathcal{A}_\lambda$ is invertible for $\lambda\in Y$. We note at first that $\ind(\mathcal{A})=0\in K(X,Y)$ in this case, which follows by deforming $\mathcal{A}$ to the family $\mathcal{A}+iI_H$. As the operators $\mathcal{A}_\lambda$ are selfadjoint, $\mathcal{A}_\lambda+iI_H$ is invertible for any $\lambda$ and so this family has indeed a trivial index bundle.\\
We now define a family

\[\hat{\mathcal{A}}:X\times\mathbb{R}\rightarrow\mathcal{CF}(H)\]
where $\mathcal{D}(\hat{\mathcal{A}}_{(\lambda,s)})=\mathcal{D}(\mathcal{A}_\lambda)$ for $(\lambda,s)\in X\times\mathbb{R}$ and

\[\hat{\mathcal{A}}_{(\lambda,s)}=\mathcal{A}_\lambda+i\,s\,I_H.\]
Note that $\hat{\mathcal{A}}_{(\lambda,s)}$ is invertible if $s\neq 0$ as $\mathcal{A}_\lambda$ is selfadjoint. Hence, $\hat{\mathcal{A}}_{(\lambda,s)}$ is in $\mathcal{CF}(H)$ for all $(\lambda,s)\in X\times\mathbb{R}$. 

\begin{lemma}
The family $\hat{\mathcal{A}}:X\times\mathbb{R}\rightarrow\mathcal{CF}(H)$ is gap-continuous.
\end{lemma}

\begin{proof}
We let $\lambda_0\in X$, $s_0\in\mathbb{R}$, and obtain from the triangle inequality

\begin{align}\label{dGtriangle}
\begin{split}
d_G(\mathcal{A}_\lambda+is\, I_H,\mathcal{A}_{\lambda_0}+is_0\, I_H)\leq&  d_G(\mathcal{A}_\lambda+is\, I_H,\mathcal{A}_{\lambda_0}+is\, I_H)\\
+&d_G(\mathcal{A}_{\lambda_0}+is\, I_H,\mathcal{A}_{\lambda_0}+is_0\, I_H).
\end{split}
\end{align}
By \cite[Thm. IV.2.17]{Kato}, we have  

\begin{align}\label{gap-estimate}
d_G(\mathcal{A}_\lambda+is\, I_H,\mathcal{A}_{\lambda_0}+is\, I_H)\leq 2(1+s^2)\,d_G(\mathcal{A}_\lambda,\mathcal{A}_{\lambda_0}).
\end{align}
For the remaining term, we note that the family of isomorphisms

\[U_s:H\times H\rightarrow H\times H,\quad U_s(u,v)=(u,v-i(s-s_0)u)\]
maps $\gra(\mathcal{A}_{\lambda_0}+is\, I_H)$ to $\gra(\mathcal{A}_{\lambda_0}+is_0\, I_H)$. Hence $U^{-1}_s P_{\gra(\mathcal{A}_{\lambda_0}+is_0\, I_H)}U_s$ is an oblique projection onto $\gra(\mathcal{A}_{\lambda_0}+is\, I_H)$. The rest of the proof is similar to Lemma \ref{gap-product}. By \cite[Thm. I.6.35]{Kato}, the corresponding orthogonal projections $P_{\gra(\mathcal{A}_{\lambda_0}+is\, I_H)}$ onto $\gra(\mathcal{A}_{\lambda_0}+is\, I_H)$ satisfy

\[\|P_{\gra(\mathcal{A}_{\lambda_0}+is\, I_H)}-P_{\gra(\mathcal{A}_{\lambda_0}+is_0\, I_H)}\|\leq \|U^{-1}_s P_{\gra(\mathcal{A}_{\lambda_0}+is_0\, I_H)}U_s- P_{\gra(\mathcal{A}_{\lambda_0}+is_0\, I_H)}\|,\]
where we use that $U_{s_0}=I_{H\oplus H}$. This shows the continuity of $\{P_{\gra(\mathcal{A}_{\lambda_0}+is\, I_H)}\}_{s\in\mathbb{R}}$ at $s_0$, and so we obtain the continuity of $\hat{\mathcal{A}}$ from \eqref{dGtriangle}.
\end{proof}
\noindent
Let us point out that our family $\hat{\mathcal{A}}$ is parametrised by the non-compact topological space $X\times\mathbb{R}$ and so we cannot apply the index bundle construction from the previous section. However, using that $\hat{\mathcal{A}}_{(\lambda,s)}$ is invertible for all $(\lambda,s)$ that are outside of the compact space $X\times\{0\}$, it is readily seen that the previous construction carries over to this slightly more general setting. Actually, we only need to use that there is a bundle as in \eqref{transversal} for the restricted family $\hat{\mathcal{A}}\mid_{X\times\{0\}}=\mathcal{A}$, which is parametrised by a compact space. However, for later reference, we note the following refined version of \eqref{transversal}.

\begin{lemma}\label{kerneltransversal}
Let $\mathcal{A}=\{\mathcal{A}_\lambda\}_{\lambda\in X}$ be a gap-continuous family in $\mathcal{CF}^\textup{sa}(H)$. Then there are $\lambda_1,\ldots,\lambda_m\in X$ such that

\begin{align}\label{transversalselfadjointspace}
V:=\ker(\mathcal{A}_{\lambda_1})+\cdots +\ker(\mathcal{A}_{\lambda_m})
\end{align}
satisfies

\begin{align}\label{transversalselfadjoint}
\im(\mathcal{A}_\lambda)+V=H,\quad \lambda\in X.
\end{align}
\end{lemma}

\begin{proof}
We note at first that the assertion is obviously true if $\mathcal{A}_\lambda$ is invertible for all $\lambda\in X$. Let us now assume that there is $\lambda_0\in X$ such that $\ker(\mathcal{A}_{\lambda_0})\neq\{0\}$. Let 

\[\psi:\mathfrak{D}(\mathcal{A})\mid_U\rightarrow U\times\gra(\mathcal{A}_{\lambda_0})\]
be a trivialisation of the domain bundle $\mathfrak{D}(\mathcal{A})$ in a neighbourhood $U$ of $\lambda_0$, and let us consider the bounded Fredholm operators $L_\lambda:=\mathcal{A}_\lambda\circ\psi^{-1}_\lambda:\gra(\mathcal{A}_{\lambda_0})\rightarrow H$ for $\lambda\in U$. If $P$ denotes the orthogonal projection onto the closed subspace $\im(\mathcal{A}_{\lambda_0})\subset H$, then the composition

\[\gra(\mathcal{A}_{\lambda_0})\xrightarrow{L_{\lambda_0}} H\xrightarrow{P}\im(\mathcal{A}_{\lambda_0})\]
is surjective. By \cite[Thm. XI.6.1]{Gohberg}, there exists a bounded right inverse $M$, i.e. $(P\circ L_{\lambda_0})\circ M=I_{\im(\mathcal{A}_{\lambda_0})}$. As $GL(\im(\mathcal{A}_{\lambda_0}))$ is open in $\mathcal{L}(\im(\mathcal{A}_{\lambda_0}))$ in the norm topology, we see that there is a neighbourhood $U_{\lambda_0}\subset U$ such that $(P\circ L_\lambda)\circ M\in GL(\im(\mathcal{A}_{\lambda_0}))$ for all $\lambda\in U_{\lambda_0}$. Consequently, $P\circ L_\lambda$ is surjective or, equivalently,

\[\im(\mathcal{A}_\lambda) +\ker(\mathcal{A}_{\lambda_0})=\im(L_\lambda)+\ker(\mathcal{A}_{\lambda_0})=H,\quad\lambda\in U_{\lambda_0},\]
where we have used that $\im(\mathcal{A}_{\lambda_0})^\perp=\ker(\mathcal{A}_{\lambda_0})$.\\
Let us now denote by $\Sigma\subset X$ the set of all $\lambda\in X$ such that $\ker(\mathcal{A}_\lambda)$ is non-invertible. As the set of invertible elements in $\mathcal{C}(H)$ is open by \cite[Thm. IV.5.2.21]{Kato}, we see that $\Sigma$ is closed as preimage of a closed set under the continuous map $\mathcal{A}:X\rightarrow\mathcal{CF}^\textup{sa}(H)$. Hence, as $X$ is compact, we can find $\lambda_1,\ldots,\lambda_m$ such that the corresponding neighbourhoods $\{U_{\lambda_i}\}_{i=1,\ldots,m}$ from the first step of the proof are a finite open cover of $\Sigma$. Finally, we set as in \eqref{transversalselfadjointspace}

\[V:=\ker(\mathcal{A}_{\lambda_1})+\cdots+\ker(\mathcal{A}_{\lambda_m})\]
and note that this space indeed satisfies \eqref{transversalselfadjoint}.
\end{proof}
\noindent 
Let us now continue with the construction of the index bundle for our family $\mathcal{A}:X\rightarrow\mathcal{CF}^\textup{sa}(H)$. By the previous Lemma \ref{kerneltransversal}, there is $V\subset H$ such that

\begin{align}\label{transversalselfadjointV}
\im(\hat{\mathcal{A}}_{(\lambda,s)})+V=H,\quad (\lambda,s)\in X\times\mathbb{R},
\end{align}
where we use that $\hat{\mathcal{A}}_{(\lambda,s)}$ is invertible if $s\neq 0$. Hence the domain bundle $E(\hat{\mathcal{A}},V)$ is defined and $\hat{\mathcal{A}}$ induces a Fredholm morphism $E(\hat{\mathcal{A}},V)\rightarrow\Theta(V)$, where $\Theta(V)$ now denotes the product bundle with fibre $V$ over $X\times\mathbb{R}$. Note that $\hat{\mathcal{A}}$ is invertible outside the compact set $X\times\{0\}$ and for all $(\lambda,0)\in Y\times\mathbb{R}$. Consequently, we obtain an odd $K$-theory class

\[\sind(\mathcal{A}):=[E(\hat{\mathcal{A}},V),\Theta(V),\hat{\mathcal{A}}]\in K(X\times\mathbb{R},Y\times\mathbb{R})=K^{-1}(X,Y),\]
which we call the \textit{index bundle} of the selfadjoint family $\mathcal{A}$. As in the non-selfadjoint case, it is not very difficult to see that $\sind(\mathcal{A})$ is well defined, i.e. it does not depend on the choice of the space $V$ in \eqref{transversalselfadjointV}. Moreover, we note the following properties:

\begin{itemize}
\item If $\mathcal{A}_\lambda$ is invertible for all $\lambda\in I$, then $\sind(\mathcal{A})=0\in K^{-1}(X,Y)$.
\item Let $\mathcal{A}_1,\mathcal{A}_2:X\rightarrow\mathcal{CF}^\textup{sa}(H)$ be two gap-continuous families such that $\mathcal{A}_{1,\lambda}$ and $\mathcal{A}_{2,\lambda}$ are invertible for all $\lambda\in Y$, then

\[\sind(\mathcal{A}_1\oplus\mathcal{A}_2)=\sind(\mathcal{A}_1)+\sind(\mathcal{A}_2)\in K^{-1}(X,Y).\]
\item If $h:I\times X\rightarrow\mathcal{CF}^\textup{sa}(H)$ is gap-continuous and $h(s,\lambda)$ is invertible for all $s\in I$ and $\lambda\in Y$, then

\[\sind(h(0,\cdot))=\sind(h(1,\cdot))\in K^{-1}(X,Y).\]  
\end{itemize}
\noindent
We discussed in the previous section the index bundle for gap-continuous families under pullbacks. For the corresponding result in the selfadjoint case, we need to introduce a further notation. If $(X',Y')$ is another compact pair and $f:(X',Y')\rightarrow(X,Y)$ continuous, then we set

\[\overline{f}:X'\times\mathbb{R}\rightarrow X\times\mathbb{R},\quad \overline{f}(\lambda,s)=(f(\lambda),s),\]
and note that

\begin{itemize}
\item $\sind(f^\ast\mathcal{A})=\overline{f}^\ast\sind(\mathcal{A})\in K^{-1}(X,Y)$, where, as before, $(f^\ast\mathcal{A})_\lambda=\mathcal{A}_{f(\lambda)}$.
\end{itemize}
\noindent
Finally, we obtain from Corollary \ref{cor-orthequ} and Lemma \ref{gap-product} the following result.

\begin{lemma}\label{MANself}
Assume that $\mathcal{A}:X\rightarrow\mathcal{CF}^\textup{sa}(H)$ is such that $\mathcal{A}_\lambda$ is invertible for all $\lambda\in Y$. Let $M:X\rightarrow GL(H)$ be a family of bounded invertible operators and let us denote by $M^\ast_\lambda$ the adjoint of $M_\lambda$. Then

\[\sind(M^\ast \mathcal{A}M)=\sind(\mathcal{A})\in K^{-1}(X,Y).\]
\end{lemma}
\noindent
In order to discuss Theorem C, we now consider the case that $(X,Y)=(I,\partial I)$. Note that there is an isomorphism $c_1:K^{-1}(I,\partial I)\rightarrow\mathbb{Z}$ which we recall in Appendix \ref{app-K}. Hence we can assign to any path in $\mathcal{CF}^\textup{sa}(H)$ having invertible endpoints an integer as first Chern number of its index bundle.

\begin{theoremC}
Let $\mathcal{A}=\{\mathcal{A}_\lambda\}_{\lambda\in I}$ be a path in $\mathcal{CF}^\textup{sa}(H)$ such that $\mathcal{A}_0$ and $\mathcal{A}_1$ are invertible. Then

\[\sfl(\mathcal{A})=c_1(\sind(\mathcal{A}))\in\mathbb{Z}.\]
\end{theoremC}
\noindent
Let us point out that it follows from \eqref{complexification} that if $E$ is a real Hilbert space and $\mathcal{A}=\{\mathcal{A}_\lambda\}_{\lambda\in I}$ a path in $\mathcal{CF}^\textup{sa}(E)$, then

\begin{align}\label{Creal}
\sfl(\mathcal{A})=c_1(\sind(\mathcal{A}^\mathbb{C}))).
\end{align}

%Following lemma for general families?
%\begin{lemma}\label{lem-sfl-conj}
%Let $\mathcal{A}:I\rightarrow\mathcal{CF}\textup{sa}(H)$ be gap continuous and $M:I\rightarrow GL(H)$ a path of invertible operators. Then $M^T\mathcal{A}M$ is a gap continuous path in $\mathcal{CF}^\textup{sa}(H)$ and 

%\[\sfl(M^T\mathcal{A}M)=\sfl(\mathcal{A}).\]
%\end{lemma}

%\begin{proof}
%We first note that $M^T_\lambda\mathcal{A}_\lambda M_\lambda$ are selfadjoint Fredholm operators by Corollary \eqref{cor-orthequ}. Moreover, the path $M^T\mathcal{A}M$ is gap continuous by Lemma \ref{gap-product}. It remains to show the equality of the spectral flows. By Kuiper's Theorem the space $GL(H)$ is contractible, and hence there exists a homotopy $h:I\times I\rightarrow GL(H)$ such that $h(0,\lambda)=M_\lambda$ and $h(1,\lambda)=I_H$ for all $\lambda\in I$. If we set $\hat{h}_\lambda=h^T_\lambda\mathcal{A}_\lambda h_\lambda$, then it follows by using  Lemma \ref{gap-product} and Corollary \ref{cor-orthequ} again that $\hat{h}$ is a homotopy in $\mathcal{CF}^\textup{sa}(H)$. Now the claim follows from the homotopy invariance of the spectral flow.  
%\end{proof}

%%%%%%%%%%%%%%%%%%%%%%%%%%%%%%%%%%%%%%%%%%%%%%%%%%%%%%%%%%%%%%%%%%%%%%%%%%%%%%%%%%%%%%%%%%%%%%%%%%%%%%%%%%%%%%%%%%%%%%%%%%%%%%%%%%%%%%%%%%%%%%%%%%%%%%%%%%%%%%%%%%%%%%%%%%%%%%%%%%%%%%%%%%%%%%%%%%%%%%%%%%%%%%%%%%%%%%%%%%%%%%%%%%%%%%%%%%%%%%%%%%%%%%%%%%%%%%%%%%%%%%%%%%%%%%%%%%%%%%%%%%%%%%%%%%%%%%%%%%%%%%%%%%%%%%%%%%%%%%%%%%%%%%%%%%%%%%%%%%%%%%%%%%%%%%%%%%%%%%%%%%%%%%%%%%%%%%%%%%%%%%%%%%%%%%%%%%%%%%%%%%%%%%%%%%%%%%%%%%%%%%%%%%%%%%%%%%%%%%%%%%%%%%%%%%%%%

\subsection{Proof of Theorem C}
We recall from Section \ref{section-sfl} that $\Omega(\mathcal{CF}^\textup{sa}(H),G\mathcal{C}^\textup{sa}(H))$ denotes the set of all paths in $\mathcal{CF}^\textup{sa}(H)$ having invertible endpoints. We set

\[\mu:\Omega(\mathcal{CF}^\textup{sa}(H),G\mathcal{C}^\textup{sa}(H))\rightarrow\mathbb{Z},\quad \mu(\mathcal{A})=c_1(\sind(\mathcal{A}))\in\mathbb{Z}\]
and now show that $\mu$ satisfies all assumptions of Theorem \ref{LeschUniqueness} in three steps.

%%%%%%%%%%%%%%%%%%%%%%%%%%%%%%%%%%%%%%%%%%%%%%%%%%%%%%%%%%%%%%%%%%%%%%%%%%%%%%%%%%%%%%%%%%%%%%%%%%%%%%%%%%%%%%%%%%%%%%%%%%%%%%%%%%%%%%%%%%%%%%%%%%%%%%%

\subsubsection*{Step 1: Homotopy Invariance}
It follows from the properties of the index bundle for families in $\mathcal{CF}^\textup{sa}(H)$ that 

\[\sind(\mathcal{A}^1)=\sind(\mathcal{A}^2)\in K^{-1}(I,\partial I)\]
if $\mathcal{A}^1$ and $\mathcal{A}^2$ are homotopic by a homotopy inside $\Omega(\mathcal{CF}^\textup{sa}(H),G\mathcal{C}^\textup{sa}(H))$. Hence the same is true for $\mu$ showing the homotopy invariance in Theorem \ref{LeschUniqueness}. 

%%%%%%%%%%%%%%%%%%%%%%%%%%%%%%%%%%%%%%%%%%%%%%%%%%%%%%%%%%%%%%%%%%%%%%%%%%%%%%%%%%%%%%%%%%%%%%%%%%%%%%%%%%%%%%%%%%%%%%%%%%%%%%%%%%%%%%%%%%%%%%%%%%%%%%%

\subsubsection*{Step 2: Additivity under Concatenation}
This step is based on the following rather technical lemma.

\begin{lemma} 
Let $\mathcal{A}\in\Omega(\mathcal{CF}^\textup{sa}(H),G\mathcal{C}^\textup{sa}(H))$ and $f_1,f_2:I\rightarrow I$ continuous functions such that $f_1(0)=0$, $f_2(1)=1$ and $f_1(1)=f_2(0)$. If $\mathcal{A}_{f_1(1)}\in G\mathcal{C}^\textup{sa}(H)$, then

\[\sind((f_1\ast f_2)^\ast\mathcal{A})=\sind(f^\ast_1\mathcal{A})+\sind(f^\ast_2\mathcal{A})\in K^{-1}(I,\partial I).\]
\end{lemma}

\begin{proof}
We note at first that the domain bundle $\mathfrak{D}(\hat{\mathcal{A}})$ is trivial as it is a bundle over the contractible space $I\times\mathbb{R}$.  Consequently, there is a global trivialisation $\psi:(I\times\mathbb{R})\times H\rightarrow\mathfrak{D}(\hat{\mathcal{A}})$ and we obtain from the basic properties of the index bundle for Fredholm morphisms for any $f:I\rightarrow I$

\begin{align}\label{proofconcatenation}
\begin{split}
\sind(f^\ast\mathcal{A})&=\overline{f}^\ast\ind(\hat{\mathcal{A}})=\overline{f}^\ast(\ind(\hat{\mathcal{A}})+\ind(\psi))\\
&=\overline{f}^\ast\ind(\hat{\mathcal{A}}\circ\psi)=\ind(\overline{f}^\ast(\hat{\mathcal{A}}\circ\psi)).
\end{split}
\end{align}
We set $\tilde{\mathcal{A}}:=\hat{\mathcal{A}}\circ\psi$ which is a family of bounded Fredholm operators on $H$ parametrised by $I\times\mathbb{R}$.\\
Let $g_1,g_2:I\times\mathbb{R}\rightarrow I\times\mathbb{R}$ be defined by

\begin{align*}
g_1(\lambda,s)&=\begin{cases}
\overline{f_1}(2\lambda,s),&\quad 0\leq \lambda\leq\frac{1}{2}\\
\overline{f_1}(1,s),&\quad \frac{1}{2}\leq \lambda\leq 1
\end{cases}\\
g_2(\lambda,s)&=\begin{cases}
\overline{f_2}(0,s),&\quad 0\leq \lambda\leq\frac{1}{2}\\
\overline{f_2}(2\lambda-1,s),&\quad \frac{1}{2}\leq \lambda\leq 1
\end{cases}.
\end{align*}
We consider the homotopy $h:I\times (I\times\mathbb{R})\rightarrow\mathcal{L}(H\oplus H)$ defined by

\begin{align*}
h_\Theta(\lambda,s)= (g^\ast_1\tilde{\mathcal{A}})_{(\lambda,s)}\oplus (g^\ast_2\tilde{\mathcal{A}})_{(\lambda,s)},\quad 0\leq \lambda\leq \frac{1}{2},
\end{align*}
and

\begin{align*}
h_\Theta(\lambda,s)=
\begin{pmatrix}
\cos\left(\frac{\pi\Theta}{2}\right)&\sin\left(\frac{\pi\Theta}{2}\right)\\
-\sin\left(\frac{\Theta\lambda}{2}\right)&\cos\left(\frac{\Theta\lambda}{2}\right)
\end{pmatrix}
\begin{pmatrix}
(g^\ast_1\tilde{\mathcal{A}})_{(\lambda,s)}&0\\
0&(g^\ast_2\tilde{\mathcal{A}})_{(\lambda,s)}
\end{pmatrix}
\begin{pmatrix}
\cos\left(\frac{\pi\Theta}{2}\right)&-\sin\left(\frac{\pi\Theta}{2}\right)\\
\sin\left(\frac{\pi\Theta}{2}\right)&\cos\left(\frac{\pi\Theta}{2}\right)
\end{pmatrix}
\end{align*}
for $\frac{1}{2}\leq \lambda\leq 1$. 
Note that $h$ is continuous, as $g_1(\frac{1}{2},s)=g_2(\frac{1}{2},s)$ and in this case the above matrix product is $(g^\ast_1\tilde{\mathcal{A}})_{(\lambda,s)}\oplus(g^\ast_2\tilde{\mathcal{A}})_{(\lambda,s)}$. Hence the definitions of $h_\Theta$ coincide for all $(\Theta,\lambda,s)\in I\times\{\frac{1}{2}\}\times\mathbb{R}$.\\
The homotopy $h$ connects the maps $h_0=g^\ast_1\tilde{\mathcal{A}}\oplus g^\ast_2\tilde{\mathcal{A}}$ and 

\begin{align*}
h_1(\lambda,s)=\begin{cases}
(g^\ast_1\tilde{\mathcal{A}})_{(\lambda,s)}\oplus(g^\ast_2\tilde{\mathcal{A}})_{(\lambda,s)},\quad 0\leq \lambda\leq\frac{1}{2}\\
(g^\ast_2\tilde{\mathcal{A}})_{(\lambda,s)}\oplus (g^\ast_1\tilde{\mathcal{A}})_{(\lambda,s)},\quad\frac{1}{2}\leq \lambda\leq 1.
\end{cases}
\end{align*} 
Moreover, it is readily seen that 

\[h_1=(\overline{f_1\ast f_2}^\ast\tilde{\mathcal{A}})\oplus\tilde{\mathcal{A}}_{\overline{f}_1(1,\cdot)}.\]
Finally, note that $h_\Theta(\lambda,s)$ is invertible if either $\lambda=0,1$, or $s\neq 0$. This allows us to use the homotopy invariance property of $K$-theory

\begin{align*}
\ind(\overline{f_1\ast f_2}^\ast\tilde{\mathcal{A}})&=\ind(\overline{f_1\ast f_2}^\ast\tilde{\mathcal{A}})+\ind(\tilde{\mathcal{A}}_{\overline{f}_1(1,\cdot)})=\ind(h_1)\\
&=\ind(h_0)=\ind(g^\ast_1\tilde{\mathcal{A}}\oplus g^\ast_2\tilde{\mathcal{A}})=\ind(g^\ast_1\tilde{\mathcal{A}})+\ind(g^\ast_2\tilde{\mathcal{A}})\\
&=\ind(\overline{f}^\ast_1\tilde{\mathcal{A}})+\ind(\overline{f}^\ast_2\tilde{\mathcal{A}}),
\end{align*} 
where we have used in the final step that $g_1$ is homotopic to $\overline{f}_1$ and $g_2$ is homotopic to $\overline{f}_2$ by canonical homotopies. Now the assertion follows from \eqref{proofconcatenation}.
\end{proof}
\noindent 
The rest of the proof of the concatenation property is now easily obtained. Indeed, let\linebreak $\mathcal{A}_1,\mathcal{A}_2\in\Omega(\mathcal{CF}^\textup{sa}(H),G\mathcal{C}^\textup{sa}(H))$ be such that the concatenation $\mathcal{A}_1\ast\mathcal{A}_2$ is defined. We define two functions $f_1,f_2:I\rightarrow I$ by 

\[f_1(t)=\frac{1}{2}t,\quad f_2(t)=\frac{1}{2}(t+1),\]
and note that $f_1\ast f_2$ is the identity on $I$, as well as $f^\ast_i(\mathcal{A}_1\ast\mathcal{A}_2)=\mathcal{A}_i$ for $i=1,2$. The assertion follows by applying the previous lemma to $f_1, f_2$ and $\mathcal{A}:=\mathcal{A}_1\ast\mathcal{A}_2$.

%%%%%%%%%%%%%%%%%%%%%%%%%%%%%%%%%%%%%%%%%%%%%%%%%%%%%%%%%%%%%%%%%%%%%%%%%%%%%%%%%%%%%%%%%%%%%%%%%%%%%%%%%%%%%%%%%%%%%%%%%%%%%%%%%%%%%%%%%%%%%%%%%%%%%%%

\subsubsection*{Step 3: Normalisation} 
We consider the path $\mathcal{A}_{nor}=\{\mathcal{A}_\lambda\}_{\lambda\in I}$ in \eqref{normalisationpath}, and recall that $\mathcal{A}_\lambda$ is not invertible if and only if $\lambda=\lambda_0:=\frac{1}{2}$. The kernel and cokernel of $\mathcal{A}_{\lambda_0}$ is the one-dimensional space $V:=\im(P_0)$ and so 

\[\im(\hat{\mathcal{A}}_{(\lambda,s)})+V=H,\quad  (\lambda,s)\in I\times\mathbb{R}.\]
As $H=\im(P_+)\oplus\im(P_0)\oplus\im(P_-)$ is a decomposition of $H$ into invariant subspaces of $\mathcal{A}_\lambda$, we see that 

\[\hat{\mathcal{A}}^{-1}_{(\lambda,s)}(V)=(\mathcal{A}_\lambda+is I_H)^{-1}(V)=V, \quad (\lambda,s)\in I\times\mathbb{R}.\] 
Moreover, 

\[\hat{\mathcal{A}}_{(\lambda,s)}\mid_V=\left(\lambda-\lambda_0\right)P_0+is I_V.\]
Hence we obtain 

\[\sind(\mathcal{A})=[\Theta(V),\Theta(V),\hat{\mathcal{A}}\mid_V]=[\Theta(\mathbb{C}),\Theta(\mathbb{C}),\kappa]\in K^{-1}(I,\partial I),\]
where 

\[\kappa:I\times\mathbb{R}\rightarrow\mathbb{C},\quad \kappa(\lambda,s)=\lambda-\lambda_0+is.\]
By \eqref{Chern}, this yields

\[\mu(\mathcal{A})=c_1(\sind(\mathcal{A}))=\frac{1}{2\pi i}\int_{S^1}{\frac{1}{z-\lambda_0}\,dz}=1\in\mathbb{Z},\]
and so Theorem C is shown.

%%%%%%%%%%%%%%%%%%%%%%%%%%%%%%%%%%%%%%%%%%%%%%%%%%%%%%%%%%%%%%%%%%%%%%%%%%%%%%%%%%%%%%%%%%%%%%%%%%%%%%%%%%%%%%%%%%%%%%%%%%%%%%%%%%%%%%%%%%%%%%%%%%%%%%%%%%%%%%%%%%%%%%%%%%%%%%%%%%%%%%%%%%%%%%%%%%%%%%%%%%%%%%%%%%%%%%%%%%%%%%%%%%
%%%%%%%%%%%%%%%%%%%%%%%%%%%%%%%%%%%%%%%%%%%%%%%%%%%%%%%%%%%%%%%%%%%%%%%%%%%%%%%%%%%%%%%%%%%%%%%%%%%%%%%%%%%%%%%%%%%%%%%%%%%%%%%%%%%%%%%%%%%%%%%%%%%%%%%%%%%%%%%%%%%%%%%%%%%%%%%%%%%%%%%%%%%%%%%%%%%%%%%%%%%%%%%%%%%%%%%%%%%%%%%%%%%%%%%%%%%%%%%%%%%%%%%%%%%%%%%%%%%%%%%%%%%%%%%%%%%%%%%%%%%%%%%%%%%%%%%%%%%%%%%%%%%%%%%%%%%%%%%%%%%%%%%%%%%%%%%%%%%

\section{Theorem A}

\subsection{Setting and Statement}
Let $E$ be a symplectic Hilbert space with symplectic form $\omega(x,y)=\langle Jx,y\rangle_E$, where $J:E\rightarrow E$ is a bounded linear operator such that $J^2=-I_E$ and $J^T=-J$. We let $S:I\times\mathbb{R}\rightarrow\mathcal{S}(E)$ be a family of selfadjoint operators on $E$ and consider the differential operators

\begin{align}\label{Alambda}
\mathcal{A}_\lambda:H^1(\mathbb{R},E)\subset L^2(\mathbb{R},E)\rightarrow L^2(\mathbb{R},E),\quad(\mathcal{A}_\lambda u)(t)=Ju'(t)+S_\lambda(t)u(t).
\end{align}
In what follows, we assume that 

\[S_\lambda(t)=B_\lambda+K_\lambda(t),\quad (\lambda,t)\in I\times\mathbb{R},\]
where

\begin{itemize}
%\item[(A1)] $JA_\lambda$ is hyperbolic for all $\lambda\in I$, i.e., there are no purely imaginary points in the spectrum,
\item[(A1)] $K_\lambda(t)$ is compact for all $(\lambda,t)\in I\times\mathbb{R}$, and the limits

\[K_\lambda(\pm\infty)=\lim_{t\rightarrow\pm\infty}K_\lambda(t)\]
exist uniformly in $\lambda$,
\item[(A2)] the operators $JB_\lambda$ and

\[JS_\lambda(\pm\infty):=J(B_\lambda+K_\lambda(\pm\infty))\]
are hyperbolic, i.e., there are no purely imaginary points in their spectra.
\end{itemize}
The next two lemmas show that the spectral flow and the Maslov index in Theorem A are well defined.

\begin{lemma}\label{lem-AFredholm}
The operators $\mathcal{A}_\lambda$ are selfadjoint Fredholm operators under the assumptions (A1) and (A2).
\end{lemma}

\begin{proof}
Let us recall that two closed subspaces $V,W\subset E$ are called commensurable if the difference of their orthogonal projections $P_V-P_W$ is compact. Their relative dimension is defined by 

\[\dim(V,W)=\dim(W\cap V^\perp)-\dim(W^\perp\cap V)\]
which is a finite number (see \cite[\S 2]{AlbertoBuch}). By (A2), the operators $J S_\lambda(\pm\infty)$ have no spectra on the imaginary axis. Hence there are splittings 

\begin{align}\label{splittingstableunstable}
E=V^-(JS_\lambda(+\infty))\oplus V^+(JS_\lambda(+\infty))=V^-(JS_\lambda(-\infty))\oplus V^+(JS_\lambda(-\infty)),
\end{align}
where $V^-(JS_\lambda(\pm\infty))$ and $V^+(JS_\lambda(\pm\infty))$ denote the invariant subspaces of $JS_\lambda(\pm\infty)$ with respect to the negative and positive complex half-plane, respectively. Moreover, by (A1), the operators $JS_\lambda(+\infty)-JS_\lambda(-\infty)$ are compact, which implies that the same is true for the differences of their spectral projections onto $V^-(JS_\lambda(+\infty))$ and $V^-(JS_\lambda(-\infty))$ (see, e.g., \cite[Lemma 3.2]{AlbertoODE}). Hence these spaces are commensurable by \cite[Lemma 3.3]{AlbertoODE} and so their relative dimension is defined.\\
It was proved in \cite[Thm. B]{AlbertoODE} that the operators $\mathcal{A}_\lambda$ are Fredholm under the assumptions (A1)-(A2) and their Fredholm index is given by

\begin{align}\label{Hermann}
\ind(\mathcal{A}_\lambda)=\dim(V^-(JS_\lambda(+\infty)),V^-(JS_\lambda(-\infty))).
\end{align}
We now claim that $\ind(\mathcal{A}_\lambda)=0$. Let us note at first that $V^-(-A^T)=V^-(A)^\perp$ for any hyperbolic operator $A$ on $E$ (see, e.g., \cite[\S 1]{AlbertoODE}). Hence, for $A=JS_\lambda(\pm\infty)$, $V^-(JS_\lambda(\pm\infty))^\perp=V^-(S_\lambda(\pm\infty)J)$. By \eqref{Lperp=JL}, we see that the number \eqref{Hermann} vanishes if 

\begin{align}\label{indA=0}
JV^-(JS_\lambda(\pm\infty))=V^-(S_\lambda(\pm\infty)J),
\end{align}   
i.e. if $V^-(JS_\lambda(\pm\infty))$ are Lagrangian subspaces of $E$. To show this equality, we only need to note that 

\[J^{-1}(\mu-JS_\lambda(\pm\infty))^{-1}J=(\mu-S_\lambda(\pm\infty)J)^{-1}\]
for any $\mu\notin\sigma(JS_\lambda(\pm\infty))$. Therefore, if $P_1$ and $P_2$ denote the spectral projections onto \linebreak $V^-(JS_\lambda(\pm\infty))$ and $V^-(S_\lambda(\pm\infty)J)$, respectively, we get that $J^{-1}P_1J=P_2$. Hence, as $J^{-1}=-J$, $P_2$ projects onto 

\[J^{-1}\im(P_1)=J\im(P_1)=JV^-(JS_\lambda(\pm\infty))\]
and \eqref{indA=0} and so \eqref{Hermann} is shown.\\
Finally, it is readily seen that $\mathcal{A}_\lambda$ is symmetric by integration by parts. Hence, it follows from Lemma \ref{symmetric-selfadjoint} that these operators are selfadjoint Fredholm operators. 
\end{proof}
\noindent
Note that each $\mathcal{A}_\lambda$ has the same domain $H^1(\mathbb{R},E)$ which makes it easy to show $\mathcal{A}=\{\mathcal{A}_\lambda\}_{\lambda\in I}$ is a continuous path in $\mathcal{BF}^\textup{sa}(H^1(\mathbb{R},E),L^2(\mathbb{R},E))$. Hence, by Theorem \ref{Leschincl}, we see that $\mathcal{A}$ is continuous in $\mathcal{CF}^\textup{sa}(L^2(\mathbb{R},E))$, and so the spectral flow $\sfl(\mathcal{A})$ is defined.

\begin{lemma}\label{lemma-FL2}
If (A1) and (A2) hold, then $(E^u_\lambda(t_0),E^s_\lambda(t_0))\in\mathcal{FL}^2(E,\omega)$ for any $\lambda\in I$ and $t_0\in\mathbb{R}$.
\end{lemma}

\begin{proof}
Abbondandolo and Majer showed in \cite[Thm. D]{AlbertoODE} that $(E^u_\lambda(t_0),E^s_\lambda(t_0))$ is a Fredholm pair if and only if $\mathcal{A}_\lambda$ is Fredholm. Hence, by Lemma \ref{lem-AFredholm}, it remains to show that $E^u_\lambda(t_0), E^s_\lambda(t_0)\in \Lambda(E,\omega)$. We consider the differential equations $u'(t)-A(t)u(t)=0$, where $A$ is a continuous path of operators such that the limits $\lim_{t\rightarrow\pm\infty} A(t)$ exist and are hyperbolic. It was shown in \cite[Thm. 2.1]{AlbertoODE} that the stable and unstable spaces $E^s(A,t_0)$ and $E^u(A,t_0)$ of such an equation satisfy

\[E^s(-A^T,t_0)=E^s(A,t_0)^\perp,\qquad E^u(-A^T,t_0)=E^u(A,t_0)^\perp.\]
We set $A(t):=JS_\lambda(t)$ and obtain

\[E^{s/u}_\lambda(t_0)^\perp=E^{s/u}(JS_\lambda,t_0)^\perp=E^{s/u}(S_\lambda J,t_0).\]
Clearly, $u$ is a solution of $u'(t)-S_\lambda(t)Ju(t)=0$ if and only if $v(t):=Ju(t)$ satisfies $Jv'(t)+S_\lambda(t)v(t)=0$. Hence

\[E^{s/u}_\lambda(t_0)^\perp=E^{s/u}(S_\lambda J,t_0)=JE^{s/u}_\lambda(t_0),\]
which shows that these spaces are Lagrangian by \eqref{Lperp=JL}.
\end{proof}
\noindent
Finally, it follows from \cite[Thm. 3.1]{AlbertoODE} that $(E^u_\lambda(t_0),E^s_\lambda(t_0))\in\mathcal{FL}^2(E,\omega)$ depends continuously on $S_\lambda:\mathbb{R}\rightarrow\mathcal{S}(E)$ with respect to the $L^\infty$-topology on $C(\mathbb{R},\mathcal{S}(E))$. Hence $\{(E^u_\lambda(t_0),E^s_\lambda(t_0))\}_{\lambda\in I}$ is a continuous family in $\mathcal{FL}^2(E,\omega)$, and so the Maslov index is defined. The main theorem of this paper, which we prove in the following section, now reads as follows:

\begin{theoremA}\label{mainHomoclinics}
Let $S:I\times\mathbb{R}\rightarrow\mathcal{S}(E)$ be a continuous family of bounded selfadjoint operators satisfying the assumptions (A1) and (A2). Then 

\[\sfl(\mathcal{A})=\mu_{Mas}(E^u_\cdot(0),E^s_\cdot(0)).\]
\end{theoremA}
\noindent
Let us now consider the equations \eqref{Hamiltonian} under the additional periodicity assumption

\begin{itemize}
\item[(A3)] $S_0(t)=S_1(t)$ for all $t\in\mathbb{R}$,
\end{itemize} 
which implies that the path $\mathcal{A}$ of operators in \eqref{Alambda} is periodic, i.e. $\mathcal{A}_0=\mathcal{A}_1$. The autonomous systems

\begin{equation}\label{HamiltonianJacII}
\left\{
\begin{aligned}
Ju'(t)+S_\lambda(\pm\infty)u(t)&=0,\quad t\in\mathbb{R}\\
\lim_{t\rightarrow\pm\infty}u(t)&=0,
\end{aligned}
\right.
\end{equation}
have the stable and unstable spaces

\begin{align*}
E^s_\lambda(\pm\infty)&=\{x\in E: \exp(tJS_\lambda(\pm\infty))x\rightarrow 0\,\text{as}\, t\rightarrow\infty\}=V^-(JS_\lambda(\pm\infty)) \\
E^u_\lambda(\pm\infty)&=\{x\in E: \exp(tJS_\lambda(\pm\infty))x\rightarrow 0\,\text{as}\, t\rightarrow-\infty\}=V^+(JS_\lambda(\pm\infty)),
\end{align*}
where $V^+(JS_\lambda(\pm\infty))$ and $V^-(JS_\lambda(\pm\infty))$ are as in \eqref{splittingstableunstable}. The following corollary generalises the main theorem of \cite{Jacobo} from $\mathbb{R}^{2n}$ to symplectic Hilbert spaces.

\begin{corA}
If (A1)-(A3) hold, then

\[\sfl(\mathcal{A})=\mu_{Mas}(E^u_\cdot(+\infty),E^s_\cdot(-\infty)).\]
\end{corA}

\begin{proof}
We take a similar approach as in \cite[Prop. 3.3]{Hu} and consider for $t_0>0$ the concatenation $\gamma_1\ast\gamma_2\ast\gamma_3$ of the paths

\begin{align*}
\gamma_1&=\{(E^u_0(\lambda\cdot t_0)),E^s_0(-\lambda\cdot t_0)\}_{\lambda\in I},\quad \gamma_2=\{(E^u_\lambda(t_0),E^s_\lambda(-t_0))\}_{\lambda\in I},\\
 \gamma_3&=\{(E^u_1((1-\lambda)t_0), E^s_1(-(1-\lambda)t_0))\}_{\lambda\in I}. 
\end{align*}
We need to verify that they are in $\mathcal{FL}^2(E,\omega)$, and firstly note that we have seen in the proof of Lemma \ref{lemma-FL2} that all of these spaces are Lagrangian. Hence we only need to show that each pair of unstable and stable spaces in $\gamma_1$, $\gamma_2$ and $\gamma_3$ is Fredholm.\\
Let us consider $\gamma_2$ and leave $\gamma_1$ and $\gamma_3$ to the reader as the argument is similar and actually simpler. We set $A_\lambda(t)=S_\lambda(t-2t_0)$ for $t\in\mathbb{R}$, as well as 

\[E^s(A_\lambda,t_0)=\{u(t_0)\in E:\, Ju'(t)+A_\lambda(t)u(t)=0,\, u(t)\rightarrow 0,\, t\rightarrow +\infty\}\subset E,\]
and note that $E^s(A_\lambda,t_0)=E^s_\lambda(-t_0)$, where the latter space is the stable space in Theorem A. By (A1), the operators $A_\lambda(t)-S_\lambda(t)=K_\lambda(t-2t_0)-K(t)$ are compact, and so \cite[Thm. 3.6]{AlbertoODE} implies that $E^s(A_\lambda,t_0)$ and $E^s_\lambda(t_0)$ are commensurable. Moreover, $(E^u_\lambda(t_0),E^s_\lambda(t_0))$ is a Fredholm pair by Lemma \ref{lemma-FL2}. Now we just need to recall that if $V,W,Z\subset E$ are subspaces such that $V,W$ are commensurable and $(Z,V)$ is a Fredholm pair, then $(Z,W)$ is a Fredholm pair as well (see \cite[Prop. 2.2.1]{AlbertoBuch}). Hence $(E^u_\lambda(t_0),E^s_\lambda(-t_0))=(E^u_\lambda(t_0),E^s(A_\lambda,t_0))\in\mathcal{FL}^2(E,\omega)$.\\  
As $\{(E^u_\lambda(0),E^s_\lambda(0))\}_{\lambda\in I}$ is homotopic to $\gamma_1\ast\gamma_2\ast\gamma_3$, we obtain from the homotopy invariance and the concatenation property of the Maslov index, as well as \eqref{reverseI},

\begin{align*}
\mu_{Mas}(E^u_\cdot(0),E^s_\cdot(0))&=\mu_{Mas}(\gamma_1)+\mu_{Mas}(\gamma_2)+\mu_{Mas}(\gamma_3)\\
&=\mu_{Mas}(\gamma_2)=\mu_{Mas}(E^u_\cdot(t_0),E^s_\cdot(-t_0)),
\end{align*}
where we have used that $S_0=S_1$ and so $\gamma_1=-\gamma_3$.\\
Finally, it was shown in \cite[Thm. 2.1]{AlbertoODE} that $E^u_\lambda(\pm t)\rightarrow E^u_\lambda(\pm\infty)$ and $E^s_\lambda(\pm t)\rightarrow E^s_\lambda(\pm\infty)$ in $\mathcal{G}(E)$ for $t\rightarrow\infty$. As the stable and unstable spaces depend continuously on the asymptotically hyperbolic family $S_\lambda$ by \cite[Thm. 3.1]{AlbertoODE}, this shows that

\[\mu_{Mas}(E^u_\cdot(0),E^s_\cdot(0))=\mu_{Mas}(E^u_\cdot(+\infty),E^s_\cdot(-\infty)).\]
Consequently, the corollary follows from Theorem A.
\end{proof}
\noindent 
Let us point out that, in the special case where $E$ is of finite dimension, this corollary also shows that the main theorem of \cite{Jacobo} follows from our previous work \cite{WaterstraatHomoclinics}, which was not known before.

%%%%%%%%%%%%%%%%%%%%%%%%%%%%%%%%%%%%%%%%%%%%%%%%%%%%%%%%%%%%%%%%%%%%%%%%%%%%%%%%%%%%%%%%%%%%%%%%%%%%%%%%%%%%%%%%%%%%%%%%%%%%%%%%%%%%%%%%%%%%%%%%%%%%%%%%%%%%%%%%%%%%%%%%%%%%%%%%%%%%%%%%%%%%%%%%%%%%%%%%%%%%%%%%%%%%%%%%%%%%%%%%%%%%%%%%%%%%%%%%%%%%%%%%%%%%%%%%%%%%%%%%%%%%%%%%%%%%%%%%%%%%%%%%%%%%%%%%%%%%%%%%%%%%%%%%%%%%%%%%%%%%%%%%%%%%%%%%%%%%%%%%%%%%%%%%%%%%%%%%%%%%%%%%%%%%%%%%%%%%%%%%%%%%%%%%%%%%%%%%%%%%%%%%%%%%%%%%%%%%%%%%%%%%%%%%%%%%%%%%%%%%%%%%%%%%%%%%%%%%%%%%%%%%%%%%%%%%%%%%%%%%%%%%%%%%%%%%%%%%%%%%%%%%%%%%%%%%%%%%%%%%%%%%%%%%%%%%%%%%%%%%%%%%%%%%%%%%%%%%%%%%%%%%%%

\subsection{Proof of Theorem A}\label{Section-ProofA}
We split the proof into four steps. Our aim of the first three steps is to prove the following weaker version of Theorem A.

\begin{theorem}\label{TheoremA-weak}
If the assumptions of Theorem A hold and the differential equations \eqref{Hamiltonian} only have the trivial solution for $\lambda=0$ and $\lambda=1$, then

\[\sfl(\mathcal{A})=\mu_{Mas}(E^u_\cdot(0),E^s_\cdot(0)).\]
\end{theorem}
\noindent
For proving Theorem \ref{TheoremA-weak}, we begin by considering the Maslov index and apply in our first step an idea of Hu and Portaluri from \cite{Hu}. Then Theorem B will be used in the second step to join the Maslov index with the spectral flow of a path of differential operators having varying domains. The index bundle construction and Theorem C will show in the third step that this spectral flow is actually $\sfl(\mathcal{A})$, which shows Theorem \ref{TheoremA-weak}. Finally, in the fourth step, we lift the additional assumption in Theorem \ref{TheoremA-weak} and obtain Theorem A in its full generality.

%%%%%%%%%%%%%%%%%%%%%%%%%%%%%%%%%%%%%%%%%%%%%%%%%%%%%%%%%%%%%%%%%%%%%%%%%%%%%%%%%%%%%%%%%%%%%%%%%%%%%%%%%%%%%%%%%%%%%%%%%%%%%%%%%%%%%%%%%%%%%%%%%%%%%%%%%%%%%%%%%%%%%%%%%%%%%%%%%%%%%%%%%%%%%%%%%%%%%%%%%%%%%%%%%%%%%%%%%%%%%%%%%%%%%%%%%%%%%%%%%%%%%%%%%%%%%%%%%%%%%%%%%%%%%%%%%%%%%%%%%%%%%%%%%%%%%%%%%%%%%%%%%%%%%%%%%%%%%%%%%%%%%%%%%%%%%%%%%%%%%%
%%%%%%%%%%%%%%%%%%%%%%%%%%%%%%%%%%%%%%%%%%%%%%%%%%%%%%%%%%%%%%%%%%%%%%%%%%%%%%%%%%%%%%%%%%%%%%%%%%%%%%%%%%%%%%%%%%%%%%%%%%%%%%%%%%%%%%%%%%%%%%%%%%%%%%%%%%%%%%%%%%%%%%%%%%%%%%%%%%%%%%%%%%%%%%%%%%%%%%%%%%%%%%%%%%%%%%%%%%%%%%%%%%%%

\subsubsection*{Step 1: From $\mu_{Mas}(E^u_\cdot(0),E^s_\cdot(0))$ to $\mathcal{Q}$}
We have noted above that $\{(E^u_\lambda(0),E^s_\lambda(0))\}_{\lambda\in I}$ is a continuous path in $\mathcal{FL}^2(E,\omega)$. Hence by Theorem B we have for every fixed $t_0>0$

\begin{align}\label{mas=sfl}
\mu_{Mas}(E^u_\cdot(0),E^s_\cdot(0))=\sfl(\mathcal{Q}),
\end{align}
where $\mathcal{Q}=\{\mathcal{Q}_\lambda\}_{\lambda\in I}$ is the path of operators 

\[\mathcal{Q}_\lambda:\mathcal{D}(\mathcal{Q}_\lambda)\subset L^2([-t_0,t_0],E)\rightarrow L^2([-t_0,t_0],E)\] defined by $\mathcal{Q}_\lambda u=Ju'$ on the domains

\[\mathcal{D}(\mathcal{Q}_\lambda)=\{u\in H^1([-t_0,t_0],E):\, u(-t_0)\in E^u_\lambda(0),\, u(t_0)\in E^s_\lambda(0)\}.\]

%%%%%%%%%%%%%%%%%%%%%%%%%%%%%%%%%%%%%%%%%%%%%%%%%%%%%%%%%%%%%%%%%%%%%%%%%%%%%%%%%%%%%%%%%%%%%%%%%%%%%%%%%%%%%%%%%%%%%%%%%%%%%%%%%%%%%%%%%%%%%%%%%%%%%%%%%%%%%%%%%%%%%%%%%%%%%%%%%%%%%%%%%%%%%%%%%%%%%%%%%%%%%%%%%%%%%%%%%%%%%%%%%%%%
%%%%%%%%%%%%%%%%%%%%%%%%%%%%%%%%%%%%%%%%%%%%%%%%%%%%%%%%%%%%%%%%%%%%%%%%%%%%%%%%%%%%%%%%%%%%%%%%%%%%%%%%%%%%%%%%%%

\subsubsection*{Step 2: From $\mathcal{Q}$ to $\mathcal{A}^0$}
We consider the differential operators

\[\mathcal{A}^0_\lambda:\mathcal{D}(\mathcal{A}^0_\lambda)\subset L^2([-t_0,t_0],E)\rightarrow L^2([-t_0,t_0],E),\quad (\mathcal{A}^0_\lambda u)(t)=Ju'(t)+S_\lambda(t)u(t),\]
on the domains

\[\mathcal{D}(\mathcal{A}^0_\lambda)=\{u\in H^1([-t_0,t_0],E):\, u(-t_0)\in E^u_\lambda(-t_0),\, u(t_0)\in E^s_\lambda(t_0)\}.\]
Let $\Psi:[-t_0,t_0]\rightarrow GL(E)$ be given by

\begin{equation}\label{Psi}
\left\{
\begin{aligned}
J\Psi'_\lambda(t)+S_\lambda(t)\Psi_\lambda(t)&=0,\quad t\in[-t_0,t_0],\\
\Psi_\lambda(0)&=I_{E}
\end{aligned}
\right.
\end{equation}
and define a family of isomorphisms $M:I\rightarrow GL(L^2([-t_0,t_0],E))$ by

\[(M_\lambda u)(t)=\Psi^{-1}_\lambda(t)\,u(t).\]
We note that by \eqref{Psi}

\begin{align*}
(\Psi^T_\lambda(t)J\Psi_\lambda(t))'&=(\Psi'_\lambda(t))^TJ\Psi_\lambda(t)+\Psi^T_\lambda(t)J\Psi'_\lambda(t)\\
&=(JS_\lambda(t)\Psi_\lambda(t))^TJ\Psi_\lambda(t)+\Psi^T_\lambda(t)J^2S_\lambda(t)\Psi_\lambda(t)=0,
\end{align*}
and see from the initial value in \eqref{Psi} that $\Psi^T_\lambda(t)J\Psi_\lambda(t)=J$ for all $t\in[-t_0,t_0]$ and $\lambda\in I$. Hence 

\begin{align}\label{sympidentities}
\Psi^{-1}_\lambda(t)=-J\Psi^T_\lambda(t)J,\qquad (\Psi^{-1}_\lambda(t))^T=-J\Psi_\lambda(t)J,\quad (\lambda,t)\in I\times[-t_0,t_0].
\end{align}
We now claim that

\begin{align}\label{Q=A0}
\mathcal{A}^0_\lambda=M^T_\lambda\mathcal{Q}_\lambda M_\lambda,\quad\lambda\in I.
\end{align}
Indeed, we first see that

\begin{align*}
\mathcal{D}(M^T_\lambda\mathcal{Q}_\lambda M_\lambda)&=M^{-1}_\lambda\mathcal{D}(\mathcal{Q}_\lambda)\\
&=\{u\in H^1([-t_0,t_0],E):\, u(-t_0)\in\Psi_\lambda(-t_0)E^u_\lambda(0),\, u(t_0)\in\Psi_\lambda(t_0)E^s_\lambda(0)\}\\
&=\{u\in H^1([-t_0,t_0],E):\, u(-t_0)\in E^u_\lambda(-t_0),\, u(t_0)\in E^s_\lambda(t_0)\}=\mathcal{D}(\mathcal{A}^0_\lambda),
\end{align*}
where we have used that $\Psi_\lambda(t)E^{s/u}_\lambda(0)=E^{s/u}_\lambda(t)$ for all $t\in\mathbb{R}$. Furthermore, it follows from \eqref{sympidentities} that for $t\in[-t_0,t_0]$

\begin{align*}
(M^T_\lambda\mathcal{Q}_\lambda M_\lambda u)(t)&=-J\Psi_\lambda(t)J(J(\Psi^{-1}_\lambda(t))'u(t)+J\Psi^{-1}_\lambda(t)u'(t))\\
&=Ju'(t)-J\Psi_\lambda(t)J(J(-J(\Psi'_\lambda(t))^TJ)u(t))\\
&=Ju'(t)-J\Psi_\lambda(t)J(-(J\Psi'_\lambda(t))^Tu(t))\\
&=Ju'(t)-J\Psi_\lambda(t)J(S_\lambda(t)\Psi_\lambda(t))^Tu(t)\\
&=Ju'(t)-J\Psi_\lambda(t)J\Psi_\lambda(t)^TS_\lambda(t)u(t)\\
&=Ju'(t)+J(\Psi_\lambda(t)^TJ\Psi_\lambda(t))^TS_\lambda(t)u(t)\\
&=Ju'(t)+JJ^TS_\lambda(t)u(t)\\
&=Ju'(t)+S_\lambda(t)u(t)=(\mathcal{A}^0_\lambda u)(t).
\end{align*}
Hence \eqref{Q=A0} is shown, which implies by Corollary \ref{cor-orthequ} that each $\mathcal{A}^0_\lambda$ is a selfadjoint Fredholm operator. Moreover, $\mathcal{A}^0$ is a gap-continuous path in $\mathcal{CF}^\textup{sa}(L^2([-t_0,t_0],E))$ by Lemma \ref{gap-product}, and so the spectral flow $\sfl(\mathcal{A}^0)$ is defined.\\
The claimed equality $\sfl(\mathcal{A}^0)=\sfl(\mathcal{Q})$ is also readily seen from \eqref{Q=A0}. Indeed, we just need to note that $M$ is homotopic in $GL(L^2([-t_0,t_0],E))$ to the constant path $I_{L^2([-t_0,t_0],E)}$. Hence the homotopy invariance of the spectral flow yields

\[\sfl(\mathcal{Q})=\sfl(M^T\mathcal{Q}M)=\sfl(\mathcal{A}^0),\] 
where the continuity of the homotopy follows once again from Lemma \ref{gap-product}. This equality was the aim of this second step of our proof.

%%%%%%%%%%%%%%%%%%%%%%%%%%%%%%%%%%%%%%%%%%%%%%%%%%%%%%%%%%%%%%%%%%%%%%%%%%%%%%%%%%%%%%%%%%%%%%%%%%%%%%%%%%%%%%%%%%%%%%%%%%%%%%%%%%%%%%%%%%%%%%%%%%%%%%%%%%%%%%%%%%%%%%%%%%%%%%%%%%%%%%%%%%%%%%%%%%%%%%%%%%%%%%%%%%%%%%%%%%%%%%%%%%%%%%%%%%%%%%%%%%%%%%%%%%%%%%%%%%%%%%%%%%%%%%%%%%%%%%%%%%%%%%%%%%%%%%%%%%%%%%%%%%%%%%%%%%%%%%%%%%%%%%%%%%%%%%%%%%%%%%

\subsubsection*{Step 3: From $\mathcal{A}^0$ to $\mathcal{A}$}
The aim of our third step is to show that 

\begin{align}\label{A0}
\sfl(\mathcal{A}^0)=\sfl(\mathcal{A}),
\end{align}
which we will do by using \eqref{complexification}, the index bundle for gap-continuous families and Theorem C. Consequently, we need to work with the complexifications of $\mathcal{A}^0$ and $\mathcal{A}$. In order to simplify our notation, we denote in this step $E^\mathbb{C}$ by $H$, but we do not introduce new symbols for the complexifications of operators and their domains.\\
%We now note that obviously Theorem ????? is true if $\ker(\mathcal{A}_\lambda)=\{0\}$ for all $\lambda\in I$.  By \cite{WaterstraatHomoclinics}, we can now assume without loss of generality that there is only one instant $\lambda_0\in(0,1)$ such that $\ker\mathcal{A}_\lambda$ is non-trivial only if $\lambda=\lambda_0$. 
In what follows, we denote by $p:L^2(\mathbb{R},H)\rightarrow L^2([-t_0,t_0],H)$ the restriction to the interval $[-t_0,t_0]$. Let us recall that

\[\mathcal{D}(\mathcal{A}^0_\lambda)=\{u\in H^1([-t_0,t_0],H):\, u(-t_0)\in E^u_\lambda(-t_0)^\mathbb{C}, u(t_0)\in E^s_\lambda(t_0)^\mathbb{C}\}\]
and $\mathcal{D}(\mathcal{A}_\lambda)=H^1(\mathbb{R},H)$. We define for $\lambda\in I$ a map

\[\iota_\lambda:\mathcal{D}(\mathcal{A}^0_\lambda)\rightarrow\mathcal{D}(\mathcal{A}_\lambda)\]
by extending $u\in \mathcal{D}(\mathcal{A}^0_\lambda)$ to the whole real line by its boundary values as follows. If $u(-t_0)\in E^u_\lambda(-t_0)^\mathbb{C}$, we can extend $u$ to the interval $(-\infty,-t_0)$ as solution of the differential equation $Ju'(t)+S_\lambda(t)u(t)=0$. Similarly, $u$ can be extended to $[t_0,+\infty)$ as $u(t_0)\in E^s_\lambda(t_0)^\mathbb{C}$. Note that $\iota_\lambda(u)$ is indeed in $H^1(\mathbb{R},H)=\mathcal{D}(\mathcal{A}_\lambda)$ due to the exponential decay of solution curves starting from $E^s_\lambda(t_0)^\mathbb{C}$ in the positive direction, or from $E^u_\lambda(-t_0)^\mathbb{C}$ in the negative direction (see \cite[Thm. 2.1]{AlbertoODE}). Moreover, $\iota_\lambda$ is injective as obviously $p\circ\iota_\lambda=I_{\mathcal{D}(\mathcal{A}^0_\lambda)}$, and the diagram

\begin{align}\label{diagram1}
\begin{split}
\xymatrix{
&\mathcal{D}(\mathcal{A}^0_{\lambda})\ar[rr]^{\mathcal{A}^0_{\lambda}\,}\ar[d]_{\iota_\lambda}&&L^2([-t_0,t_0],H)\\
&\mathcal{D}(\mathcal{A}_{\lambda})\ar[rr]^{\mathcal{A}_{\lambda}}&&L^2(\mathbb{R},H)\ar[u]_{p}}
\end{split}
\end{align}
is commutative. Finally, 

\begin{align}\label{iota}
\iota_\lambda(\ker(\mathcal{A}^0_\lambda))=\ker(\mathcal{A}_\lambda), \quad\lambda\in I,
\end{align}
and so $\mathcal{A}^0_\lambda$ is invertible if and only if $\mathcal{A}_\lambda$ is invertible. In particular, as $\mathcal{A}$ has invertible endpoints by assumption, the same is true for $\mathcal{A}^0$. Hence $\sind(\mathcal{A})$ and $\sind(\mathcal{A}^0)$ are defined, and by Theorem C we now need to show that these classes coincide in $K^{-1}(I,\partial I)$ for proving \eqref{A0}.\\
We now consider as in Section \ref{section-sflbundle} the corresponding families of operators

\begin{align*}
\hat{\mathcal{A}^0}_{(\lambda,s)}&:\mathcal{D}(\hat{\mathcal{A}^0}_{(\lambda,s)})\subset L^2([-t_0,t_0],H)\rightarrow L^2([-t_0,t_0],H),\quad \hat{\mathcal{A}^0}_{(\lambda,s)}=\mathcal{A}^0_\lambda+isI_H\\
 \hat{\mathcal{A}}_{(\lambda,s)}&:\mathcal{D}(\hat{\mathcal{A}}_{(\lambda,s)})\subset L^2(\mathbb{R},H)\rightarrow L^2(\mathbb{R},H),\quad \hat{\mathcal{A}}_{(\lambda,s)}=\mathcal{A}_\lambda+isI_H
\end{align*}
from the construction of the index bundle for selfadjoint operators, which are parametrised by $(\lambda,s)\in I\times\mathbb{R}$. Let us recall that $\mathcal{D}(\hat{\mathcal{A}}_{(\lambda,s)})=\mathcal{D}(\mathcal{A}_\lambda)$ and $\mathcal{D}(\hat{\mathcal{A}^0}_{(\lambda,s)})=\mathcal{D}(\mathcal{A}^0_\lambda)$ for all $(\lambda,s)\in I\times\mathbb{R}$.\\
By Lemma \ref{kerneltransversal}, there are $\lambda_1,\ldots,\lambda_m\in I$ such that if we let $V\subset L^2([-t_0,t_0],H)$ be the sum of the $\ker(\mathcal{A}^0_{\lambda_i})$ and $W$ the sum of the $\ker(\mathcal{A}_{\lambda_i})$ for $i=1,\ldots,m$, then 

\[\im(\hat{\mathcal{A}^0}_{(\lambda,s)})+V=L^2([-t_0,t_0],H),\quad \im(\hat{\mathcal{A}}_{(\lambda,s)})+W=L^2(\mathbb{R},H),\quad (\lambda,s)\in I\times\mathbb{R}.\] 
We set

\[W_0=\{\chi_{[-t_0,t_0]}\,u:\, u\in W\},\quad W_1=\{\chi_{\mathbb{R}\setminus[-t_0,t_0]}\,u:\, u\in W\},\]
where $\chi_{[-t_0,t_0]}$ and $\chi_{\mathbb{R}\setminus[-t_0,t_0]}$ are characteristic functions. Note that there is a canonical injective linear map $M:V\rightarrow W_0\oplus W_1$ onto $W_0$ such that $p\circ M=I_V$. Moreover, since $W\subset W_0\oplus W_1$, we see that the latter space is transversal to the image of $\hat{\mathcal{A}}$ as in \eqref{transversalselfadjointV} and so the bundle $E(\hat{\mathcal{A}},W_0\oplus W_1)$ is defined.\\
If now $u\in E(\hat{\mathcal{A}^0},V)_{(\lambda,0)}$, then $\hat{\mathcal{A}^0}_{(\lambda,0)}u\in V$ and so $M(\hat{\mathcal{A}^0}_{(\lambda,0)}u)\in W_0\subset W_0\oplus W_1$. On the other hand, $\hat{\mathcal{A}}_{(\lambda,0)}(\iota_{\lambda}u)\in W_0$ which shows that $\iota_{\lambda}(E(\hat{\mathcal{A}^0},V)_{(\lambda,0)})\subset E(\hat{\mathcal{A}},W_0\oplus W_1)_{(\lambda,0)}$. Moreover, as \eqref{diagram1} is commutative, $p(\hat{\mathcal{A}}_{(\lambda,0)}(\iota_{\lambda} u))=\hat{\mathcal{A}^0}_{(\lambda,0)}u$ and so $\hat{\mathcal{A}}_{(\lambda,0)}(\iota_{\lambda} u)=M(\hat{\mathcal{A}^0}_{(\lambda,0)}u)$, where we use that $\hat{\mathcal{A}}_{(\lambda,0)}(\iota_{\lambda}u)\in W_0$ and $M\circ p\mid_{W_0}=I_{W_0}$.\\
As $\hat{\mathcal{A}}_{(\lambda,s)}$ and $\hat{\mathcal{A}^0}_{(\lambda,s)}$ are invertible for $s\neq 0$, it is now readily seen that the maps $\iota_\lambda$, $\lambda\in I$, extend to an injective bundle morphism $\iota:E(\hat{\mathcal{A}^0},V)\rightarrow E(\hat{\mathcal{A}},W_0\oplus W_1)$ such that we have a commutative diagram 

\begin{align}\label{diagram2}
\begin{split}
\xymatrix{&E(\hat{\mathcal{A}^0},V)\ar[r]^{\hat{\mathcal{A}^0}}\ar[d]_\iota&\Theta(V)\ar[d]^{M}\\
&E(\hat{\mathcal{A}},W_0\oplus W_1)\ar[r]^{\hat{\mathcal{A}}}&\Theta(W_0\oplus W_1)}
\end{split}
\end{align}
As $\iota$ is injective, $E_0:=\iota(E(\hat{\mathcal{A}^0},V))$ is a subbundle of $E(\hat{\mathcal{A}},W_0\oplus W_1)$. Let $E_1$ be a complementary bundle, i.e. $E_0\oplus E_1=E(\hat{\mathcal{A}},W_0\oplus W_1)$. Since $\hat{\mathcal{A}}(E_0)\subset W_0$ by the commutativity of \eqref{diagram2}, 

\[\hat{\mathcal{A}}:E_0\oplus E_1\rightarrow\Theta(W_0\oplus W_1)\]
is of the form

\begin{align}\label{BundleMatrix}
\hat{\mathcal{A}}=\begin{pmatrix}
\hat{\mathcal{A}}\mid_{E_0}&C\\
0&B
\end{pmatrix}
\end{align}
for bundle morphisms $B:E_1\rightarrow\Theta(W_1)$ and $C:E_1\rightarrow\Theta(W_0)$. Moreover, as $\ker(\hat{\mathcal{A}})=\ker(\hat{\mathcal{A}}\mid_{E_0})$ by \eqref{iota}, the morphism $B:E_1\rightarrow\Theta(W_1)$ is injective. By \eqref{dimE},

\[\dim(W_0)=\dim(V)=\dim(E(\hat{\mathcal{A}^0},V))=\dim(E_0),\quad \dim(W_0\oplus W_1)=\dim(E_0\oplus E_1),\]
which implies that $\dim(E_1)=\dim(W_1)$ and shows that $B$ is an isomorphism.\\
We now deform $C$ in \eqref{BundleMatrix} linearly to $0$ and obtain from the homotopy invariance of $K$-theory in Lemma \ref{homotopyII}

\begin{align*}
\sind(\mathcal{A})&=[E(\hat{\mathcal{A}},W_0\oplus W_1),\Theta(W_0\oplus W_1),\hat{\mathcal{A}}]=[E_0\oplus E_1,\Theta(W_0\oplus W_1),\hat{\mathcal{A}}]\\
&=[E_0\oplus E_1,\Theta(W_0\oplus W_1),\hat{\mathcal{A}}\mid_{E_0}\oplus B]=[E_0,\Theta(W_0),\hat{\mathcal{A}}\mid_{E_0}]+[E_1,\Theta(W_1),B]\\
&=[E_0,\Theta(W_0),\hat{\mathcal{A}}\mid_{E_0}].
\end{align*}
Finally, we note that $\iota:E(\hat{\mathcal{A}^0},V)\rightarrow E_0$ and $M:\Theta(V)\rightarrow\Theta(W_0)$ are bundle isomorphisms. Hence the commutativity of \eqref{diagram2} implies that 

\begin{align*}
[E_0,\Theta(W_0),\hat{\mathcal{A}}\mid_{E_0}]=[E(\hat{\mathcal{A}^0},V),\Theta(V),\hat{\mathcal{A}^0}]=\sind(\mathcal{A}^0).
\end{align*}
Now \eqref{A0} follows from Theorem C. In summary, the first three steps of our proof have shown Theorem \ref{TheoremA-weak}.

%%%%%%%%%%%%%%%%%%%%%%%%%%%%%%%%%%%%%%%%%%%%%%%%%%%%%%%%%%%%%%%%%%%%%%%%%%%%%%%%%%%%%%%%%%%%%%%%%%%%%%%%%%%%%%%%%%%%%%%%%%%%%%%%%%%%%%%%%%%%%%%%%%%%%%%%%%%%%%%%%%%%%%%%%%%%%%%%%%%%%%%%%%%%%%%%%%%%%%%%%%%%%%%%%%%%%%%%%%%%%%%%%%%%%%%%%%%%%%%%%%%%%%%%%%%%%%%%%%%%%%%%%%%%%%%%%%%%%%%%%%%%%%%%%%%%%%%%%%%%%%%%%%%%%%%%%%%%%%%%%%%%%%%%%%%%%%%%%%%%%%

\subsubsection*{Step 4: Non-admissible Paths}
The aim of this step is to lift the assumption that $\mathcal{A}$ has invertible endpoints, i.e., we want to obtain Theorem A from Theorem \ref{TheoremA-weak}.\\
As the Fredholm property is stable under small perturbations by \cite[Thm. XVII.4.2]{Gohberg}, there is $\delta>0$ such that 

\[h(\lambda,s)=\mathcal{A}_\lambda+s\delta I_{L^2(\mathbb{R},E)}\]
is Fredholm for all $(\lambda,s)\in I\times[-1,1]$. Moreover, since $0$ is either in the resolvent set or an isolated eigenvalue of finite multiplicity for selfadjoint Fredholm operators, we can assume that $h(0,s)$ and $h(1,s)$ are invertible if $s\neq 0$. Finally, we assume that $\delta$ is sufficiently small for Lemma \ref{sflperturbation} to hold.\\
The stable and unstable subspaces for the corresponding differential equations

\begin{equation*}
\left\{
\begin{aligned}
Ju'(t)+(S_\lambda(t)+s\delta I_{2n})u(t)&=0,\quad t\in\mathbb{R}\\
\lim_{t\rightarrow\pm\infty}u(t)&=0.
\end{aligned}
\right.
\end{equation*}
yield a two parameter family 

\[(E^u_{(\lambda,s)}(0),E^s_{(\lambda,s)}(0))\in\mathcal{FL}^2(E,\omega),\quad (\lambda,s)\in I\times [-1,1],\]
such that

\begin{align}\label{homotopystableunstable}
(E^u_{(\lambda,0)}(0),E^s_{(\lambda,0)}(0))&=(E^u_\lambda(0),E^s_\lambda(0)),
\end{align}
where $(E^u_\lambda(0),E^s_\lambda(0))$ are the stable and unstable spaces in Theorem A. Using the notation from Section \ref{section-Maslovmidpoint}, we have

\begin{align*}
\mu_{Mas}(E^u_{(0,\cdot)}(0),E^s_{(0,\cdot)}(0))&=\mu^{(-,0)}_{Mas}(E^u_{(0,\cdot)}(0),E^s_{(0,\cdot)}(0))+\mu^{(+,0)}_{Mas}(E^u_{(0,\cdot)}(0),E^s_{(0,\cdot)}(0)),\\
\mu_{Mas}(E^u_{(1,\cdot)}(0),E^s_{(1,\cdot)}(0))&=\mu^{(-,0)}_{Mas}(E^u_{(1,\cdot)}(0),E^s_{(1,\cdot)}(0))+\mu^{(+,0)}_{Mas}(E^u_{(1,\cdot)}(0),E^s_{(1,\cdot)}(0)).
\end{align*}
Let us now consider the two-parameter family 

\[\{(E^u_{(\lambda,s)}(0),E^s_{(\lambda,s)}(0))\}_{(\lambda,s)\in I\times I}\]
on the smaller rectangle $I\times I\subset I\times[-1,1]$. By the homotopy invariance, the Maslov index of the path obtained from restricting this family to the boundary of $I\times I$ vanishes. Hence, it follows from the concatenation property, \eqref{homotopystableunstable}, \eqref{reverseI} and \eqref{reverseII} that

\begin{align*}
\mu_{Mas}(E^u_\cdot(0),E^s_\cdot(0))&=\mu^{(+,0)}_{Mas}(E^u_{(0,\cdot)}(0),E^s_{(0,\cdot)}(0))+\mu_{Mas}(E^u_{(\cdot,1)}(0),E^s_{(\cdot,1)}(0))\\
&-\mu^{(+,0)}_{Mas}((E^u_{(1,\cdot)}(0),E^s_{(1,\cdot)}(0))).
\end{align*}   
As 

\[\mu_{Mas}(E^u_{(\cdot,1)}(0),E^s_{(\cdot,1)}(0))=\sfl(h(\cdot,1))=\sfl(\mathcal{A}^\delta)=\sfl(\mathcal{A})\]
by Theorem \ref{TheoremA-weak} and Lemma \ref{sflperturbation}, we obtain

\[\mu_{Mas}(E^u_\cdot(0),E^s_\cdot(0))=\mu^{(+,0)}_{Mas}(E^u_{(0,\cdot)}(0),E^s_{(0,\cdot)}(0))+\sfl(\mathcal{A})-\mu^{(+,0)}_{Mas}(E^u_{(1,\cdot)}(0),E^s_{(1,\cdot)}(0)).\] 
We now claim that

\begin{align}\label{final-kernel}
\mu^{(+,0)}_{Mas}(E^u_{(0,\cdot)}(0),E^s_{(0,\cdot)}(0))=\mu^{(+,0)}_{Mas}(E^u_{(1,\cdot)}(0),E^s_{(1,\cdot)}(0))=0,
\end{align}
which will prove Theorem A.\\
We consider the paths of operators $h(0,\cdot), h(1,\cdot):[-1,1]\rightarrow\mathcal{CF}^\textup{sa}(E)$ which have invertible endpoints. By Theorem \ref{TheoremA-weak}, we see that 

\[\sfl(h(0,\cdot))=\mu_{Mas}(E^u_{(0,\cdot)}(0),E^s_{(0,\cdot)}(0)),\quad \sfl(h(1,\cdot))=\mu_{Mas}(E^u_{(1,\cdot)}(0),E^s_{(1,\cdot)}(0)).\]  
On the other hand, it readily follows from the definition of the spectral flow \eqref{sfl} that

\begin{align*}
\sfl(h(0,\cdot))=\dim\ker(\mathcal{A}_0) \,\text{ and }\, \sfl(h(1,\cdot))=\dim\ker(\mathcal{A}_1).
\end{align*}
Hence

\begin{align*}
\mu_{Mas}(E^u_{(0,\cdot)}(0),E^s_{(0,\cdot)}(0))&=\dim\ker(\mathcal{A}_0)=\dim(E^u_0(0)\cap E^s_0(0))\\
&=\dim(E^u_{(0,0)}(0)\cap E^s_{(0,0)}(0)),
\end{align*}
as well as

\begin{align*}
\mu_{Mas}(E^u_{(1,\cdot)}(0),E^s_{(1,\cdot)}(0))&=\dim\ker(\mathcal{A}_1)=\dim(E^u_1(0)\cap E^s_1(0))\\
&=\dim(E^u_{(1,0)}(0)\cap E^s_{(1,0)}(0)).
\end{align*}
Since

\begin{align*}
\dim(E^u_{(0,s)}(0)\cap E^s_{(0,s)}(0))&=\dim\ker(h(0,s))=0,\\
\dim(E^u_{(1,s)}(0)\cap E^s_{(1,s)}(0))&=\dim\ker(h(1,s))=0
\end{align*}
for $s\neq 0$, \eqref{final-kernel} follows from Lemma \ref{MaslovMidpoint}.

%%%%%%%%%%%%%%%%%%%%%%%%%%%%%%%%%%%%%%%%%%%%%%%%%%%%%%%%%%%%%%%%%%%%%%%%%%%%%%%%%%%%%%%%%%%%%%%%%%%%%%%%%%%%%%%%%%%%%%%%%%%%%%%%%%%%%%%%%%%%%%%%%%%%%%%%%%%%%%%%%%%%%%%%%%%%%%%%%%%%%%%%%%%%%%%%%%%%%%%%%%%%%%%%%%%%%%%%%%%%%%%%%%%%%%%%%%%%%%%%%%%%%%%%%%%%%%%%%%%%%%%%%%%%%%%%%%%%%%%%%%%%%%%%%%%%%%%%%%%%%%%%%%%%%%%%%%%%%%%%%%%%%%%%%%%%%%%%%%%%%%%%%%%%%%%%%%%%%%%%%%%%%%%%%%%%%%%%%%%%%%%%%%%%%%%%%%%%%%%%%%%%%%%%%%%%%%%%%%%%%%%%%%%%%%%%%%%%%%%%%%%%%%%%%%%%%%%%%%%%%%%%%%%%%%%%%%%%%%%%%%%%%%%%%%%%%%%%%%%%%%%%%%%%%%%%%%%%%%%%%%%%%%%%%%%%%%%%%%%%%%%%%%%%%%%%%%%%%%%%%%%%%%%%%%%%%%%%%%%%%%%%%%%%%%%%%%%%%%%%%%%%%%%%%%%%%%%%%%%%%%%%%%%%%%%%%%%%%%%%%%%%%%%%%%%%%%%%%%%%%%%%%%%%%%%%%%%%%%%%%%%%%%%%%%%%%%%%%%%%%%%%%%%%%%%%%%%%%%%%%%%%%%%%%%%%%%%%%%%%%%%%%%%%%%%%%%%%%%%%%%%%%%%%%%%%%%%%%%%%%%%%%%%%%%%

\appendix

\section*{Appendix}

%%%%%%%%%%%%%%%%%%%%%%%%%%%%%%%%%%%%%%%%%%%%%%%%%%%%%%%%%%%%%%%%%%%%%%%%%%%%%%%%%%%%%%%%%%%%%%%%%%%%%%%%%%%%%%%%%%

\section{Construction of the Maslov Index and a Simple Lemma}\label{app-Maslov}

\subsection{Construction of the Maslov Index}
The aim of this section is to briefly recap the construction of the Maslov index from \cite{Furutani}. Let us point out that an alternative construction of the Maslov index in this setting can be found in \cite{Hermann}.\\
Let $E$ be a real separable Hilbert space with scalar product $\langle\cdot,\cdot\rangle$. Let $\omega:E\times E\rightarrow\mathbb{R}$ be a symplectic form on $E$ such that $\omega(x,y)=\langle Jx,y\rangle$ for a bounded operator $J:E\rightarrow E$ such that $J^{2}=-I_E$ and $J^T=-J$. We can regard $E$ as complex Hilbert space through the almost complex structure $J$, where the complex inner product is given by $\langle\cdot,\cdot\rangle_J=\langle\cdot,\cdot\rangle-i\omega(\cdot,\cdot)$. In what follows we denote by $\mathcal{U}(E_J)$ the unitary operators on $E$, and set

\[\mathcal{U}_{\mathcal{F}}(E_J)=\{U\in \mathcal{U}(E_J):\, U+I_E\,\,\,\text{Fredholm}\}.\]  
The first important step in the construction is to show that there is a \textit{winding number} $w$ for paths in $\mathcal{U}_{\mathcal{F}}(E_J)$ which is defined as follows (see \cite[\S 2.1]{Furutani}). If $d:I\rightarrow \mathcal{U}_{\mathcal{F}}(E_J)$ is a path, then there is a partition $0=\lambda_0<\lambda_1<\cdots<\lambda_{m-1}<\lambda_m=1$ of $I$ and positive numbers $0<\varepsilon_j<\pi$, $j=1,\ldots,m$, such that for $\lambda_{j-1}\leq \lambda\leq \lambda_j$

\begin{align}\label{appcond}
e^{i(\pi\pm\varepsilon_j)}\notin\sigma(d(\lambda))
\end{align}
and

\begin{align}\label{appsumI}
\sum_{|\theta|\leq\varepsilon_j}{\dim\ker(d(\lambda)-e^{i(\pi+\theta)})}<\infty.
\end{align}
Now the \textit{winding number} of $d$ is defined by

\begin{align}\label{winding}
w(d)=\sum^m_{j=1}{(k(\lambda_j,\varepsilon_j)-k(\lambda_{j-1},\varepsilon_j))},
\end{align}
where

\begin{align}\label{appsumII}
k(\lambda,\varepsilon_j)=\sum_{0\leq\theta\leq\varepsilon_j}\dim\ker(d(\lambda)-e^{i(\pi+\theta)}),\quad \lambda_{j-1}\leq \lambda\leq \lambda_j.
\end{align}
It is shown in \cite[Prop. 2.3]{Furutani} that $w(d)$ does only depend on the path $d$ and neither on the partition $0=\lambda_0<\lambda_1<\cdots<\lambda_{m-1}<\lambda_m=1$ nor on the numbers $\varepsilon_j$ in \eqref{appcond} and \eqref{appsumI}. Let us point out that there are different limits for the sums in \eqref{appsumI} and \eqref{appsumII}, and the latter is the number of all eigenvalues of $d(\lambda)$ between $-1$ and $e^{i(\pi+\varepsilon_j)}$. Hence, roughly speaking, $w(d)$ is the number of eigenvalues of $d(0)$ crossing $-1$ whilst the parameter $\lambda$ travels along the unit interval.\\
Let now $W\in\Lambda(E,\omega)$ be a fixed Lagrangian subspace. The \textit{Souriau map} is defined by

\[S_W(\tilde{W})=-(I_E-2P_{\tilde{W}})(I_E-2P_W),\]
where $P_W$ and $P_{\tilde{W}}$ are the orthogonal projections onto $W$ and $\tilde{W}$, respectively. Of course, $S_W$ is defined for any closed subspace $\tilde{W}$ of $E$, but it is shown in \cite[\S 1.5]{Furutani} that $S_W$ maps $\Lambda(E,\omega)$ into $\mathcal{U}(E_J)$. Moreover, $S_W(\mathcal{FL}_W(E,\omega))\subset \mathcal{U}_{\mathcal{F}}(E_J)$ and, for any $\tilde{W}\in\mathcal{FL}_W(E,\omega)$,

\begin{align}\label{dimensionequalitySauriou}
\dim_{\mathbb{R}}(\tilde{W}\cap W)=\dim_{\mathbb{C}}\ker(S_W(\tilde{W})+I_E).
\end{align}
In other words, the dimension of the intersection $\tilde{W}\cap W$ is the multiplicity of $-1$ as an eigenvalue of $S_W(\tilde{W})\in\mathcal{U}_{\mathcal{F}}(E_J)$.\\
Finally, the \textit{Maslov index} of a path $\Lambda:I\rightarrow\mathcal{FL}_W(E,\omega)$ is defined as the composition

\[\mu_{Mas}(\Lambda,W)=w(S_W(\Lambda(\cdot)))\in\mathbb{Z}.\] 
Note that it follows from the definition of the winding number that the Maslov index has indeed the heuristic interpretation that we mentioned in Section \ref{subsection-Maslov}, i.e. it is the net number of non-trivial intersections of $\Lambda(\lambda)$ with $W$ whilst $\lambda$ travels along the unit interval. Finally, let us note that if $-\Lambda:I\rightarrow\mathcal{FL}_W(E,\omega)$ denotes the reverse path $-\Lambda(\lambda)=\Lambda(1-\lambda)$, $\lambda\in I$, then

\begin{align}\label{reverseI}
\mu_{Mas}(-\Lambda,W)=-\mu_{Mas}(\Lambda,W).
\end{align}
This is an immediate consequence of \eqref{winding} and the injectivity of the Souriau map $S_W$.

%%%%%%%%%%%%%%%%%%%%%%%%%%%%%%%%%%%%%%%%%%%%%%%%%%%%%%%%%%%%%%%%%%%%%%%%%%%%%%%%%%%%%%%%%%%%%%%%%%%%%%%%%%%%%%%%%%

\subsection{A Simple Lemma}\label{section-Maslovmidpoint}
For a path $\Lambda:I\rightarrow\mathcal{FL}_W(E,\omega)$ and $\lambda_0\in I$, we denote by 

\[\mu_{Mas}^{(+,\lambda_0)}(\Lambda,W)\quad\text{and}\quad \mu_{Mas}^{(-,\lambda_0)}(\Lambda,W)\] 
the Maslov index of the restriction of $\Lambda$ to $[\lambda_0,1]$ and $[0,\lambda_0]$, respectively. Note that, by the concatenation property, we have

\[\mu_{Mas}(\Lambda,W)=\mu_{Mas}^{(+,\lambda_0)}(\Lambda,W)+\mu_{Mas}^{(-,\lambda_0)}(\Lambda,W),\]
and moreover, it is readily seen from \eqref{reverseI} that

\begin{align}\label{reverseII}
\begin{split}
\mu_{Mas}^{(+,\lambda_0)}(-\Lambda,W)&=-\mu_{Mas}^{(-,\lambda_0)}(\Lambda,W),\\
\mu_{Mas}^{(-,\lambda_0)}(-\Lambda,W)&=-\mu_{Mas}^{(+,\lambda_0)}(\Lambda,W).
\end{split}
\end{align}
It is an interesting question to determine the contributions of $\mu_{Mas}^{(\pm,\lambda_0)}(\Lambda,W)$ to $\mu_{Mas}(\Lambda,W)$. We will not deal with this question in its full generality, but note the following special case that we need in the proof of Theorem A.

\begin{lemma}\label{MaslovMidpoint}
Let $\Lambda:I\rightarrow\mathcal{FL}_W(E,\omega)$ be a path and let $\lambda_0\in I$ be the only parameter value where $\Lambda$ and $W$ intersect non-trivially.

\begin{itemize}
\item If

\[\mu_{Mas}(\Lambda,W)=\dim(\Lambda(\lambda_0)\cap W),\]
then

\[ \mu^{(-,\lambda_0)}_{Mas}(\Lambda,W)=\dim(\Lambda(\lambda_0)\cap W),\quad \mu^{(+,\lambda_0)}_{Mas}(\Lambda,W)=0.\]
\item If

\[\mu_{Mas}(\Lambda,W)=-\dim(\Lambda(\lambda_0)\cap W),\]
then

\[\mu^{(-,\lambda_0)}_{Mas}(\Lambda,W)=0,\quad \mu^{(+,\lambda_0)}_{Mas}(\Lambda,W)=-\dim(\Lambda(\lambda_0)\cap W).\]
\end{itemize} 
\end{lemma}

\begin{proof}
We only need to prove the first assertion as this implies the second one by \eqref{reverseI} and \eqref{reverseII}. We set $d(\lambda)=S_W(\Lambda(\lambda))$, $\lambda\in I$. By the concatenation property of the Maslov index, we can assume without loss of generality that there is $0<\varepsilon<\pi$ such that $e^{i(\pi\pm\varepsilon)}\notin\sigma(d(\lambda))$ and

\begin{align}\label{simplelemmaker}
\dim\ker(d(\lambda_0)+I_E)=\sum_{|\theta|\leq\varepsilon}{\dim\ker(d(\lambda)-e^{i(\pi+\theta)})}<\infty
\end{align}
for all $\lambda\in I$. Therefore

\begin{align}\label{simplelemmamaslovdefinition}
\begin{split}
\dim(\Lambda(\lambda_0)\cap W)&=\mu_{Mas}(\Lambda,W)=k(1,\varepsilon)-k(0,\varepsilon)\\
&=\sum_{0\leq\theta\leq\varepsilon}\dim\ker(d(1)-e^{i(\pi+\theta)})-\sum_{0\leq\theta\leq\varepsilon}\dim\ker(d(0)-e^{i(\pi+\theta)}).
\end{split}
\end{align}
As $\dim(\Lambda(\lambda_0)\cap W)=\dim\ker(d(\lambda_0)+I_E)$ by \eqref{dimensionequalitySauriou}, we see from \eqref{simplelemmaker} that $\dim(\Lambda(\lambda_0)\cap W)$ is an upper bound for  

\[\sum_{0\leq\theta\leq\varepsilon}\dim\ker(d(1)-e^{i(\pi+\theta)})\quad\text{and}\quad \sum_{0\leq\theta\leq\varepsilon}\dim\ker(d(0)-e^{i(\pi+\theta)}).\]
Hence \eqref{simplelemmamaslovdefinition} implies that

\[\sum_{0\leq\theta\leq\varepsilon}\dim\ker(d(1)-e^{i(\pi+\theta)})=\dim(\Lambda(\lambda_0)\cap W),\quad \sum_{0\leq\theta\leq\varepsilon}\dim\ker(d(0)-e^{i(\pi+\theta)})=0.\]
As by \eqref{dimensionequalitySauriou} and \eqref{simplelemmaker},

\[\sum_{0\leq\theta\leq\varepsilon}\dim\ker(d(\lambda_0)-e^{i(\pi+\theta)})=\dim\ker(d(\lambda_0)+I_E)=\dim(\Lambda(\lambda_0)\cap W),\]
we obtain

\begin{align*}
\mu^{(+,\lambda_0)}(\Lambda,W)&=\sum_{0\leq\theta\leq\varepsilon}\dim\ker(d(1)-e^{i(\pi+\theta)})-\sum_{0\leq\theta\leq\varepsilon}\dim\ker(d(\lambda_0)-e^{i(\pi+\theta)})=0
\end{align*}
as well as

\begin{align*}
\mu^{(-,\lambda_0)}(\Lambda,W)&=\sum_{0\leq\theta\leq\varepsilon}\dim\ker(d(\lambda_0)-e^{i(\pi+\theta)})-\sum_{0\leq\theta\leq\varepsilon}\dim\ker(d(0)-e^{i(\pi+\theta)})\\
&=\dim(\Lambda(\lambda_0)\cap W).
\end{align*}
\end{proof}
\noindent
Finally, let us note that a corresponding statement holds for the relative Maslov index as well, which follows straight from its definition.

%%%%%%%%%%%%%%%%%%%%%%%%%%%%%%%%%%%%%%%%%%%%%%%%%%%%%%%%%%%%%%%%%%%%%%%%%%%%%%%%%%%%%%%%%%%%%%%%%%%%%%%%%%%%%%%%%%

\section{$K$-theory}\label{app-K}
In this appendix, we give a brief overview of topological $K$-theory for locally compact spaces, where we follow \cite{Lawson} and \cite{Segal}.\\
Let $X$ be a locally compact topological space, and $\{E_0,E_1,a\}$ a triple consisting of two complex vector bundles $E_0$ and $E_1$ over $X$ and a bundle morphism $a:E_0\rightarrow E_1$ between them. The \textit{support} $\supp(\xi)$ of $\xi=\{E_0,E_1,a\}$ is the subset of those points $\lambda\in X$ for which $a_\lambda:E_{0,\lambda}\rightarrow E_{1,\lambda}$ is not an isomorphism. Note that $\supp(\xi)\subset X$ is closed. We call $\xi$ \textit{trivial} if $\supp(\xi)=\emptyset$, i.e., if $a$ is an isomorphism. Two triples $\{E^0_0,E^0_1,a_0\}$ and $\{E^1_0,E^1_1,a_1\}$ are said to be \textit{isomorphic} if there are bundle isomorphisms $\varphi_0:E^0_0\rightarrow E^1_0$ and $\varphi_1:E^0_1\rightarrow E^1_1$ such that

\begin{align*}
\xymatrix{E^1_0\ar[r]^{a_1}&E^1_1\\
E^0_0\ar[u]^{\varphi_0}\ar[r]^{a_0}&E^0_1\ar[u]_{\varphi_1}
}
\end{align*}  
is a commutative diagram.\\
Now let us assume that $Y\subset X$ is a closed subspace. We denote by $L(X,Y)$ the set of isomorphism classes of triples $\xi=\{E_0,E_1,a\}$ on $X$ such that $\supp(\xi)$ is a compact subset of $X\setminus Y$. Note that the direct sum of vector bundles makes $L(X,Y)$ a semi-group and

\[\supp(\xi^0\oplus \xi^1)=\supp(\xi^0)\cup\supp(\xi^1),\quad \xi^0,\xi^1\in L(X,Y).\]
We say that $\xi^0=\{E^0_0,E^0_1,a_0\}$ and $\xi^1=\{E^1_0,E^1_1,a_1\}$ are \textit{homotopic}, and write $\xi^0\simeq \xi^1$, if there is an element of $L(X\times[0,1],Y\times[0,1])$ such that its restrictions to $X\times\{0\}$ and $X\times\{1\}$ are isomorphic to $\xi^0$ and $\xi^1$, respectively. Finally, we obtain an equivalence relation $\sim$ on $L(X,Y)$ by setting $\xi^0\sim \xi^1$ if there are $\eta^0,\eta^1\in L(X,Y)$ which are trivial and such that 

\[\xi^0\oplus \eta^0\simeq \xi^1\oplus \eta^1.\]
Now the \textit{K-theory} $K(X,Y)$ of the pair $(X,Y)$ is the set of equivalence classes of $L(X,Y)$. In what follows, $[E_0,E_1,a]$ denotes the class of $\{E_0,E_1,a\}\in L(X,Y)$ in $K(X,Y)$. It is readily seen that 

\[[E^0_0,E^0_1,a_0]+[E^1_0,E^1_1,a_1]:=[E^0_0\oplus E^1_0, E^0_1\oplus E^1_1, a_0\oplus a_1]\in K(X,Y)\]
makes $K(X,Y)$ an abelian group, where the neutral element is given by the equivalence class of any trivial element in $L(X,Y)$.\\
For proper maps $f:(X,Y)\rightarrow (X',Y')$ of topological pairs, we obtain a group homomorphisms

\[f^\ast:K(X',Y')\rightarrow K(X,Y),\quad f^\ast[E_0,E_1,a]=[f^\ast E_0,f^\ast E_1,f^\ast a],\]
where $f^\ast E_0$, $f^\ast E_1$ are the pullback bundles and $(f^\ast a)_\lambda=a_{f(\lambda)}$, $\lambda\in X$. If $g:(X,Y)\rightarrow (X',Y')$ and $f:(X',Y')\rightarrow (\tilde{X},\tilde{Y})$ are proper maps, then

\[(f\circ g)^\ast=g^\ast\circ f^\ast:K(\tilde{X},\tilde{Y})\rightarrow K(X,Y).\]
Moreover, $f^\ast=g^\ast:K(X',Y')\rightarrow K(X,Y)$ if $f\simeq g:(X,Y)\rightarrow(X',Y')$ are homotopic by a proper homotopy. Hence, $K$ is a contravariant functor from the category of pairs of locally compact spaces and closed subspaces to the category of abelian groups. A further homotopy invariance property, which is often useful in computations, is as follows:

\begin{lemma}\label{homotopyII}
Let $E_0$ and $E_1$ be vector bundles over $X$ and $a:[0,1]\rightarrow \hom(E_0,E_1)$ a path of bundle morphisms. If

\[ \supp \{E_0,E_1,a_t\}\subset K\subset X\setminus Y,\quad t\in [0,1],\]
for some compact set $K\subset X$, then

\[[E_0,E_1,a_0]=[E_0,E_1,a_1]\in K(X,Y).\] 
\end{lemma}
\noindent
The following lemma is usually called the \textit{logarithmic property} of $K$.

\begin{lemma}\label{logarithmic}
For $[E_0,E_1,a_0]$, $[E_1,E_2,a_1]\in K(X,Y)$,

\[[E_0,E_1,a_0]+[E_1,E_2,a_1]=[E_0,E_2,a_1\circ a_0]\in K(X,Y).\]
\end{lemma}
\noindent
The groups

\[K^{-1}(X,Y)=K(X\times\mathbb{R},Y\times\mathbb{R})\]
are called the \textit{odd $K$-theory groups} and, as above, any proper map $f:(X,Y)\rightarrow(X',Y')$ induces a homomorphism

\[f^\ast:K^{-1}(X',Y')\rightarrow K^{-1}(X,Y).\]
Finally, we want to recall the well known isomorphism $c_1:K^{-1}(I,\partial I)\rightarrow\mathbb{Z}$ coming from the first Chern number, where $I\subset\mathbb{R}$ is a compact interval. If 

\[[E_0,E_1,a]\in K^{-1}(I,\partial I)=K(I\times\mathbb{R},\partial I\times\mathbb{R}),\]
then there are global trivialisations $\psi: E_0\rightarrow(I\times\mathbb{R})\times\mathbb{C}^n$ and $\varphi:E_1\rightarrow(I\times\mathbb{R})\times\mathbb{C}^n$, as any bundle over the contractible space $I\times\mathbb{R}$ is trivial. Now, the first Chern number can be computed by

\begin{align}\label{Chern}
c_1([E_0,E_1,a])=w(\det(\varphi\circ a\circ\psi^{-1})\circ\gamma,0)\in\mathbb{Z},
\end{align}
where $\gamma:S^1\rightarrow I\times\mathbb{R}$ is any positively oriented simple curve surrounding the support of $\{E_0,E_1,a\}$. Here, $w(\cdot,0)$ denotes the winding number for closed curves in $\mathbb{C}\setminus\{0\}$ with respect to $0$.

\vspace{1cm}
Nils Waterstraat\\
School of Mathematics,\\
Statistics \& Actuarial Science\\
University of Kent\\
Canterbury\\
Kent CT2 7NF\\
UNITED KINGDOM\\
E-mail: n.waterstraat@kent.ac.uk

\end{document}